\pdfoutput=1
\documentclass[twoside]{article}
\usepackage[letterpaper, left=3cm, right=3cm, top=4cm, bottom=2.5cm]{geometry}
\usepackage{graphicx}
\usepackage{amsmath,amssymb,amsthm}
\usepackage{authblk}
\usepackage[utf8]{inputenc}
\usepackage{microtype}
\usepackage{mathrsfs} 
\usepackage{enumitem}
\usepackage{calc}
\usepackage{esint}
\usepackage{fancyhdr}
\usepackage{xcolor}
\usepackage[labelfont = bf]{caption}
\usepackage{doi}
\usepackage{subfigure}
\usepackage[numbers]{natbib}
\usepackage{stmaryrd}
\usepackage{mathtools}
\usepackage{booktabs}
\usepackage{colortbl}
\usepackage{tabulary}
\usepackage{diagbox}
\usepackage{caption}
\usepackage{float}

\usepackage{physics}
\usepackage{amsmath}
\usepackage{tikz}
\usepackage{mathdots}
\usepackage{yhmath}
\usepackage{cancel}
\usepackage{color}
\usepackage{siunitx}
\usepackage{array}
\usepackage{multirow}
\usepackage{amssymb}
\usepackage{gensymb}
\usepackage{tabularx}
\usepackage{extarrows}
\usepackage{booktabs}
\usetikzlibrary{fadings}
\usetikzlibrary{patterns}
\usetikzlibrary{shadows.blur}
\usetikzlibrary{shapes}

\newtheorem{theorem}{Theorem}[section]
\numberwithin{equation}{section}
\newtheorem{proposition}[theorem]{Proposition}
\newtheorem{definition}[theorem]{Definition}
\newtheorem{corollary}[theorem]{Corollary}
\newtheorem{remark}[theorem]{Remark}
\newtheorem{lemma}[theorem]{Lemma}

\newtheorem{algorithm}[theorem]{Scheme}
\newtheorem{assumption}[theorem]{Assumption}

\providecommand{\Vo}{{\mathaccent23 V}}
\providecommand{\Qo}{{\mathaccent23 Q}}
 
\providecommand{\D}{{\scriptscriptstyle \mathbf{D}_x}}
\providecommand{\plus}{{\scriptscriptstyle +}}
\providecommand{\bXp}{{\mathbfcal{U}^{q,p}_{\D}}}
\providecommand{\bXplus}{{\mathbfcal{U}^{q,p}_{\plus}}}
\providecommand{\bXminus}{{\mathbfcal{U}^{q,p}_{\minus}}}
\providecommand{\bXpn}{{\mathbfcal{U}^{q_n,p_n}_{\D}}}
\providecommand{\bVp}{{\mathbfcal{V}^{q,p}_{\D}}}

\providecommand{\minus}{{\scriptscriptstyle -}}
\DeclareMathAlphabet\mathbfcal{OMS}{cmsy}{b}{n}

\usepackage{titlesec}
\titleformat{\section}{\normalfont\scshape\centering}{\thesection.}{0.5em}{}
\titleformat*{\subsection}{\itshape}
\titleformat*{\subsubsection}{\itshape}

\providecommand{\keywords}[1]
{
	{\small\textit{Keywords:~~} #1}
}

\providecommand{\MSC}[1]
{
	{\small\textit{AMS MSC (2020): ~~} #1}
}

\usepackage{hyperref}
\hypersetup{
	colorlinks=true,
	linkcolor=blue,
	citecolor = blue,
	filecolor=magenta, 
	urlcolor=cyan,
}

\AtBeginDocument{%
	\def\MR#1{}
}

\begin{document}
	\setlength{\abovedisplayskip}{5.5pt}
	\setlength{\belowdisplayskip}{5.5pt}
	\setlength{\abovedisplayshortskip}{5.5pt}
	\setlength{\belowdisplayshortskip}{5.5pt}

	\title{Convergence analysis for a fully-discrete finite element approximation of the unsteady $p(\cdot,\cdot)$-Navier--Stokes equations}
	\author[1]{Luigi C. Berselli\thanks{Email: \texttt{luigi.carlo.berselli@unipi.it}}\thanks{funded by  INdAM GNAMPA;  MIUR 
			within PRIN20204NT8W;
			 MIUR Excellence, Department of Mathematics, University of Pisa, CUP I57G22000700001.}}
	\author[2]{Alex Kaltenbach\thanks{Email: \texttt{kaltenbach@math.tu-berlin.de}}\thanks{funded by the Deutsche Forschungsgemeinschaft (DFG, German Research
		Foundation) - 525389262.}}
\date{\today}
\affil[1]{\small{Department of Applied Mathematics, University of Pisa, Largo Bruno Pontecorvo 5, 56127 Pisa}}
\affil[2]{\small{Institute of Mathematics, Technical University of Berlin, Straße des 17. Juni 136, 10623 Berlin}}
\maketitle

	\pagestyle{fancy}
	\fancyhf{}
	\fancyheadoffset{0cm}
	\addtolength{\headheight}{-0.25cm}
	\renewcommand{\headrulewidth}{0pt} 
	\renewcommand{\footrulewidth}{0pt}
	\fancyhead[CO]{\textsc{Convergence of a FE approximation of the $p(\cdot,\cdot)$-Navier--Stokes equations}}
	\fancyhead[CE]{\textsc{L. C. Berselli and A. Kaltenbach}}
	\fancyhead[R]{\thepage}
	\fancyfoot[R]{}
	\thispagestyle{empty}
	
	\begin{abstract}
	       In the present paper, we establish the well-posedness, stability, and (weak) convergence of a fully-discrete approximation of the unsteady $p(\cdot,\cdot)$-Navier--Stokes equations employing an implicit Euler step in time and a discretely inf-sup-stable finite element approximation~in~space.~Moreover, numerical experiments are carried out that supplement the theoretical findings.
	\end{abstract}

	
	\MSC{35Q35; 65M12; 65N30; 76A05.}
	
	\section{Introduction}\label{sec:intro}

    \hspace{5mm}In this paper, we examine a fully-discrete Rothe--Galerkin approximation for unsteady systems of  \textit{$p(\cdot,\cdot)$-Navier--Stokes type}, \textit{i.e.},\enlargethispage{1mm}
\begin{equation}
	\label{eq:ptxNavierStokes}
	\begin{aligned}
		\partial_t\boldsymbol{v}-\mathrm{div}_x(\mathbf{S}(\cdot,\cdot,\mathbf{D}_x\boldsymbol{v}))+\mathrm{div}_x(\boldsymbol{v}\otimes\boldsymbol{v})+\nabla_{\! x}    \pi&=\boldsymbol{f} +\mathrm{div}_x\boldsymbol{F}&&\quad\text{ in }Q_T\,,\\
		\mathrm{div}_x\boldsymbol{v}&=0 &&\quad\text{ in }Q_T\,,
		\\
		\boldsymbol{v} &= \mathbf{0} &&\quad\text{ on } \Gamma_T\,,\\
  \boldsymbol{v}(0) &= \mathbf{v}_0 &&\quad\text{ in } \Omega\,,
	\end{aligned}
\end{equation}
using an implicit Euler step in time and a discretely inf-sup-stable finite element approximation~in~space. 
More precisely,
for a given vector field $\boldsymbol{f}\colon Q_T\to \mathbb{R}^d$ and a given tensor field $\boldsymbol{F}\colon Q_T\to \mathbb{R}^{d\times d}_{\mathrm{sym}}$, jointly describing external
body forces, an incompressibility constraint \eqref{eq:ptxNavierStokes}$_2$, a no-slip
boundary condition~\eqref{eq:ptxNavierStokes}$_3$, and an initial velocity vector field $\mathbf{v}_0\colon \Omega\to \mathbb{R}^d$ at initial time $t=0$,  the system \eqref{eq:ptxNavierStokes} seeks~for~a~\textit{velocity vector field} $\boldsymbol{v}\coloneqq (v_1,\ldots,v_d)^\top\colon \overline{Q_T}\to
	\mathbb{R}^d$ and a scalar \textit{kinematic~\mbox{pressure}}~${\pi\colon Q_T\to \mathbb{R}}$~\mbox{solving}~\eqref{eq:ptxNavierStokes}.
Here, $\Omega\subseteq \mathbb{R}^d$, $d\ge 2$, is a bounded simplicial Lipschitz domain, $I\coloneqq (0,T)$, $0<T<\infty$, a finite time interval, $Q_T\coloneqq I\times\Omega$ the corresponding time-space~cylinder~and~$\Gamma_T\coloneqq I\times\partial\Omega$~the~parabolic~boundary.~The~\textit{extra-stress tensor} $\mathbf{S}(\cdot,\cdot,\mathbf{D}_x\boldsymbol{v})\colon  Q_T\to \mathbb{R}^{d\times d}_{\textup{sym}}$\footnote{Here, $\mathbb{R}^{d\times d}_{\textup{sym}}\coloneqq \{\mathbf{A}\in \mathbb{R}^{d\times d}\mid \mathbf{A}=\mathbf{A}^\top\}$ and $\mathbf{A}^{\textup{sym}}\coloneqq[\mathbf{A}]^{\textup{sym}}\coloneqq \frac{1}{2}(\mathbf{A}+\mathbf{A}^\top)\in \mathbb{R}^{d\times d}_{\textup{sym}}$ for all $\mathbf{A}\in\mathbb{R}^{d\times d}$.}  depends on the \textit{strain-rate} \mbox{tensor} $\smash{\mathbf{D}_x\boldsymbol{v} \coloneqq  [\nabla_{\! x}  \boldsymbol{v}]^{\textup{sym}}
	\colon  Q_T\to  \mathbb{R}^{d\times d}_{\textup{sym}}}$, \textit{i.e.}, the symmetric part of the velocity gradient  $\nabla_{\! x}  \boldsymbol{v}\hspace*{-0.1em}=\hspace*{-0.1em}(\partial_{x_j} v_i)_{i,j=1,\ldots,d}
\colon \hspace*{-0.1em}Q_T\hspace*{-0.1em}\to\hspace*{-0.1em} \mathbb{R}^{d\times d}$. The \textit{convective~term} ${\mathrm{div}_x(\boldsymbol{v}\otimes\hspace*{-0.1em}\boldsymbol{v})\colon Q_T\hspace*{-0.1em}\to\hspace*{-0.1em} \mathbb{R}^d}$ is defined by $(\mathrm{div}_x(\boldsymbol{v}\otimes\boldsymbol{v}))_i\coloneqq \sum_{j=1}^d{\partial_{x_j}(v_i v_j)}$ for all $i=1,\ldots,d$.\pagebreak

Throughout the paper, we  assume that the extra-stress tensor $\mathbf{S}$ has a \textit{weak\footnote{The prefix \textit{weak} refers to the fact that the extra-stress tensor does not need to possess any higher~regularity~properties.} $(p(\cdot,\cdot),\delta)$-structure}, \textit{i.e.},
for some $\delta\ge 0$ and  some \textit{power-law index} $p\colon Q_T\to [0,+\infty)$, which is at least (Lebesgue)~measurable~with\vspace*{-0.5mm}
\begin{align}\label{eq:pm}
	1<p^-\coloneqq \underset{(t,x)^\top\in Q_T}{\textrm{ess\,inf}}{p(t,x)}\leq p^+\coloneqq \underset{(t,x)^\top\in Q_T}{\textrm{ess\,sup}}{p(t,x)}<+\infty\,,\\[-7mm]\notag
\end{align}
the following properties are satisfied:
	\begin{itemize}[noitemsep,topsep=2pt,leftmargin=!,labelwidth=\widthof{\itshape (S.3)},font=\itshape]
		\item[(S.1)]\hypertarget{S.1}{} $\textbf{S}\colon Q_T\times \mathbb{R}^{d\times d}_{\textup{sym}}\to 
		\mathbb{R}^{d\times d}_{\textup{sym}}$ is a Carath\'eodory mapping\footnote{$\textbf{S}(\cdot,\cdot,\mathbf{A})\colon Q_T\to \mathbb{R}^{d\times d}_{\textup{sym}}$ is (Lebesgue) measurable for all $\mathbf{A}\in \mathbb{R}^{d\times d}_{\textup{sym}}$ and $\textbf{S}(t,x,\cdot)\colon \mathbb{R}^{d\times d}_{\textup{sym}}\to \mathbb{R}^{d\times d}_{\textup{sym}}$ is continuous for a.e.\ $(t,x)^\top\in Q_T$.};
		\item[(S.2)]\hypertarget{S.2}{} $\vert\mathbf{S}(t,x,\mathbf{A})\vert \leq \alpha\,(\delta+\vert \mathbf{A}\vert)^{p(t,x)-2}\vert\mathbf{A}\vert+\beta(t,x)$ 
		for all $\mathbf{A}\in \mathbb{R}^{d\times d}_{\textup{sym}}$ and a.e.\  $(t,x)^\top\in Q_T$,
		\newline where $\alpha\ge 1$ and $\beta\in L^{p'(\cdot,\cdot)}(Q_T;\mathbb{R}_{\ge 0})$;
		\item[(S.3)]\hypertarget{S.3}{} $\mathbf{S}(t,x, \mathbf{A}):\mathbf{A}\ge 
		c_0\,(\delta+\vert \mathbf{A}\vert)^{p(t,x)-2}\vert\mathbf{A}\vert^2-c_1(t,x)$ 
		for all ${\mathbf{A}\in \mathbb{R}^{d\times d}_{\textup{sym}}}$~and~a.e.~${(t,x)^\top\in Q_T}$,\newline
		where $c_0>0$ and $c_1\in L^1(Q_T;\mathbb{R}_{\ge 0})	$;
		\item[(S.4)]\hypertarget{S.4}{} $(\mathbf{S}(t,x,\mathbf{A})-\mathbf{S}(t,x,\mathbf{B})):
		(\mathbf{A}-\mathbf{B})\ge 0$ for all 
		$\mathbf{A},\mathbf{B}\in \mathbb{R}^{d\times d}_{\textup{sym}}$ and a.e.\  $(t,x)^\top\in Q_T$.
	\end{itemize}
 \hspace*{5mm}A prototypical example for $\textbf{S}\colon Q_T\times \mathbb{R}^{d\times d}_{\textup{sym}}\to 
 \mathbb{R}^{d\times d}_{\textup{sym}}$  satisfying (\hyperlink{S.1}{S.1})--(\hyperlink{S.1}{S.1}),  for  a.e.\   $(t,x)^\top\in Q_T$   and every  $\mathbf{A}\in \smash{\mathbb{R}^{d\times d}_{\textup{sym}}}$, is defined by\vspace*{-1mm}
\begin{align}\label{eq:example-stress}
\smash{\mathbf{S}(t,x,\mathbf{A})\coloneqq \mu_0\, (\delta+\vert \mathbf{A}\vert)^{p(t,x)-2}\mathbf{A}\,,}
\end{align}
where $\mu_0\hspace*{-0.15em}>\hspace*{-0.15em}0$, $\delta\hspace*{-0.15em}\ge\hspace*{-0.15em} 0$, and the power-law index $p\colon \hspace*{-0.15em}Q_T\hspace*{-0.15em}\to\hspace*{-0.15em} [0,+\infty)$ is at least (Lebesgue)~\mbox{measurable}~with~\eqref{eq:pm}. 

The unsteady $p(\cdot,\cdot)$-Navier--Stokes equations \eqref{eq:ptxNavierStokes} is a prototypical example of a non-linear system with variable growth conditions. It  appears naturally in physical models for \textit{smart fluids}, \textit{e.g.}, electro-rheological fluids (\textit{cf}.\ \cite{RR1,rubo}), micro-polar electro-rheological fluids (\textit{cf}.\ \cite{win-r,eringen-book}), magneto-rheological fluids (\textit{cf}.\ \cite{magneto}), chemically reacting fluids (\textit{cf.}~\cite{LKM78,HMPR10,KPS18}), and thermo-rheological~fluids~(\textit{cf.}~\cite{Z97,AR06}).  \linebreak In all these models, the power-law index $p(\cdot,\cdot)$ is a function depending on certain physical~\mbox{quantities}, \textit{e.g.}, an electric field, a magnetic field, a concentration field of a chemical material or a temperature~field, and, consequently, implicitly depends on $(t,x)^\top\in Q_T$.  
Smart fluids have the potential for an application in various areas, such as, \textit{e.g.},
in electronic,  
automobile, heavy machinery, military, and biomedical~industry (\textit{cf}.~\cite[Chap.\ 6]{smart_fluids}~for~an~overview).

\subsection{Related contributions}\vspace*{-0.5mm}

\hspace{5mm}We \hspace*{-0.1mm}recall \hspace*{-0.1mm}known \hspace*{-0.1mm}analytical \hspace*{-0.1mm}and \hspace*{-0.1mm}numerical \hspace*{-0.1mm}results \hspace*{-0.1mm}for \hspace*{-0.1mm}the \hspace*{-0.1mm}unsteady \hspace*{-0.1mm}$p(\cdot,\cdot)$-Navier--Stokes \hspace*{-0.1mm}equations~\hspace*{-0.1mm}\eqref{eq:ptxNavierStokes}.\vspace*{-1mm}\enlargethispage{15mm}

\subsubsection{Existence analyses}\vspace*{-0.5mm}

\hspace{5mm}The existence analysis of the unsteady $p(\cdot,\cdot)$-Navier--Stokes equations \eqref{eq:ptxNavierStokes} is by now well-understood and developed essentially in the following main steps:\vspace*{-0.5mm}

\begin{itemize}[noitemsep,topsep=2pt,leftmargin=!,labelwidth=\widthof{\itshape(iii)},font =\normalfont\itshape] 
	
     \item[(i)] In \cite{rubo}, deploying the $A$-approximation technique,  M.\ R\r{u}\v{z}i\v{c}ka proved the weak solvability~of~\eqref{eq:ptxNavierStokes}~for\vspace*{-0.75mm}
     \begin{align*}
     	p\in C^1(\overline{Q_T})\quad \text{ with }\quad 2\leq p^-\leq  p^+< \min\Big\{\frac{8}{3},\frac{3+p^-}{2(3-p^-)}\Big\}\,;
     \end{align*}
     
     \item[(ii)] In \cite{Z08,Pas11}, by means of hydro-mechanical Parabolic Compensated Compactness~Principles~\mbox{(PCCPs)}, V.\ V.\ Zhikov and S.\ E.\ Pastukhova proved the weak solvability of \eqref{eq:ptxNavierStokes} for\vspace*{-0.75mm}
     \begin{align*}
     	\hspace*{-0.75cm}p \in L^\infty(Q_T)\text{ with }p(t,\cdot)\in \mathcal{P}^{\log}(\Omega)\text{ for a.e.\ }t\in I\quad\text{ and }\quad\frac{3d}{d+2}<p^-\leq p^+<\min\Big\{p^-+2,p^-\frac{d+2}{d}\Big\}\,;
     \end{align*}
     
     \item[(iii)] In \cite{alex-book}, combining a generalized parabolic pseudo-monotonicity theory with the PCCPs~in~\cite{Z08,Pas11}, the weak solvability of \eqref{eq:ptxNavierStokes} was proved for\vspace*{-1.75mm}
     \begin{align*}
     	p\in \mathcal{P}^{\log}(Q_T)\quad \text{ with }\quad p^->\frac{3d}{d+2}\,;
     \end{align*}

     \item[(iv)] In \cite{BGSW22}, M.\ Buli\v{c}ek \textit{et al}.\ proved the weak solvability of \eqref{eq:ptxNavierStokes} for\vspace*{-0.75mm}
      \begin{align*}
     	\hspace*{-0.75cm}p \in L^\infty(Q_T)\text{ with }p(t,\cdot)\in \mathcal{P}^{\log}(\Omega)\text{ for a.e.\ }t\in I\,,\quad\textrm{ess\,sup}_{t\in I}{c_{\log}(p(t,\cdot))}<\infty\,,\quad\text{ and }\quad p^-\ge \frac{3d+2}{d+2}\,.
     \end{align*}
\end{itemize}

Apart from that, in \cite{BG19}, in the case of space periodic boundary conditions, 
D.~Breit~and~F.~Gemeinder  proved the weak solvability of \eqref{eq:ptxNavierStokes} for $ p\in W^{1,\infty}(Q_T)$ with $p^->\frac{3d}{d+2}$.\vspace*{-0.5mm}\newpage

\subsubsection{Numerical analyses}\vspace*{-1mm}

\hspace*{5mm}
On the contrary, the numerical analysis of unsteady problems 
with variable exponents non-linearity  is less developed and we are  only aware of the following contributions treating the unsteady case:\enlargethispage{3mm}

\begin{itemize}[noitemsep,topsep=2pt,leftmargin=!,labelwidth=\widthof{\itshape(iii)},font =\normalfont\itshape] 

     \item[(i)] The \hspace*{-0.1mm}first \hspace*{-0.1mm}semi-discrete \hspace*{-0.1mm}(\textit{i.e.}, \hspace*{-0.1mm}discrete-in-time) \hspace*{-0.1mm}numerical \hspace*{-0.1mm}analysis \hspace*{-0.1mm}of\hspace*{-0.1mm} the \hspace*{-0.1mm}unsteady~\hspace*{-0.1mm}\mbox{$p(\cdot,\cdot)$-Navier--Stokes} equations \eqref{eq:ptxNavierStokes} goes back to L.\ Diening (\textit{cf}.\ \cite{die-diss}).

    \item[(ii)] The first fully-discrete (\textit{i.e.}, discrete-in-time and -space) numerical analysis of the unsteady~$p(\cdot,\cdot)$-Navier--Stokes equations \eqref{eq:ptxNavierStokes} 
    goes back to E.\ Carelli \textit{et al}.\ (\textit{cf}.\ \cite{CHP10}), but considers solely a time-independent power-law index which is less realistic for the models for smart fluids mentioned~above. More precisely, in \cite{CHP10},
    the (weak) convergence  of a fully-discrete Rothe--Galerkin approximation~to the unsteady~$p(\cdot)$-Navier--Stokes equations \eqref{eq:ptxNavierStokes} was established employing a regularized~\mbox{convective} term and continuous approximations $(p_h)_{h>0}\subseteq C^0(\overline{\Omega})$ satisfying $p_h\to p$ in $C^0(\overline{\Omega})$~$(h\to 0)$ and $p_h\ge p$ in $\overline{\Omega}$~for~all~$h>0$, which is restrictive in applications.

    \item[(iii)] In \cite{BR20},  D.\ Breit and P.\ R.\ Mensa
    derived a priori error estimates for a fully-discrete~\mbox{Rothe--Galerkin} approximation of parabolic systems with variable growth conditions without divergence constraint like \eqref{eq:ptxNavierStokes}$_2$ or lower-order terms like the convective term.\vspace*{-2mm}
\end{itemize}

	\subsection{New contributions of the paper}\vspace*{-1mm}\enlargethispage{6mm}

\hspace*{5mm}The purpose of this paper is to improve the literature with respect to several aspects and, in~this~way, to establish a firm foundation for further numerical analyses of related complex~models~involving the unsteady $p(\cdot,\cdot)$-Navier--Stokes equations \eqref{eq:ptxNavierStokes} such as, \textit{e.g.},~\mbox{electro-rheological}~\mbox{fluids}~(\textit{cf}.~\mbox{\cite{RR1,rubo}}),~magneto-rheological fluids (\textit{cf}.\ \mbox{\cite{magneto}}), micro-polar electro-rheological~\mbox{fluids} (\textit{cf}.\ \cite{win-r,eringen-book}), chemically reacting fluids (\textit{cf}.\ \cite{LKM78,HMPR10}), and thermo-rheological~fluids (\textit{cf}.\ \mbox{\cite{Z97,AR06}}):

\begin{itemize}[noitemsep,topsep=2pt,leftmargin=!,labelwidth=\widthof{3.}]
	\item[1.] \textit{Time-dependency of the power-law index:} In this paper, we treat a power-law index $p(\cdot,\cdot)$~that~is~both time- and space-dependent, which makes the extraction of compactness properties more difficult.

	\item[2.] \textit{Discretization of the power-law index:} In this paper, we allow arbitrary uniform approximations~of~the power-law index $p(\cdot,\cdot)$, \textit{i.e.}, we only require that 
	%
	$(p_h)_{h>0}\subseteq L^\infty(Q_T)$ with $p_h\to p$ in $L^\infty(Q_T)$~$(h\to 0)$.

	\item[3.] \textit{Unregularized scheme:} 
	In this paper,  we consider the unregularized convective term. This leads to the restriction $p^-> \smash{\frac{3d+2}{d+2}}$. Deploying a regularized convective~term, allows us to treat~the~case~$p^-\ge 2$. However,  then one needs to perform~a~\mbox{two-level} passage to the limit: first, one passes~to~the~limit~with the discretization parameters of the fully-discrete Rothe--Galerkin approximation,~in~this~way, approximating the regularized~continuous~problem; second, one passes to~the~limit~with~the~\mbox{regularization}. 
	As there is no parabolic Lipschitz truncation technique available (\textit{cf}.\ \cite[Sec.\ 6.3]{alex-book}),~the~\mbox{second}~passage to the limit 
	is still an open problem for the unsteady $p(\cdot,\cdot)$-Navier--Stokes equations \eqref{eq:ptxNavierStokes}. Therefore, 
	we  approximate the unsteady $p(\cdot,\cdot)$-Navier--Stokes equations \eqref{eq:ptxNavierStokes} with a single passage~to~the~limit.
		\item[4.] \textit{Numerical experiments:} Since the weak convergence of a fully-discrete Rothe--Galerkin~\mbox{approximation} is difficult to measure in numerical experiments, we carry out numerical experiments for low~regularity data.
		More precisely, we construct solutions with fractional regularity, gradually reduce the fractional regularity parameters, and measure~convergence~rates that reduce with decreasing (but are stable 
		for fixed) value of the fractional regularity parameters  to get an indication that weak convergence~is~likely.

\end{itemize}

\textit{This paper is organized as follows:} In Section \ref{sec:preliminaries}, we  introduce the used notation, recall~the~\mbox{definitions} of variable Lebesgue spaces and variable Sobolev spaces, and~collect~auxiliary~results.~In~\mbox{Section}~\ref{sec:variable_bochner_lebesgue}, we introduce variable Bochner--Lebesgue spaces and variable Bochner--Sobolev spaces,~and~collect~auxiliary results. In \mbox{Section}~\ref{sec:convergence_of_exponents}, we examine the behavior of the weak topologies of variable Lebesgue~spaces~and~variable Bochner--Lebesgue spaces with respect to converging sequences of variable exponents.~In~Section~\ref{sec:discrete_ptxNavierStokes}, we introduce the discrete spaces and projectors used in the fully-discrete Rothe-Galerkin approximation. In Section \ref{sec:conditionM}, we introduce the concept of non-conforming Bochner condition~(M).~In~\mbox{Section}~\ref{sec:rothe_galerkin},~we~specify a general fully-discrete Rothe--Galerkin approximation and establish its well-posedness,~stability,~and (weak) convergence. 
In Section \ref{sec:application}, we apply the general framework of Section~\ref{sec:rothe_galerkin} to prove~the~two~main results of the paper, \textit{i.e.}, the well-posedness, stability, and (weak) convergence of a  fully-discrete Rothe--Galerkin approximation of
the unsteady $p(\cdot,\cdot)$-Stokes equations (\textit{i.e.}, \eqref{eq:ptxNavierStokes} without convective~term) (\textit{cf}.\ Theorem \ref{thm:S}) and of the unsteady $p(\cdot,\cdot)$-Navier--Stokes equations  \eqref{eq:ptxNavierStokes} (\textit{cf}.\ Theorem \ref{prop:p-NS-scheme}). In~Section~\ref{sec:experiments}, we complement the theoretical findings with some numerical experiments.
\newpage

    \section{Preliminaries}\label{sec:preliminaries}

    \hspace*{5mm}Throughout the entire paper, let $\Omega\hspace*{-0.05em}\subseteq\hspace*{-0.05em} \mathbb{R}^d$,~$d\hspace*{-0.05em}\ge\hspace*{-0.05em}2$,  be a bounded simplicial Lipschitz~domain,~${I\hspace*{-0.05em}\coloneqq \hspace*{-0.05em} (0,T)}$, $0<T<\infty$, a finite time interval, and $Q_T\coloneqq I\times \Omega$ the corresponding time cylinder. The integral mean of a (Lebesgue) integrable function, vector or tensor field $\mathbf{u}\colon M\to\mathbb{R}^{\ell}$, $\ell\in \mathbb{N}$, over a (Lebesgue) measurable set $M\subseteq \mathbb{R}^n$, $n\in \mathbb{N}$, with $\vert M\vert>0$ is denoted by $\langle \mathbf{u}\rangle_M\coloneqq  \smash{\frac 1 {|M|}\int_M \mathbf{u} \,\textup{d}x}$. For (Lebesgue) measurable functions, vector or tensor fields $\mathbf{u},\mathbf{w}\colon M\to \mathbb{R}^{\ell}$, $\ell\in \mathbb{N}$, and a (Lebesgue) measurable set $M\subseteq \mathbb{R}^n$,~${n\in \mathbb{N}}$, 
    we  write $(\mathbf{u},\mathbf{w})_M\coloneqq \int_M \mathbf{u} \odot \mathbf{w}\,\textup{d}x$,  whenever the right-hand side is well-defined,~where
    ${\odot\colon \mathbb{R}^{\ell}\times \mathbb{R}^{\ell}\to \mathbb{R}}$ either denotes scalar multiplication, the Euclidean inner product~or~the~Frobenius~inner~product.

    \subsection{Variable Lebesgue spaces and variable Sobolev spaces}
    
    \hspace{5mm}In this section, we summarize all essential information
    concerning variable Lebesgue spaces and variable Sobolev spaces, which
    will find use in the hereinafter analysis. For an extensive presentation, we refer the reader to the textbooks \cite{dhhr,CUF13}.
    
   Let $M\subseteq \mathbb{R}^n$, $n\in \mathbb{N}$, be a (Lebesgue) measurable set and  denote 
  the set of (Lebesgue) measurable functions, vector or tensor fields defined in $M$  by  $L^0(M;\mathbb{R}^{\ell})$, $\ell\in\mathbb{N}$. A function
   $p\in L^0(M;\mathbb{R}^1)$ is called   \textit{variable
   	exponent} if $p\ge 1$ a.e.\ in $M$ and  the  \textit{set of variable exponents} is denoted by $\mathcal{P}(M)$.~For~$p\in \mathcal{P}(M)$, we denote by
   ${p^+\coloneqq \textup{ess\,sup}_{x\in
   		M}{p(x)}}$~and~${p^-\coloneqq \textup{ess\,inf}_{x\in
   		M}{p(x)}}$ its constant~\textit{limit~exponents}. Then, the \textit{set of bounded variable exponents} is defined by
   $\mathcal{P}^{\infty}(M)\coloneqq \{p\in\mathcal{P}(M)\mid
   p^+<\infty\}$. For ${p\in\mathcal{P}^\infty(M)}$ and  $\mathbf{u}\in L^0(M;\mathbb{R}^{\ell})$, $\ell\in \mathbb{N}$, we define the \textit{modular (with respect to $p$)} 
   by 
   \begin{align*}
   	\rho_{p(\cdot),M}(\mathbf{u})\coloneqq \int_{M}{\vert \mathbf{u}(x)\vert^{p(x)}\,\mathrm{d}x}\,.
   \end{align*}
   
   Then, for $p\in \mathcal{P}^\infty(M)$ and $\ell\in \mathbb{N}$,
   the \textit{variable Lebesgue space} is defined by
   \begin{align*}
   	\smash{L^{p(\cdot)}(M;\mathbb{R}^{\ell})\coloneqq \big\{ \mathbf{u}\in L^0(M;\mathbb{R}^{\ell})\mid \rho_{p(\cdot),M}(\mathbf{u})<\infty\big\}\,.}
   \end{align*}
  In the special case $\ell =1$, we abbreviate $L^{p(\cdot)}(M)\coloneqq L^{p(\cdot)}(M;\mathbb{R}^1)$. The \textit{Luxembourg norm}, \textit{i.e.},
  $$
  		\| \mathbf{u}\|_{p(\cdot),M}\coloneqq \inf\Big\{\lambda> 0\;\Big|\; \rho_{p(\cdot),M}\Big(\frac{\mathbf{u}}{\lambda}\Big)\leq 1\Big\}\,,
  $$ 
  turns $L^{p(\cdot)}(M;\mathbb{R}^{\ell})$~into~a~\mbox{Banach}~space (\textit{cf}.\ \cite[Thm.\  3.2.7]{dhhr}).

    For $p\in \mathcal{P}^\infty(M)$ with $p^->1$,  $p'\coloneqq \frac{p}{p-1} \in
    \mathcal{P}^\infty(M)$  is the conjugate exponent. Then, for $p\in \mathcal{P}^\infty(M)$ with $p^->1$ and $\ell\in \mathbb{N}$, 
    we
    make frequent use of the \textit{variable (exponent) H\"older inequality}, \textit{i.e.},
    \begin{align}
    	\|\mathbf{u}\odot \mathbf{w}\|_{1,M} \leq 2\, \|\mathbf{u}\|_{p'(\cdot),M}\|\mathbf{w}\|_{p(\cdot),M}\,, \label{eq:gen_hoelder}
    \end{align}
     valid for all $\mathbf{u}\in\smash{
    L^{p'(\cdot)}(M;\mathbb{R}^{\ell})}$ and $\mathbf{w}\in \smash{L^{p(\cdot)}(M;\mathbb{R}^{\ell})}$ (\textit{cf}.\ \cite[Lem.\  3.2.20]{dhhr}), where
    $\odot\colon \mathbb{R}^{\ell}\times \mathbb{R}^{\ell}\to \mathbb{R}$ either denotes scalar multiplication, the Euclidean inner product or the Frobenius inner product.
    
    For a bounded (Lebesgue) measurable set $M\subseteq \mathbb{R}^n$, $n\in \mathbb{N}$,  $q,p\in \mathcal{P}^\infty(M)$ with $q\leq p$ a.e.\  in $M$, and $\ell\in \mathbb{N}$, every   $\mathbf{u}\in L^{p(\cdot)}(M;\mathbb{R}^{\ell})$ satisfies $\mathbf{u}\in L^{q(\cdot)}(M;\mathbb{R}^{\ell})$ (\textit{cf}.\ \cite[Cor.\ 3.3.4]{dhhr}) with 
    \begin{align}
    	\|\mathbf{u}\|_{q(\cdot),M}\leq 2\,(1+\vert M\vert)\,\|\mathbf{u}\|_{p(\cdot),M}\,.\label{eq:hoelder_embeddding}
    \end{align}
    
For   $q,p\in\mathcal{P}^\infty(\Omega)$, the \textit{symmetric variable Sobolev space}~is~defined~by\enlargethispage{5mm}
\begin{align*}
		E^{q(\cdot),p(\cdot)}\coloneqq \big\{ \mathbf{u}\in L^{q(\cdot)}(\Omega;\mathbb{R}^d)\mid \mathbf{D}_x \mathbf{u}\in L^{p(\cdot)}(\Omega;\mathbb{R}^{d\times d}_{\textup{sym}})\big\}\,.
\end{align*} 
The \textit{symmetric variable Sobolev norm}, \textit{i.e.},
$$
\| \cdot\|_{q(\cdot),p(\cdot),\Omega}\coloneqq \| \cdot\|_{q(\cdot),\Omega}+\| \mathbf{D}_x(\cdot)\|_{p(\cdot),\Omega}\,,
$$ 
turns $E^{q(\cdot),p(\cdot)}$ into a Banach space  (\textit{cf}.\  \cite[Prop.\  2.9]{alex-book}). 
Then, the closure of $C^\infty_c(\Omega;\mathbb{R}^d)$ in $E^{q(\cdot),p(\cdot)}$ is denoted by $\smash{E^{q(\cdot),p(\cdot)}_0}$, the closure of 
\begin{align*}
	\mathcal{V}\coloneqq \big\{\boldsymbol{\varphi}\in C^\infty_c(\Omega;\mathbb{R}^d)\mid \textup{div}_x\boldsymbol{\varphi}=0\textup{ in }\Omega\big\}
\end{align*}
in \hspace*{-0.15mm}$E^{q(\cdot),p(\cdot)}$ \hspace*{-0.15mm}is \hspace*{-0.15mm}denoted \hspace*{-0.15mm}by \hspace*{-0.25mm}$E^{q(\cdot),p(\cdot)}_{0,\textup{div}_x}$, \hspace*{-0.15mm}and \hspace*{-0.15mm}the \hspace*{-0.15mm}closure \hspace*{-0.15mm}of \hspace*{-0.15mm}$	\mathcal{V}$ \hspace*{-0.15mm}in \hspace*{-0.15mm}$Y\hspace*{-0.15em}\coloneqq\hspace*{-0.15em} L^2(\Omega;\mathbb{R}^d)$~\hspace*{-0.15mm}is~\hspace*{-0.15mm}\mbox{denoted}~\hspace*{-0.15mm}by~\hspace*{-0.25mm}$H$.

    \section{Variable Bochner--Lebesgue spaces and variable Bochner--Sobolev spaces}\label{sec:variable_bochner_lebesgue}\vspace*{-0.25mm}

    \hspace{5mm}In \hspace*{-0.05mm}this \hspace*{-0.05mm}section, \hspace*{-0.05mm}we \hspace*{-0.05mm}recall \hspace*{-0.05mm}the \hspace*{-0.05mm}definitions \hspace*{-0.05mm}and \hspace*{-0.05mm}discuss \hspace*{-0.05mm}properties \hspace*{-0.05mm}of \hspace*{-0.05mm}variable \hspace*{-0.05mm}Bochner--Lebesgue~\hspace*{-0.05mm}spaces~\hspace*{-0.05mm}and variable Bochner--Sobolev spaces, the appropriate substitutes of classical Bochner--Lebesgue spaces and classical Bochner--Sobolev spaces for the treatment of the unsteady $p(\cdot,\cdot)$-Navier--Stokes~equations~\eqref{eq:ptxNavierStokes}; we refer the reader to the textbook \cite{alex-book} or to the thesis \cite{alex-diss}, for a extensive presentation.  We emphasize that variable Bochner--Lebesgue~spaces and variable Bochner--Sobolev spaces --as it is done, \textit{e.g.}, in \cite{alex-diss}-- should more accurately  be called variable \textit{exponent} Bochner--Lebesgue spaces and variable \textit{exponent} Bochner--Sobolev spaces, respectively. However, we suppress the word \textit{`exponent'} in favor of readability.\vspace*{-0.25mm}\enlargethispage{7mm}

    \subsection{Variable Bochner--Lebesgue spaces}\vspace*{-0.25mm}
	
	\hspace{5mm}Throughout the entire subsection, unless otherwise specified, let $q,p\in \mathcal{P}^\infty(Q_T)$ with $q^-,p^->1$.\vspace*{-0.25mm}

	\begin{definition}\label{3.1} We define for a.e.\   $t\in I$, the 
		\textup{time slice spaces}\vspace*{-0.25mm}
		\begin{align*}
			\smash{U^{q,p}_{\D}(t)\coloneqq  E^{q(t,\cdot),p(t,\cdot)}_0 \,,\qquad
			V^{q,p}_{\D}(t)\coloneqq  E^{q(t,\cdot),p(t,\cdot)}_{0,\textup{div}_x} \,.}
		\end{align*}
		Moreover, we set  $\smash{U^{q,p}_{\plus}\coloneqq E^{q^{\plus},p^{\plus}}_0}$, $\smash{U^{q,p}_{\minus}\coloneqq  E^{q^{\minus},p^{\minus}}_0 }$,
		$	\smash{V^{q,p}_{\plus}\coloneqq E^{q^{\plus},p^{\plus}}_{0,\textup{div}_x}}$, and $\smash{V^{q,p}_{\minus}\coloneqq  	E^{q^{\minus},p^{\minus}}_{0,\textup{div}_x} }$. 	
	\end{definition}

	By means of the time slice spaces $\smash{\{U^{q,p}_{\D}(t)\}_{t\in I}}$ and $\smash{\{V^{q,p}_{\D}(t)\}_{t\in I}}$, we next introduce\vspace*{-0.25mm}
	
	\begin{definition}[Variable Bochner--Lebesgue spaces]\label{3.3} We define the \textup{variable Bochner--Lebesgue spaces}\vspace*{-0.25mm}
		\begin{align*}
			\begin{aligned}
		\bXp
		&\coloneqq \big\{\boldsymbol{u}\in L^{q(\cdot,\cdot)}(Q_T;\mathbb{R}^d)
		\mid \mathbf{D}_x\boldsymbol{u}\in L^{p(\cdot,\cdot)}(Q_T;\mathbb{R}_{\mathrm{sym}}^{d\times d})\,,\;\boldsymbol{u}(t)\in U^{q,p}_{\D}(t)\text{ for a.e.\  }t\in I\big\}\,,\\
  \bVp
		&\coloneqq \big\{\boldsymbol{v}\in L^{q(\cdot,\cdot)}(Q_T;\mathbb{R}^d)
		\mid \mathbf{D}_x\boldsymbol{v}\in L^{p(\cdot,\cdot)}(Q_T;\mathbb{R}_{\mathrm{sym}}^{d\times d})\,,\;\;\boldsymbol{v}(t)\in V^{q,p}_{\D}(t)\text{ for a.e.\  }t\in I\big\}\,.
	\end{aligned}
		\end{align*}
		Moreover, we set  $	\smash{\bXplus\hspace*{-0.175em}\coloneqq\hspace*{-0.175em} L^{\max\{q^{\plus},p^{\plus}\}}(I,U^{q,p}_{\plus})}$, $\smash{\bXminus\hspace*{-0.175em}\coloneqq\hspace*{-0.175em} L^{\min\{q^{\minus},p^{\minus}\}}(I,U^{q,p}_{\minus})}$, $\smash{
		\mathbfcal{V}_{\plus}^{q,p}\hspace*{-0.175em}\coloneqq \hspace*{-0.175em} L^{\max\{q^{\plus},p^{\plus}\}}(I,V^{q,p}_{\plus})}$, and $\smash{\mathbfcal{V}_{\minus}^{q,p}\coloneqq L^{\min\{q^{\minus},p^{\minus}\}}(I,V^{q,p}_{\minus})}$. 
	\end{definition}

	\begin{proposition}\label{3.4}
		The space  $\smash{\bXp}$ is a separable, reflexive Banach space when equipped with~the~norm\vspace*{-0.25mm}
        \begin{align*}
		\|\cdot\|_{\smash{\bXp}}
		\coloneqq \|\cdot\|_{q(\cdot,\cdot),Q_T}
		+\|\mathbf{D}_x(\cdot)\|_{p(\cdot,\cdot),Q_T}\,.
		\end{align*}
        In addition, the space $\smash{\bVp}$ is a closed subspace of $\smash{\bXp}$.
	\end{proposition}
	
	\begin{proof}
		See~\cite[Prop.\ 3.2, Prop.\ 3.3 \& Prop.\ 4.1]{alex-book}.
	\end{proof}
 
	We have the following characterization for the dual space of $\smash{\bXp}$.\vspace*{-0.25mm}
	
	\begin{proposition}\label{3.7}The linear mapping
		$\smash{\mathbfcal{J}_{\D}\colon 
		L^{q'(\cdot,\cdot)}(Q_T;\mathbb{R}^d)\times 
		L^{p'(\cdot,\cdot)}(Q_T;\mathbb{R}_{\mathrm{sym}}^{d\times d})\to (\bXp)^*}$,   for every 
		$\boldsymbol{f}\in L^{q'(\cdot,\cdot)}(Q_T;\mathbb{R}^d)$, 
		$\boldsymbol{F}\in L^{p'(\cdot,\cdot)}(Q_T;\mathbb{R}_{\mathrm{sym}}^{d\times d})$, 
		and $\boldsymbol{u}\in\bXp$ defined 
		by\vspace*{-0.25mm}
		\begin{align*}
		\langle	\mathbfcal{J}_{\D}(\boldsymbol{f},
		\boldsymbol{F}),\boldsymbol{u}\rangle
		_{\bXp}\coloneqq ( \boldsymbol{f},\boldsymbol{u})_{Q_T}
		+(\boldsymbol{F},
	\mathbf{D}_x\boldsymbol{u})_{Q_T}\,,
		\end{align*}
		is well-defined, linear, and Lipschitz continuous with constant bounded by $2$. For~\mbox{every}~$\smash{\boldsymbol{f}^*\in (\bXp)^*}$, there exist $\boldsymbol{f}\in L^{q'(\cdot,\cdot)}(Q_T;\mathbb{R}^d)$ and
		$\boldsymbol{F}\in L^{p'(\cdot,\cdot)}(Q_T;\mathbb{R}_{\mathrm{sym}}^{d\times d})$ such that $\boldsymbol{f}^*= \mathbfcal{J}_{\D}(\boldsymbol{f},
		\boldsymbol{F})$~in~$(\bXp)^*$~and\vspace*{-1mm}
        \begin{align*}
            \frac{1}{2}\,\|\boldsymbol{f}^*\|_{(\bXp)^*}\leq \|\boldsymbol{f}\|_{q'(\cdot,\cdot),Q_T}+\|\boldsymbol{F}\|_{p'(\cdot,\cdot),Q_T}\leq 2\,\|\boldsymbol{f}^*\|_{(\bXp)^*}\,.
        \end{align*}
	\end{proposition}
	
	\begin{proof}
		See~\cite[Prop.\ 3.4]{alex-book}.
	\end{proof}
	
	\begin{remark}[The notation $q(\cdot,\cdot)$,  $p(\cdot,\cdot)$]
		We employ the notation $q(\cdot,\cdot)$, $p(\cdot,\cdot)$ instead of $q(\cdot)$, $p(\cdot)$, respectively, to emphasize that the variable exponents $q,p\in \mathcal{P}^\infty(Q_T)$ are both time- and space-dependent.
	\end{remark}

   The variable exponents $q,p\in \mathcal{P}^{\infty}(Q_T)$ are  \textit{$\log$-Hölder continuous} (written $q,p\in \mathcal{P}^{\log}(Q_T)$) if there exists a constant $c>0$ such that for every $z, \tilde{z}\in Q_T$ with $0<\vert z-\tilde{z}\vert\leq \frac{1}{2}$, it holds that\vspace*{-0.25mm}
    \begin{align*}
    	\vert q(z)-q(\tilde{z})\vert +\vert p(z)-p(\tilde{z})\vert \leq \frac{c}{-\ln\vert z-\tilde{z}\vert}\,.
    \end{align*}

    \begin{proposition}\label{prop:density_in_Vqp}
    	Let $q,p\in \mathcal{P}^{\log}(Q_T)$ with $q\ge p$ in $Q_T$. Then,  the~space $  \mathbfcal{D}_T\coloneqq \{\boldsymbol{\varphi}\in C^\infty_c(Q_T;\mathbb{R}^d)\mid \mathrm{div}_x\boldsymbol{\varphi}=0\text{ in }Q_T\}$ lies densely in $\bVp$.
    \end{proposition}
    
    \begin{proof}
    	See   \cite[Prop.\ 4.16]{alex-book}.
    \end{proof}

    \subsection{Variable Bochner--Sobolev spaces}

    \hspace{5mm}Throughout the entire subsection, unless otherwise specified, let $q,p\in \mathcal{P}^{\log}(Q_T)$ with $q\ge p$ in $Q_T$ and  $q^-,p^->1$.\enlargethispage{9mm}

	\begin{definition}\label{def:generalized_time_derivative}  A function $\boldsymbol{v}\in L^1_{\textup{loc}}(Q_T)$ has a \textup{generalized time derivative in} $(\bVp)^*$ if there exists a functional $\boldsymbol{f}^*\in (\bVp)^*$  such that for every $\boldsymbol{\varphi}\in \mathbfcal{D}_T$, it holds that
		\begin{align*}
		-(\boldsymbol{v},\partial _t\boldsymbol{\varphi})_{Q_T}=\langle\boldsymbol{f}^*,\boldsymbol{\varphi}\rangle_{\smash{\bVp}}\,.
		\end{align*}
		In this case, we define 
		\begin{align}\label{def:generalized_time_derivative.1}
			\frac{\mathbf{d}_\sigma\boldsymbol{v}}{\mathbf{dt}}\coloneqq \boldsymbol{f}^*\quad\text{ in }(\bVp)^*\,.
		\end{align} 
	\end{definition}
	
	\begin{remark}\label{4.2} The generalized time derivative in the sense of Definition~\ref{def:generalized_time_derivative} is unique and, therefore, \eqref{def:generalized_time_derivative.1} is well-defined (\textit{cf}.\ \cite[Lem.\ 4.5]{alex-book}). This is a direct consequence of the density result~in~Proposition~\ref{prop:density_in_Vqp}.
	\end{remark}
	
	By means of the generalized time derivative $\frac{\mathbf{d}_\sigma}{\mathbf{dt}}$ (\textit{cf}.\ Definition \ref{def:generalized_time_derivative}), we next introduce 
	
	\begin{definition}[Variable Bochner--Sobolev space]\label{def:bochner_sobolev}
		We define the \textup{variable Bochner--Sobolev~space}
		\begin{align*}
		\mathbfcal{W}^{q,p}_{\D}\coloneqq 
		\bigg\{\boldsymbol{v}\in \bVp\;\Big|\; \exists\frac{\mathbf{d}_\sigma\boldsymbol{v}}{\mathbf{dt}}\in 	(\bVp)^*\bigg\}\,.
		\end{align*}
	\end{definition}
	
	\begin{proposition}\label{prop:bochner_sobolev_norm}
		\hspace*{-0.75mm}The \hspace*{-0.15mm}space \hspace*{-0.15mm}$\mathbfcal{W}^{q,p}_{\D}$ \hspace*{-0.15mm}forms \hspace*{-0.15mm}a  \hspace*{-0.15mm}separable, \hspace*{-0.15mm}reflexive \hspace*{-0.15mm}Banach \hspace*{-0.15mm}space \hspace*{-0.15mm}when~\hspace*{-0.15mm}equipped~\hspace*{-0.15mm}with~\hspace*{-0.15mm}the~\hspace*{-0.15mm}norm
		\begin{align*}
		\|\cdot\|_{\mathbfcal{W}^{q,p}_{\D}}\coloneqq \|\cdot\|_{\smash{\bVp}}+\bigg\|\frac{\mathbf{d}_\sigma\,\cdot}{\mathbf{dt}}\bigg\|_{\smash{(\bVp)^*}}\,.
		\end{align*}
	\end{proposition}
	
	\begin{proof}
		See~\cite[Prop.\ 4.18]{alex-book}.
	\end{proof}
	
	We have the following equivalent characterization of $\mathbfcal{W}^{q,p}_{\D}$.
	\begin{proposition}\label{prop:equivalent_characterization}
		Let $\boldsymbol{v}\in \bVp$  and $\boldsymbol{f}^*\in  (\bVp)^*$. Then, the following~statements~are~\mbox{equivalent:}
		\begin{description}[noitemsep,topsep=2pt,leftmargin=!,labelwidth=\widthof{(ii)},font=\normalfont\itshape]
			\item[(i)]  
			$\boldsymbol{v}\in \mathbfcal{W}^{q,p}_{\D}$ with $\frac{\mathbf{d}_\sigma\boldsymbol{v}}{\mathbf{dt}}=\boldsymbol{f}^*$ in $(\bVp)^*$.\vspace*{0.5mm}
			\item[(ii)] For every $\mathbf{z}\in V^{q,p}_{\plus}$ and
			$\varphi\in C^\infty_c(I)$, it holds that $-(\boldsymbol{v},\mathbf{z}\varphi^\prime)_{Q_T}=\langle \boldsymbol{f}^*,\mathbf{z}\varphi\rangle_{\smash{\bVp}}$.
		\end{description}
	\end{proposition}

	\begin{proof}
		See \cite[Prop.\ 4.20]{alex-book}.
	\end{proof}

    The following integration-by-parts formula is a cornerstone in the (weak) convergence analysis for the fully-discrete Rothe--Galerkin approximation of the unsteady $p(\cdot,\cdot)$-Navier--Stokes equations \eqref{eq:ptxNavierStokes}.
	
	\begin{proposition}\label{prop:integration-by-parts}
		Let $q,p\in \mathcal{P}^{\log}(Q_T)$ with $ q\ge p\ge 2$ in $Q_T$.
		Then, the following statements apply:
		\begin{description}[noitemsep,topsep=2pt,leftmargin=!,labelwidth=\widthof{(ii)},font=\normalfont\itshape]
			\item[(i)] Each function $\boldsymbol{v}\in \mathbfcal{W}^{q,p}_{\D}$ (defined a.e.\ in $I$) has a unique representation $\boldsymbol{v}_{c}\in C^0(\overline{I};H)$ and the resulting mapping $(\cdot)_{c}\colon \mathbfcal{W}^{q,p}_{\D}\to
			C^0(\overline{I};H)$ is an embedding.
			\item[(ii)] For every $\boldsymbol{v},\boldsymbol{z}\in \mathbfcal{W}^{q,p}_{\D}$ and
			$t,t'\in \overline{I}$ with $t'\leq t$, it holds that
			\begin{align*}
			\bigg\langle
				\frac{\mathbf{d}_\sigma\boldsymbol{v}}{\mathbf{dt}},\boldsymbol{z}\chi_{\left[t',t\right]}\bigg\rangle_{\smash{\bVp}}
			=\left[(\boldsymbol{v}_c(s), 
			\boldsymbol{z}_c(s))_{\Omega}\right]^{s=t}_{s=t'}-\bigg\langle
				\frac{\mathbf{d}_\sigma\boldsymbol{z}}{\mathbf{dt}},\boldsymbol{v}\chi_{\left[t',t\right]}\bigg\rangle_{\smash{\bVp}}\,.
			\end{align*}
		\end{description}  
	\end{proposition}

	\begin{proof}
		See~\cite[Prop.\ 4.23]{alex-book}.
	\end{proof} 
	
	\begin{remark} 
		The requirement $p^-\ge 2$  in Proposition \ref{prop:integration-by-parts} has an entirely technical origin (\textit{cf.} \cite[Rem.~4.16]{alex-book}) and might be improved in the future. If one is capable~of~proving~Proposition~\ref{prop:integration-by-parts} without the restriction $p^-\ge 2$, then one can equally remove this restriction from Assumption \ref{assumption} below.
	\end{remark}
	
	\begin{definition}\label{def:generalized_evolution_equation}
		Let $q,p\in \mathcal{P}^{\log}(Q_T)$ with $q\ge p\ge 2$ in $Q_T$. Moreover, let $\mathbf{v}_0\in H$, $\boldsymbol{f}^*\in (\bVp)^*$, and $\mathbfcal{A}\colon \mathbfcal{W}^{q,p}_{\D}\to (\bVp)^*$. Then, the initial value problem
		\begin{align}\label{eq:generalized_evolution_equation}
			\begin{aligned}
				\frac{\mathbf{d}_\sigma\boldsymbol{v}}{\mathbf{dt}}+\mathbfcal{A}\boldsymbol{v}&=\boldsymbol{f}^*&&\quad \text{ in }(\bVp)^*\,,\\
				\boldsymbol{v}_c(0)&=\mathbf{v}_0&&\quad\text{ in }H\,,
			\end{aligned}
		\end{align}
		is called a \textup{generalized evolution equation}. Here, the initial condition is to be understood in the sense of the unique continuous representation $\boldsymbol{v}_c\in C^0(\overline{I};H)$ (\textit{cf}.\ Proposition \ref{prop:integration-by-parts}).
	\end{definition}

    \section{Convergence results for sequences of variable exponents}\label{sec:convergence_of_exponents}

	\label{sec:8.2}
	\hspace{5mm}One important aspect in the numerical approximation of the unsteady $p(\cdot,\cdot)$-Navier--Stokes equations \eqref{eq:ptxNavierStokes} 
    consists in the discretization of the time- and space-dependent  extra-stress tensor. It~is~convenient to approximate the power-law index 
    via variable exponents that possess a more discrete~structure, \textit{e.g.}, are element-wise constant with respect to a given time-space discretization~of~the~\mbox{time-space}~\mbox{cylinder}~$Q_T$.  
    This section collects continuity results with respect~to~the~convergence~of~exponents, \textit{i.e.},~we~examine~the
     behavior~of~the~norms $\|\cdot\|_{p_n(\cdot),Q_T}$, $\|\cdot\|_{\smash{\mathbfcal{U}^{q_n,p_n}_{\D}}}$~and~${\|\!\cdot\!\|_{\smash{(\mathbfcal{U}^{q_n,p_n}_{\D})^*}}}$~when~exponents~${(q_n)_{n\in \mathbb{N}},(p_n)_{n\in \mathbb{N}}\hspace*{-0.15em}\subseteq \hspace*{-0.15em} L^\infty(Q_T)}$ converge uniformly to $q,p\in \mathcal{P}^{\log}(Q_T)$, respectively. 
	As a starting point serves the following lemma.\enlargethispage{5mm}

	\begin{lemma}\label{8.2.1}
		Let $M\subseteq\mathbb{R}^m$, $m\in\mathbb{N}$, be a (Lebesgue) measurable set, $(p_n)_{n\in \mathbb{N}}\subseteq \mathcal{P}^\infty(M)$,~and~${p\in \mathcal{P}^\infty(M)}$ such that $p_n\to p$ $(n\to \infty)$ a.e.\  in $M$. Then, it follows that  $\|\mathbf{u}\|_{p(\cdot),M}\leq \liminf_{n\to\infty}{\|\mathbf{u}\|_{p_n(\cdot),M}}$ for all $\mathbf{u}\in L^0(M;\mathbb{R}^{\ell})$, $\ell\in \mathbb{N}$.
	\end{lemma}

	\begin{proof}
		See \cite[Cor.\  3.5.4]{dhhr}.
	\end{proof}
	
	Aided by Lemma \ref{8.2.1}, we obtain the following extension concerning uniform convergence of exponents. 
	This result is a mainstay in the numerical treatment of the variable exponent structure via uniform approximation with
	discontinuous exponents. In what follows, we denote for $q,p\in L^\infty(M)$~by~$q\prec p$~in~$M$ (or $q\succ p$ in $M$) if there exists $\varepsilon>0$ such that
	$q+\varepsilon\leq p$ a.e.\  in $M$ (or $q+\varepsilon\geq p$ a.e.\  in $M$).\enlargethispage{2mm}
	
	\begin{lemma}\label{8.2.2}
		Let $M\hspace*{-0.05em}\subseteq\hspace*{-0.05em} \mathbb{R}^m$, $m\hspace*{-0.05em}\in\hspace*{-0.05em} \mathbb{N}$, be a bounded (Lebesgue) measurable set, $p\hspace*{-0.05em}\in\hspace*{-0.05em} \mathcal{P}^\infty(M)$~such~that~${p^-\hspace*{-0.05em}>\hspace*{-0.05em}1}$ and $(p_n)_{n\in \mathbb{N}}\subseteq \mathcal{P}^\infty(M)$ such that $p_n\to  p $ in $L^\infty(M)$ $(n\to \infty)$. Moreover,~let~${\mathbf{u}_n\in L^{p_n(\cdot)}(M;\mathbb{R}^{\ell})}$,~${n\in \mathbb{N}}$, where $\ell\in \mathbb{N}$, be a sequence such that $\sup_{n\in \mathbb{N}}{\|\mathbf{u}_n\|_{p_n(\cdot),M}}<\infty$.  
		Then, there exists a~subsequence~$\mathbf{u}_{n_k}\hspace*{-0.1em}\in\hspace*{-0.1em} L^{p_{n_k}(\cdot)}(M;\mathbb{R}^{\ell})$, $k\hspace*{-0.1em}\in\hspace*{-0.1em} \mathbb{N}$, and a limit $\mathbf{u}\in L^{p(\cdot)}(M;\mathbb{R}^{\ell})$ such that for every $r\in  \mathcal{P}^\infty(M)$~with~${r\prec  p}$~in~$M$, it holds that
        \begin{align*}
            \mathbf{u}_{n_k}\rightharpoonup \mathbf{u}\quad\text{ in }L^{r(\cdot)}(M;\mathbb{R}^{\ell})\quad (k\to \infty)\,.
        \end{align*}		
	\end{lemma}

	\begin{proof}
		Set $c_0\coloneqq \sup_{n\in \mathbb{N}}{\|\mathbf{u}_n\|_{p_n(\cdot),M}}<\infty$ and let $s_0\in \mathcal{P}^\infty(M)$ be such that $s_0\prec p$ in $M$ and $s_0^->1$. Then, there exists $n_0\in \mathbb{N}$ such that $s_0\leq p_n$ a.e.\ in $M$ for all $n\in \mathbb{N}$ with $n\ge n_0$. Then, resorting~to~the variable Hölder inequality \eqref{eq:hoelder_embeddding},  for every $n\in \mathbb{N}$ with $n\ge n_0$, we find that
		\begin{align}
			\|\mathbf{u}_n\|_{s_0(\cdot),M}\leq 2\,(1+\vert M\vert)\,\|\mathbf{u}_n\|_{p_n(\cdot),M}\leq 2\,(1+\vert M\vert)\,c_0\,,\label{subs1}
		\end{align}
		\textit{i.e.}, $(\mathbf{u}_n)_{n\in \mathbb{N}}\subseteq L^{s_0(\cdot)}(M;\mathbb{R}^{\ell})$ is bounded. Owing to the reflexivity of $L^{s_0(\cdot)}(M;\mathbb{R}^{\ell})$ (\textit{cf}.~\cite[Thm.~3.4.7]{dhhr}),  
  there exists a subsequence $\mathbf{u}_{n_k}\in L^{p_{n_k}(\cdot)}(M;\mathbb{R}^{\ell})$, $k\in \mathbb{N}$,
        and a (weak) limit $\mathbf{u}\in  L^{s_0(\cdot)}(M;\mathbb{R}^{\ell})$  such that
		\begin{align}
				\mathbf{u}_{n_k}\rightharpoonup \mathbf{u}\quad\text{ in }L^{s_0(\cdot)}(M;\mathbb{R}^{\ell})\quad(k\to \infty)\,.\label{subs2}
		\end{align}
		If $s\in \mathcal{P}^\infty(M)$ is another variable exponent such that $s\prec p$  in $M$ and $s^->1$. Then, there exists $n_s\in \mathbb{N}$ such that $s\leq p_n$ a.e.\ in $M$ for all $n\in \mathbb{N}$ with $n\ge n_s$. Hence,
		$(\mathbf{u}_{n_k})_{k\in \mathbb{N}}\subseteq L^{s(\cdot)}(M;\mathbb{R}^{\ell})$ is again bounded, \textit{i.e.}, satisfies \eqref{subs1} with the same constant, and, thus, 
        there~exists~a~subsequence~${(\mathbf{u}_{n_{k_{i}^s}})_{i\in \mathbb{N}}\subseteq L^{s(\cdot)}(M;\mathbb{R}^{\ell})}$ (initially depending on this fixed $s\in \mathcal{P}^\infty(M)$) and a (weak) limit $\mathbf{u}_s\in L^{s(\cdot)}(M;\mathbb{R}^{\ell})$ such that
        \begin{align}\label{subs3}
            	\mathbf{u}_{n_{k_{i}^s}}\rightharpoonup \mathbf{u}_s\quad\text{ in }L^{s(\cdot)}(M;\mathbb{R}^{\ell})\quad(i\to \infty)\,.
        \end{align}
		By \eqref{subs2}, \eqref{subs3}, and the uniqueness of weak limits, we have that $\mathbf{u}_s= \mathbf{u}\in L^{\max\{s_0,s\}}(M;\mathbb{R}^{\ell})$. As a result, the standard convergence principle  (\textit{cf.}~\cite[Kap.\ I, Lem.\ 5.4]{GGZ}) yields~that~${\mathbf{u}_{n_k}\hspace*{-0.15em}\rightharpoonup \hspace*{-0.15em} \mathbf{u}} $~in~$L^{s(\cdot)}(M;\mathbb{R}^{\ell})$~${(k\hspace*{-0.15em} \to \hspace*{-0.15em}\infty)}$. 
		By the weak lower semi-continuity of $\|\cdot\|_{s(\cdot),M}$ and \eqref{subs1}, for every $s\hspace*{-0.1em}\in\hspace*{-0.1em}  \mathcal{P}^\infty(M)$ with $s \hspace*{-0.1em}\prec\hspace*{-0.1em} p$ in~$M$,~we~obtain
		\begin{align}
			\|\mathbf{u}\|_{s(\cdot),M}\leq \liminf_{ k\to \infty}{\|\mathbf{u}_{n_k}\|_{s(\cdot),M}}\leq 2\,(1+\vert M\vert)\,c_0\,.\label{subs4}
		\end{align}
		Next, choosing $(s_n)_{n\in \mathbb{N}} \subseteq \mathcal{P}^\infty(M)$ with $s_n\prec p$ in $M$ for all $n\in \mathbb{N}$  and $s_n \to  p$ in $L^\infty(M)$ $(n \to  \infty)$, using   Lemma \ref{8.2.1}  and \eqref{subs4}, we conclude that 
		\begin{align*}
			\|\mathbf{u}\|_{p(\cdot),M}\leq\liminf_{n\to \infty}{\|\mathbf{u}\|_{s_n(\cdot),M}}\leq 2\,(1+\vert M\vert)\,c_0\,,
		\end{align*}
        \textit{i.e.}, the claimed integrability $\mathbf{u}\in L^{p(\cdot)}(M;\mathbb{R}^{\ell})$.
	\end{proof}

	Next, we derive an analogue of Lemma \ref{8.2.2} for  variable Bochner--Lebesgue spaces and their~dual~spaces. To this end, for the remainder of the paper, we make~the~following~assumption.

    \begin{assumption}[Uniform exponent approximation]\label{assum:p_h}
		We assume that $q,p\in \mathcal{P}^{\log}(Q_T)$ with $q^-,p^-> 1$ and that $(q_n)_{n\in \mathbb{N}},(p_n)_{n\in \mathbb{N}}\subseteq \mathcal{P}^\infty(Q_T)$ are sequences such that 
		\begin{align*}
			\begin{aligned}
				q_n&\to  q&& \quad\text{ in }L^\infty(Q_T)&&\quad(n\to \infty)\,,\\
				p_n&\to  p &&\quad\text{ in }L^\infty(Q_T)&&\quad(n\to \infty)\,.
			\end{aligned}
		\end{align*}
	\end{assumption}

	\begin{proposition}\label{prop:Xqp_weak_compact}
		Let  $\boldsymbol{u}_n\in \bXpn$, $n\in \mathbb{N}$, be a sequence satisfying
		\begin{align}
			\sup_{n\in \mathbb{N}}{\|\boldsymbol{u}_n\|_{\bXpn}}<\infty\,.\label{eq:Xqp_weak_compact.1}
		\end{align}
		Then, there exists a subsequence $\boldsymbol{u}_{n_k}\in \mathbfcal{U}_{\D}^{q_{\smash{n_k}},p_{\smash{n_k}}}$, $k\in \mathbb{N}$,
        and a limit $\boldsymbol{u}\in \bXp$, such that  for~every ${r,s\in \mathcal{P}^\infty(Q_T)}$ with $r\prec q$ in $Q_T$ and $s\prec p$ in $Q_T$, it holds that
		\begin{align*}
			\boldsymbol{u}_{n_k} \rightharpoonup \boldsymbol{u}\quad\textup{ in }\mathbfcal{U}_{\D}^{r,s}\quad(k\to \infty )\,.
		\end{align*}
	\end{proposition}

	\begin{proof}
        From \eqref{eq:Xqp_weak_compact.1}, it follows that $\sup_{n\in \mathbb{N}}{\|\boldsymbol{u}_n\|_{\smash{q_n(\cdot,\cdot),Q_T}}}<\infty$ and $\sup_{n\in \mathbb{N}}{\|\mathbf{D}_x\boldsymbol{u}_n\|_{\smash{p_n(\cdot,\cdot),Q_T}}}<\infty$.~Hence, 
		Lemma \ref{8.2.2} yields a subsequence $\boldsymbol{u}_{n_k}\in  \mathbfcal{U}_{\D}^{q_{\smash{n_k}},p_{\smash{n_k}}}$, $k\in \mathbb{N}$, and (weak) limits  $\boldsymbol{u}\in L^{q(\cdot,\cdot)}(Q_T;\mathbb{R}^d)$~and $\mathbfcal{D}\in  L^{p(\cdot,\cdot)}(Q_T;\mathbb{R}^{d\times d}_{\textup{sym}})$ such that for every  $r,s\in \mathcal{P}^\infty(Q_T)$ with $r\prec q$ in $Q_T$ and $s\prec p$ in $Q_T$,~it~holds~that
		\begin{align*}
            \begin{aligned}
			\boldsymbol{u}_{n_k}& \rightharpoonup \boldsymbol{u}&&\quad\text{ in }L^{r(\cdot,\cdot)}(Q_T;\mathbb{R}^d)&&\quad(k\to \infty)\,,\\
			\mathbf{D}_x\boldsymbol{u}_{n_k}&\rightharpoonup \mathbfcal{D}&&\quad\text{ in }L^{s(\cdot,\cdot)}(Q_T;\mathbb{R}^{d\times d}_{\textup{sym}})&&\quad(k\to \infty)\,.
            \end{aligned}
		\end{align*}
        Then, for $k_0=k_0(r,s)\in \mathbb{N}$ such that   $r\leq q_{n_k}$ a.e.\  in $Q_T$ and $s\leq p_{n_k}$ a.e.\  in $Q_T$~for~all~${k\in \mathbb{N}}$~with~${k\ge k_0}$,  $(\boldsymbol{u}_{n_k})_{k\in \mathbb{N}, k\ge k_0 }$ is a weak Cauchy sequence in $\mathbfcal{U}_{\D}^{r,s}$. Since $\mathbfcal{U}_{\D}^{r,s}$ is weakly closed, we  infer that $\boldsymbol{u}\in \mathbfcal{U}_{\D}^{r,s}$ with $\mathbf{D}_x\boldsymbol{u}=\mathbfcal{D}$ in $L^{s(\cdot,\cdot)}(Q_T;\mathbb{R}^{d\times d}_{\textup{sym}})$. Eventually, since $\boldsymbol{u}\in L^{q(\cdot,\cdot)}(Q_T;\mathbb{R}^d)$ and $\mathbfcal{D}\in L^{p(\cdot,\cdot)}(Q_T;\mathbb{R}^{d\times d}_{\textup{sym}})$, we conclude that $\boldsymbol{u}\in \bXp$.
	\end{proof}

	\begin{proposition}\label{prop:Xqp_prime_weak_compact}
			Let $\boldsymbol{f}_n^*\in (\bXpn)^*$, $n\in \mathbb{N}$, be a sequence satisfying
		\begin{align}
		\sup_{n\in \mathbb{N}}{\|\boldsymbol{f}_n^*\|_{(\bXpn)^*}}<\infty\,.\label{eq:8.3.a}
		\end{align}
		Then, there exists a subsequence $\boldsymbol{f}_{n_k}^*\in  ( \mathbfcal{U}_{\D}^{q_{\smash{n_k}},p_{\smash{n_k}}})^*$, $k\in  \mathbb{N}$, and a limit $\boldsymbol{f}^*\in  (\bXp)^*$, such~that~for~every $r,s\in \mathcal{P}^\infty(Q_T)$ with $r\succ q$ in $Q_T$ and $s\succ p$ in $Q_T$, it holds that
		\begin{align}
		(\mathrm{id}_{\smash{\mathbfcal{U}_{\D}^{r,s}}})^*\boldsymbol{f}_{n_k}^*\rightharpoonup (\mathrm{id}_{\smash{\mathbfcal{U}_{\D}^{r,s}}})^*\boldsymbol{f}^*\quad\textup{ in }(\mathbfcal{U}_{\D}^{r,s})^*\quad(k\to \infty)\,,\label{eq:8.3.b}
		\end{align}
		where, for every $n\in \mathbb{N}$, the mapping $	(\mathrm{id}_{\smash{\mathbfcal{U}_{\D}^{r,s}}})^*\colon (\bXpn)^*\to (\mathbfcal{U}_{\D}^{r,s})^*$ denotes the adjoint operator to the\vspace*{0.5mm} identity mapping $\mathrm{id}_{\smash{\mathbfcal{U}_{\D}^{r,s}}}\colon \mathbfcal{U}_{\D}^{r,s}\to \bXpn$, defined by $\mathrm{id}_{\smash{\mathbfcal{U}_{\D}^{r,s}}}\boldsymbol{u}=\boldsymbol{u}$ in $\bXpn$ for all $\boldsymbol{u}\in \mathbfcal{U}_{\D}^{r,s}$.
	\end{proposition}

	\begin{proof} Proposition \ref{3.7} yields both  $\boldsymbol{f}_n\in L^{q_n'(\cdot,\cdot)}(Q_T;\mathbb{R}^d)$, $n\in\mathbb{N}$, and $\boldsymbol{F}_n\in L^{p_n'(\cdot,\cdot)}(Q_T;\mathbb{R}^{d\times d}_{\textup{sym}})$, $n\in\mathbb{N}$, such that for every $n\in \mathbb{N}$, it holds that $\boldsymbol{f}_n^*=\mathbfcal{J}_{\D}(\boldsymbol{f}_n,\boldsymbol{F}_n)$ in $(\bXpn)^*$ and 
		\begin{align}
			\frac{1}{2}\,\|\boldsymbol{f}_n^*\|_{(\bXpn)^*}\leq\|\boldsymbol{f}_n\|_{q_n'(\cdot,\cdot),Q_T}+\|\boldsymbol{F}_n\|_{p_n'(\cdot,\cdot),Q_T}\leq 2\,\|\boldsymbol{f}_n^*\|_{(\bXpn)^*}\,.
		\end{align}
        Thus, \hspace*{-0.15mm}Lemma \hspace*{-0.15mm}\ref{8.2.2} \hspace*{-0.15mm}yields \hspace*{-0.15mm}subsequences \hspace*{-0.15mm}$\boldsymbol{f}_{n_k}\hspace*{-0.15em}\in\hspace*{-0.15em} L^{q_{n_k}'(\cdot,\cdot)}(Q_T;\mathbb{R}^d)$, $k\hspace*{-0.15em}\in\hspace*{-0.15em} \mathbb{N}$, \hspace*{-0.15mm}and \hspace*{-0.15mm}$\boldsymbol{F}_{n_k}\hspace*{-0.15em}\in\hspace*{-0.15em} L^{p_{n_k}'(\cdot,\cdot)}(Q_T;\mathbb{R}^{d\times d}_{\textup{sym}})$,~\hspace*{-0.15mm}${k\hspace*{-0.15em}\in\hspace*{-0.15em} \mathbb{N}}$, such that~for~every  $r,s\in \mathcal{P}^\infty(Q_T)$ with $r\succ q$ in $Q_T$ and $s\succ p$ in $Q_T$, it holds that
		\begin{align*}
            \begin{aligned}
		          \boldsymbol{f}_{n_k}&\rightharpoonup \boldsymbol{f}&&\quad\text{ in }L^{r'(\cdot,\cdot)}(Q_T;\mathbb{R}^d)&&\quad(k\to \infty)\,,\\
		          \boldsymbol{F}_{n_k}&\rightharpoonup \boldsymbol{F}&&\quad\text{ in }L^{s'(\cdot,\cdot)}(Q_T;\mathbb{R}^{d\times d}_{\textup{sym}})&&\quad(k\to \infty)\,.
            \end{aligned}
		\end{align*}
		Using that  $\mathbfcal{J}_{\D}\colon L^{r'(\cdot,\cdot)}(Q_T;\mathbb{R}^d)\times L^{s'(\cdot,\cdot)}(Q_T;\mathbb{R}^{d\times d}_{\textup{sym}})\to (\mathbfcal{U}_{\D}^{r,s})^*$ is weakly continuous (\textit{cf.}~\mbox{Proposition}~\ref{3.7}), setting $\boldsymbol{f}^*\coloneqq \mathbfcal{J}_{\D}(\boldsymbol{f},\boldsymbol{F})\in (\bXp)^*$, we conclude that \eqref{eq:8.3.b} applies.
	\end{proof}
	
	\newpage
	We are able to prove
	an analogue of a parabolic embedding for $\bXp\cap L^\infty(I;Y)$ (\textit{cf}.\ \cite[Prop.~3.8]{alex-book}), which holds for uniformly continuous exponents, now in the  context of $\bXpn\cap L^\infty(I;Y)$,~$n\in \mathbb{N}$.

	\begin{proposition}[Parabolic embedding for $\bXpn\cap L^\infty(I;Y)$, $n\in \mathbb{N}$]\label{prop:non_conform_poincare}
	For every $ \varepsilon\in (0,(p^-)_*-1]$, there exists $n_0\coloneqq n_0(\varepsilon)\in \mathbb{N}$ and a constant $c_\varepsilon\coloneqq c_\varepsilon(p,\Omega)>0$ such that for every $n\in  \mathbb{N}$ with $n\ge n_0$ and $\boldsymbol{u}_n\in  \bXpn\cap L^\infty(I;Y)$, it holds that $\boldsymbol{u}_n\in  L^{(p_n)_*(\cdot,\cdot)-\varepsilon}(Q_T;\mathbb{R}^d)$ with
		\begin{align}\label{eq:non_conform_poincare}
		\|\boldsymbol{u}_n\|_{(p_n)_*(\cdot,\cdot)-\varepsilon,Q_T}
		\leq c_\varepsilon\,\big(\|\mathbf{D}_x\boldsymbol{u}_n\|_{p_n(\cdot,\cdot),Q_T}+\|\boldsymbol{u}_n\|_{L^\infty(I;Y)}\big)\,,
		\end{align}
        where 
        $r_\ast\coloneqq (r)_\ast\coloneqq r\frac{d+2}{d}$ for all $r\in [1,+\infty)$.
	\end{proposition}

    Proposition \ref{prop:non_conform_poincare}  is an immediate consequence of the following point-wise estimate in~terms~of~modulars.\enlargethispage{10mm}

	\begin{lemma}\label{8.2.4.1}
		For every~${\varepsilon\in (0,(p^-)_*-1]}$, there exists $n_0\coloneqq n_0(\varepsilon)\in \mathbb{N}$ and constants $c_\varepsilon\coloneqq c(\varepsilon,p,\Omega)$, $\gamma_{\varepsilon}\coloneqq \gamma_\varepsilon(\varepsilon,p)>0$ such that for every $n\in \mathbb{N}$ with $n\ge  n_0$ and $\mathbf{u}_n\in  U^{q_n,p_n}_{\D}(t)\cap Y$, it~holds~that~$\mathbf{u}_n\in L^{(p_n)_*(t,\cdot)-\varepsilon}(\Omega;\mathbb{R}^d)$ with 
		\begin{align}\label{eq:8.2.4.1}
		\rho_{(p_n)_*(t,\cdot)-\varepsilon,\Omega}(\mathbf{u}_n)
		\leq c_\varepsilon\,\big(1+\rho_{p_n(t,\cdot),\Omega}(\mathbf{D}_x\mathbf{u}_n)+\|\mathbf{u}_n\|_{2,\Omega}^{\gamma_\varepsilon}\big)\,.
		\end{align}
	\end{lemma}

	\begin{proof}
		Based on  $p_n\to p$ in $L^\infty(Q_T)$ $(n\to\infty)$ and $(r\mapsto(r)_*)\in C^0([1,+\infty))$,~for~every~$\varepsilon\in (0,(p^-)_*-1]$, there exists $n_0\coloneqq n_0(\varepsilon)\in \mathbb{N}$ such that $\|(p_n)_*-p_*\|_{\infty,Q_T}+\|p_n-p\|_{\infty,Q_T}<\frac{\varepsilon}{4}$ for all $n\in \mathbb{N}$ with $n\ge n_0$, \textit{i.e.}, $p_*-\frac{\varepsilon}{4}\leq (p_n)_*\leq p_*+\frac{\varepsilon}{4}$ a.e.\ in $Q_T$ and $p-\frac{\varepsilon}{4}\leq p_n\leq p+\frac{\varepsilon}{4}$ a.e.\ in $Q_T$~for~all~${n\in \mathbb{N}}$~with~$n\ge n_0$. On the other hand, due to $\frac{\varepsilon}{4}\in (0,(p^--\frac{\varepsilon}{4})_*-1]$, \cite[Lem.\  3.5]{alex-book} yields constants ${c_\varepsilon\coloneqq c(\varepsilon,p,\Omega),\gamma_\varepsilon\coloneqq \gamma(\varepsilon,p)>0}$  such that for every $\mathbf{u}\in U^{q,p-\frac{\varepsilon}{4}}_{\D}(t)\cap Y$, it holds that $\mathbf{u}\in L^{(p-\frac{\varepsilon}{4})_*(t,\cdot)-\frac{\varepsilon}{4}}(\Omega;\mathbb{R}^d)$ with 
		\begin{align}
		\rho_{(p-\frac{\varepsilon}{4})_*(t,\cdot)-\frac{\varepsilon}{4},\Omega}(\mathbf{u})\leq c_\varepsilon\,\big(1+\rho_{p(t,\cdot)-\frac{\varepsilon}{4},\Omega}(\mathbf{D}_x\mathbf{u})+\|\mathbf{u}\|_{2,\Omega}^{\gamma_\varepsilon}\big)\,.\label{8.2.4.1.a}
		\end{align}
		Since $(p-\frac{\varepsilon}{4})_*-\frac{\varepsilon}{4}\ge p_*-\frac{3\varepsilon}{4}\ge (p_n)_*-\varepsilon$ in $Q_T$, we infer from \eqref{8.2.4.1.a} for every $n\in \mathbb{N}$ with $n\ge n_0$  and $\mathbf{u}_n\in U^{q_n,p_n}_{\D}(t)\cap Y$ that $\mathbf{u}_n\in L^{(p_n)_*(t,\cdot)-\varepsilon}(\Omega;\mathbb{R}^d)$ with
		\begin{align*}
		\rho_{(p_n)_*(t,\cdot)-\varepsilon,\Omega}(\mathbf{u}_n)&\leq c_\varepsilon\,\big(1+ \rho_{(p-\frac{\varepsilon}{4})_*(t,\cdot)-\frac{\varepsilon}{4},\Omega}(\mathbf{u}_n)\big)\\&\leq c_\varepsilon\,\big(1+\rho_{p(t,\cdot)-\frac{\varepsilon}{4},\Omega}(\mathbf{D}_x\mathbf{u}_n)+\|\mathbf{u}_n\|_{2,\Omega}^{\gamma_\varepsilon}\big)
		\\&\leq c_\varepsilon\, \big(1+\rho_{p_n(t,\cdot),\Omega}(\mathbf{D}_x\mathbf{u}_n)+\|\mathbf{u}_n\|_{2,\Omega}^{\gamma_\varepsilon}\big)\,,
		\end{align*}
		which is the claimed point-wise estimate in terms of modulars \eqref{eq:8.2.4.1}.
	\end{proof}
	
	\begin{proof}[Proof (of Proposition \ref{prop:non_conform_poincare}).]
		Follows along the lines of the proof of \cite[Prop.\ 3.6]{alex-book} resorting to Lemma~\ref{8.2.4.1}, up to minor adjustments.
	\end{proof}
	
	Having the approximative Poincar\'e inequality (\textit{cf}.\ Proposition \ref{prop:non_conform_poincare}) at hand, we are in~the~position~to derive the following approximative parabolic compactness principle.

    \begin{proposition}[Compactness principle for $\bXpn\cap L^\infty(I;Y) $, $n\in \mathbb{N}$]\label{prop:non_conform_landes}
		Let  $\boldsymbol{u}_n\in \bXpn\cap L^\infty(I;Y) $, $n\in \mathbb{N}$, be a sequence such that
        \begin{align}
		\hphantom{xx.}\sup_{n\in\mathbb{N}}{\|\boldsymbol{u}_n\|_{\smash{\bXpn}}}<\infty\,,\label{prop:non_conform_landes.1}
		\end{align}
        and
		\begin{alignat}{3}
		\boldsymbol{u}_n&\overset{\ast}{\rightharpoondown}\boldsymbol{u}&&\quad\text{ in }L^\infty(I;Y) &&\quad(n\to\infty)\,,\label{prop:non_conform_landes.2}\\
		\boldsymbol{u}_n(t)&\rightharpoonup\boldsymbol{u}(t)&&\quad\text{ in }Y&&\quad (n\to \infty)\quad\text{ for a.e.\  }t\in I\,.\label{prop:non_conform_landes.3}
		\end{alignat}
		Then, it holds that $\boldsymbol{u}_n\to\boldsymbol{u}$ in $L^{\max\{2,p_*(\cdot,\cdot)\}-\varepsilon}(Q_T;\mathbb{R}^d)$ $(n\to\infty)$ for all $\varepsilon\in (0,(p^-)_*-1]$.
	\end{proposition}
	
	\begin{proof}
    Owing to $\inf_{n\in\mathbb{N}}{p_n^-}>1$, $Y\hookrightarrow L^1(\Omega;\mathbb{R}^d)$, and \eqref{prop:non_conform_landes.1}--\eqref{prop:non_conform_landes.3}, the sequence  $(\boldsymbol{u}_n)_{n\in \mathbb{N}}$ is bounded in  $L^{p^-}(I,W^{1,\smash{\inf_{n\in\mathbb{N}}{p_n^-}}}_0(\Omega;\mathbb{R}^d))\cap L^\infty(I;L^1(\Omega;\mathbb{R}^d)) $ and 
    $\boldsymbol{u}_n(t)\hspace*{-0.15em}\rightharpoonup\hspace*{-0.15em}\boldsymbol{u}(t)$ \hspace*{-0.1mm}in \hspace*{-0.1mm}$L^1(\Omega;\mathbb{R}^d)$ \hspace*{-0.1mm}$(n\hspace*{-0.15em}\to\hspace*{-0.15em} \infty)$~\hspace*{-0.1mm}for~\hspace*{-0.1mm}a.e.~\hspace*{-0.1mm}${t\hspace*{-0.15em}\in\hspace*{-0.15em} I}$.~\hspace*{-0.1mm}Thus, the Landes--Mustonen compactness principle (\textit{cf}.\ \cite[Prop.\  1]{landes-87}) yields a subsequence such that $\boldsymbol{u}_{n_k}\to\boldsymbol{u}$  $(k\to \infty)$ a.e.\  in $Q_T$. On the other hand, 
	 Proposition~\ref{prop:non_conform_poincare} implies that $(\boldsymbol{u}_{n_k})_{k\in \mathbb{N}}$ is $L^{\max\{2,p_*(\cdot,\cdot)\}-\varepsilon}(Q_T)$-uniformly integrable for all $\varepsilon\in (0,(p^-)_*-1]$.~Therefore,~Vitali's~convergence theorem for variable Lebesgue spaces (\textit{cf.}~\cite[Prop.\ 2.5]{alex-book}) and the standard convergence principle 
	(\textit{cf.}~\cite[Kap.\ I, Lem.\ 5.4]{GGZ}) yield the assertion.
	\end{proof}
    \newpage
    
    \section{Space and time discretization}\label{sec:discrete_ptxNavierStokes}
   \hspace*{5mm}In this section, we introduce the discrete spaces and projection operators employed later on in the fully-discrete Rothe--Galerkin approximation of a generalized evolution equation (\textit{cf}.\ Definition \ref{def:generalized_evolution_equation}).\enlargethispage{5mm}
   
   \subsection{Space discretization}
   
   \subsubsection{Triangulations}
   
  	\hspace*{5mm}We denote by $\{\mathcal{T}_h\}_{h>0}$ a family of uniformly shape regular (\textit{cf}.\ \cite{EG21}) triangulations of $\Omega\subseteq \mathbb{R}^d$,~$d\ge 2$, consisting of $d$-dimensional simplices. 
  Here, $h>0$ refers to the \textit{maximal~mesh-size}, \textit{i.e.}, if   $h_T\coloneqq  \textup{diam}(T)$ for all $T\in \mathcal{T}_h$,~then~${h 
  	\coloneqq \max_{T\in \mathcal{T}_h}{h_T}
  }$.
   
   \subsubsection{Finite element spaces and projectors}
   
   \hspace{5mm}Given
   $m \in \mathbb N_0$ and $h>0$, we denote by $\mathbb{P}^m(\mathcal{T}_h)$ the space of possibly discontinuous scalar~functions that are polynomials of degree at most $m$ on each simplex $T\in \mathcal{T}_h$, and~set~${\mathbb{P}^m_c(\mathcal{T}_h) \coloneqq  \mathbb{P}^m(\mathcal{T}_h)\cap C^0(\overline{\Omega})}$.
   Then, for the remainder of the paper, for given $m\in \mathbb{N}$ and $\ell \in \mathbb{N}\cup\{0\}$, we~denote~by
   \begin{align}
   	\begin{aligned}
   		V_h&\subseteq (\mathbb{P}^m_c(\mathcal{T}_h))^d\,, &&\,\Vo_h\coloneqq V_h\cap W^{1,1}_0(\Omega;\mathbb{R}^d)\,,\\
   		Q_h&\subseteq \mathbb{P}^{\ell}(\mathcal{T}_h)\,, &&\Qo_h \coloneqq 
   		Q_h\cap L^1_0(\Omega)\,,
   	\end{aligned}
   \end{align}
   where $L^1_0(\Omega)\coloneqq \{z\in L^1(\Omega)\mid \langle z\rangle_{\Omega}=0\}$, 
   conforming
   finite element spaces such that the following two assumptions are satisfied.
   
      \begin{assumption}[Projection operator $\Pi_h^Q$]
    	\label{ass:PiY}
    	We assume that  for every $r\in [1,\infty)$ and $z\in L^r(\Omega)$,  we have that
    	\begin{align}
    		\inf_{z_h\in Q_h}{\|z -z_h\|_{r,\Omega}}\to 0\quad (h\to 0)\,,
    	\end{align}
    	and there exists a linear projection operator
    	$\Pi_h^Q\colon L^1(\Omega) \to Q_h$, that is $ \Pi_h^Qz_h=z_h$ for all $z_h\in Q_h$, which for every $r\in (1,\infty)$ is \textup{globally $L^r(\Omega)$-stable}, \textit{i.e.}, there exists a constant $c>0$ such that for every $z\in L^r(\Omega)$, it holds that
    	\begin{align}
    		\label{eq:PiYstab}
    		\|\Pi_h^Q z\|_{r,\Omega} \leq c\, \| z\|_{r,\Omega}\,.
    	\end{align}
    \end{assumption}
    
    \begin{assumption}[Projection operator $\Pi_h^V$]\label{ass:proj-div}
    	We assume that $(\mathbb{P}^1_c(\mathcal{T}_h))^d \subseteq V_h$ and that there
    	exists a linear projection operator $\Pi_h^V\colon W^{1,1}(\Omega;\mathbb{R}^d)\to V_h$, that is $ \Pi_h^V\mathbf{z}_h=\mathbf{z}_h$ for all $\mathbf{z}_h\in V_h$,  which has the following properties: 
    	\begin{enumerate}[noitemsep,topsep=2pt,leftmargin=!,labelwidth=\widthof{(ii)},font=\normalfont\itshape]
    		\item[(i)] \textup{Preservation of divergence in the $Q_h^*$-sense:} For every $\mathbf{z} \in W^{1,1}(\Omega;\mathbb{R}^d)$ and  $ z_h \in Q_h$, it holds that
    		\begin{align}
    			\label{eq:div_preserving}
    			(\mathrm{div}_x \mathbf{z},z_h)_\Omega &= (\mathrm{div}_x\Pi_h^V
    			\mathbf{z},z_h)_\Omega \,;
    		\end{align} 
    		\item[(ii)] \textup{Preservation of homogeneous Dirichlet boundary values:} $\Pi_h^V(W^{1,1}_0(\Omega;\mathbb{R}^d))\subseteq \Vo_h$;
    		\item[(iii)] \textup{Global $W^{1,r}_0(\Omega;\mathbb{R}^d)$-stability:} For every $r\in (1,\infty)$, there exists a constant $c>0$ such that
    		for every $\mathbf{z} \in W^{1,r}_0(\Omega;\mathbb{R}^d)$, it holds that
    		\begin{align}
    			\label{eq:Pidivcont}
    			\|\nabla\Pi_h^V\mathbf{z}\|_{r,\Omega}\leq c\,\|\nabla\mathbf{z}\|_{r,\Omega}\,.
    		\end{align}
    	\end{enumerate}
    \end{assumption}
    
    \begin{lemma}[Approximation properties of $\Pi_h^V$ and $\Pi_h^Q$]\label{rmk:proj} The following statements apply:
    	\begin{itemize}[noitemsep,topsep=2pt,leftmargin=!,labelwidth=\widthof{(ii)},font=\normalfont\itshape]
    		\item[(i)] For every $r\in [1,\infty)$ and $z\in L^r(\Omega)$, it holds that 
    		\begin{align*}
    			\Pi_h^Q z \to z\quad \text{ in }L^r(\Omega)\quad  (h\to 0)\,;
    		\end{align*}
    		\item[(ii)] For every $r\in [1,\infty)$ and $\mathbf{z} \in W^{1,r}_0(\Omega;\mathbb{R}^d)$, it holds that 
    		\begin{align*}
    			\Pi_h^V\mathbf{z}  \to \mathbf{z}\quad \text{ in }W^{1,r}_0(\Omega;\mathbb{R}^d)\quad (h\to 0)\,.
    		\end{align*} 
    	\end{itemize}
    \end{lemma}
    
    \begin{proof}
    	See \cite[Prop.\ 7 \& (3.5)]{DKS13b}.
    \end{proof}

    \newpage
    Next, we present a list of common mixed finite element spaces $\{V_h\}_{h>0}$ and $\{Q_h\}_{h>0}$ with projectors $\{\Pi_h^V\}_{h>0}$ and $\{\Pi_h^Q\}_{h>0}$ on regular triangulations $\{\mathcal{T}_h\}_{h>0}$ satisfying  Assumption~\ref{ass:PiY}~and~Assumption~\ref{ass:proj-div}, respectively; for a detailed presentation, we recommend 
    the textbook \cite{BBF13}.
    
    \begin{remark}\label{FEM.Q}
    	The following discrete spaces and projectors satisfy Assumption~\ref{ass:PiY}:
    	\begin{description}[noitemsep,topsep=2pt,leftmargin=!,labelwidth=\widthof{(iii)},font=\normalfont\itshape]
    		\item[(i)] If $Q_h= \mathbb{P}^{\ell}(\mathcal{T}_h)$ for some $\ell\in \mathbb{N}\cup \{0\}$, then $\Pi_h^Q$ can be chosen as (local) $L^2$-projection operator or, more generally, as a Cl\'ement type quasi-interpolation operator (\textit{cf}.\ \cite{clement}).
    		
    		\item[(ii)] If \hspace*{-0.1mm}$Q_h\hspace*{-0.1em}=\hspace*{-0.1em}\mathbb{P}^{\ell}_c(\mathcal{T}_h)$ \hspace*{-0.1mm}for \hspace*{-0.1mm}some \hspace*{-0.1mm}$\ell\hspace*{-0.1em}\in  \hspace*{-0.1em} \mathbb{N}$, \hspace*{-0.1mm}then \hspace*{-0.1mm}$\Pi_h^Q$ \hspace*{-0.1mm}can \hspace*{-0.1mm}be \hspace*{-0.1mm}chosen \hspace*{-0.1mm}as \hspace*{-0.1mm}a \hspace*{-0.1mm}Cl\'ement \hspace*{-0.1mm}type \hspace*{-0.1mm}quasi-interpolation~\hspace*{-0.1mm}operator (\textit{cf}.\ \cite[Sec.\ 1.6.1]{EG04}).
    		
    	\end{description}
    \end{remark}
    
    \begin{remark}\label{FEM.V}
    	The following discrete spaces  and projectors satisfy Assumption~\ref{ass:proj-div}:
    	\begin{description}[noitemsep,topsep=2pt,leftmargin=!,labelwidth=\widthof{(iii)},font=\normalfont\itshape]
    		\item[(i)] The \textup{MINI element} for $d\in \{2,3\}$, \textit{i.e.}, $V_h=(\mathbb{P}^1_c(\mathcal{T}_h))^d\bigoplus(\mathbb{B}(\mathcal{T}_h))^d$, where $\mathbb{B}(\mathcal{T}_h)$ is the bubble function space, and $Q_h=\mathbb{P}^1_c(\mathcal{T}_h)$, introduced in \cite{ABF84} for $d=2$; see also \cite[Chap.\ II.4.1]{GR86} and \cite[Sec.\ 8.4.2, 8.7.1]{BBF13}. An operator  $\Pi_h^V$ satisfying Assumption \ref{ass:proj-div} is given in~\cite[Appx.~A.1]{bdr-phi-stokes}~or~\smash{\cite[Lem.~4.5]{GL01}};
    		
    			\item[(ii)] The \textup{$\mathbb{P}^2$-$\mathbb{P}^0$-element} for $d=2$, \textit{i.e.}, $V_h=(\mathbb{P}^2_c(\mathcal{T}_h))^2$ and $Q_h=\mathbb{P}^0(\mathcal{T}_h)$; see, \textit{e.g.}, \cite[Sec 8.4.3]{BBF13}, where the projection operator $\Pi_h^V$ is given and Assumption \ref{ass:proj-div}(i) is shown. The Assumption~\ref{ass:proj-div}(ii) can be proved similarly as for the MINI element; see, \textit{e.g.}, \cite[Appx.\ A.1]{bdr-phi-stokes} and \cite[p.\ 990]{DKS13};
    		
    		\item[(iii)] The \textup{Taylor--Hood element} for $d\in\{2,3\}$, \textit{i.e.}, $V_h=(\mathbb{P}^2_c(\mathcal{T}_h))^d$ and $Q_h=\mathbb{P}^1_c(\mathcal{T}_h)$, introduced~in~\cite{TH73} for $d=2$; see  also \cite[Chap.\ II.4.2]{GR86}, and its generalizations; see, \textit{e.g.}, \cite[Sec.\ 8.8.2]{BBF13}. An operator $\Pi_h^V$ satisfying Assumption \ref{ass:proj-div} is given in \cite[Thm.\ 3.1, 32]{GS03} or \cite{DST2021};
    		
    		\item[(iv)] The \textup{conforming Crouzeix--Raviart element} for $d=2$, \textit{i.e.}, $V_h=(\mathbb{P}^2_c(\mathcal{T}_h))^2\bigoplus(\mathbb{B}(\mathcal{T}_h))^2$~and~$Q_h=\mathbb{P}^1(\mathcal{T}_h)$, introduced in \cite{CR73}; see also \cite[Ex.\ 8.6.1]{BBF13}. An operator $\Pi_h^V$ satisfying Assumption~\ref{ass:proj-div}(i) is given in \cite[p.\ 49]{CR73} and it can be shown to satisfy Assumption \ref{ass:proj-div}(ii); see, \textit{e.g.}, \cite[Thm.\ 3.3]{GS03};
    		
    		\item[(v)] The \textup{first order Bernardi--Raugel element} for $d\in \{2,3\}$, \textit{i.e.}, 
    		$V_h=(\mathbb{P}^1_c(\mathcal{T}_h))^d\bigoplus(\mathbb{B}_{\tiny \mathscr{F}}(\mathcal{T}_h))^d$,~where $\mathbb{B}_{\tiny \mathscr{F}}(\mathcal{T}_h)$ is the facet bubble function space, and $Q_h=\mathbb{P}^0(\mathcal{T}_h)$, introduced in \cite[Sec. II]{BR85}. For $d=2$ is often  referred to as \textup{reduced $\mathbb{P}^2$-$\mathbb{P}^0$-element} or as \textup{2D SMALL element}; see,  \textit{e.g.},  \cite[Rem.\ 8.4.2]{BBF13} and \cite[Chap.\ II.2.1]{GR86}.  An operator $\Pi_h^V$ satisfying Assumption \ref{ass:proj-div} is given in
    		\cite[Sec.\ II.]{BR85};
    		
    		\item[(vi)] The \hspace*{-0.15mm}{\textup{second \hspace*{-0.15mm}order \hspace*{-0.15mm}Bernardi--Raugel \hspace*{-0.15mm}element}} \hspace*{-0.15mm}for \hspace*{-0.15mm}$d\hspace*{-0.15em}=\hspace*{-0.15em}3$, \hspace*{-0.15mm}introduced \hspace*{-0.1mm}in \hspace*{-0.15mm}\cite[\hspace*{-0.5mm}Sec.\ \hspace*{-0.5mm}III]{BR85}; \hspace*{-0.1mm}see \hspace*{-0.15mm}also~\hspace*{-0.15mm}\mbox{\cite[\hspace*{-0.5mm}Ex.~\hspace*{-0.5mm}8.7.2]{BBF13}} \hspace*{-0.1mm}and \hspace*{-0.1mm}\cite[Chap.\ \hspace*{-0.1mm}II.2.3]{GR86}. \hspace*{-0.1mm}An \hspace*{-0.1mm}operator \hspace*{-0.1mm}$\Pi_h^V$ \hspace*{-0.1mm}satisfying \hspace*{-0.1mm}Assumption \hspace*{-0.1mm}\ref{ass:proj-div} \hspace*{-0.1mm}is \hspace*{-0.1mm}given~\hspace*{-0.1mm}in~\hspace*{-0.1mm}\mbox{\cite[Sec.~\hspace*{-0.1mm}III.3]{BR85}}.
    	\end{description}
    \end{remark}
    
    We define the space of \textit{discretely divergence-free vector fields in $\Vo_h$} by
    \begin{align*}
    	\Vo_{h,0}\coloneqq \big\{\mathbf{z}_h\in \Vo_h\mid (\mathrm{div}_x\mathbf{z}_h,z_h)_\Omega=0\text{ for all }z_h\in Q_h\big\}\,.
    \end{align*}
    Then, the following approximative weak closedness result applies for appropriately bounded sequences~of variable Bochner--Lebesgue functions that are a.e.\ in time discretely divergence-free.\enlargethispage{6mm}
    
    \begin{proposition}\label{prop:approx_closedness} 
    	For a sequence $\boldsymbol{v}_n\in \bXpn\cap L^\infty(I;Y)$, $n\in \mathbb{N}$, 
    	with $\boldsymbol{v}_n(t)\in \Vo_{h_n,0}$ for a.e.\  $t\in I$, where $(h_n)_{n\in \mathbb{N}}\subseteq (0,+\infty)$  is a sequence such that $h_n\to 0$ $(n\to \infty)$, 
    	 from 
    	\begin{align*}
    		\sup_{n\in \mathbb{N}}{\|\boldsymbol{v}_n\|_{\smash{\bXpn}}}<\infty\,, \qquad\boldsymbol{v}_n\overset{\ast}{\rightharpoondown}\boldsymbol{v}\quad\textup{ in }L^\infty(I;Y)\quad(n\to\infty)\,,
    	\end{align*}
    	it follows that $\boldsymbol{v}\in \bVp\cap L^\infty(I;H)$. 
    \end{proposition}
    
    \begin{proof}
    	Propositions \ref{prop:Xqp_weak_compact} and the uniqueness of weak and weak-$*$ limits imply the existence of a subsequence $\boldsymbol{v}_{n_k}\in \mathbfcal{U}^{q_{\smash{n_k}},p_{\smash{n_k}}}_{\D}\cap L^\infty(I;Y)$, $k\in \mathbb{N}$,  and 
    	$\boldsymbol{v}\in \bXp\cap L^\infty(I;Y)$  such that for every  $r,s\in \mathcal{P}^\infty(Q_T)$ with $r\prec q$ in $Q_T$ and $s\prec p$ in $Q_T$, it holds that
    	\begin{align}
    		\boldsymbol{v}_{n_k}\rightharpoonup\boldsymbol{v}\quad\textup{ in }\mathbfcal{U}^{r,s}_{\D}\quad(k\to\infty )\,.\label{prop:approx_closedness.1}
    	\end{align}
    	Next, let $ z\in C_c^\infty(\Omega)$ and $ \varphi\in C_c^\infty(I)$. Due to Lemma \ref{rmk:proj}(ii), the sequence  $z_{h_n}\coloneqq \Pi_{h_n}^Qz\in Q_{h_n}$, $n\in\mathbb{N}$, satisfies $z_{h_n}\to z$ in $L^{\smash{(s^-)'}}(\Omega)$ $(n\to\infty)$. Since, on the other hand, it holds that $(\mathrm{div}_x\boldsymbol{v}_n(t),z_{h_n})_{\Omega}=0$ for a.e.\  $t\in I$ and all $n\in \mathbb{N}$, due to $\boldsymbol{v}_n(t)\in \Vo_{h_n,0}$ for a.e.\  $t\in I$ and all $n\in \mathbb{N}$, using \eqref{prop:approx_closedness.1}, we~deduce~that
    	\begin{align}
    		(\mathrm{div}_x\boldsymbol{v},\varphi z)_{Q_T}=\lim_{k\to\infty}{	(\mathrm{div}_x\boldsymbol{v}_{n_k},\varphi z_{h_{n_k}})_{Q_T}}=0\,.\label{prop:approx_closedness.2}
    	\end{align}
    	The fundamental theorem in the calculus of variations applied to \eqref{prop:approx_closedness.2} yields that
    	$\boldsymbol{v}\in \bVp\cap L^\infty(I;H)$.
    \end{proof}
    \newpage
    
    In \cite{CHP10}, the compactness properties of an approximative sequence relied on the Simon compactness principle (\textit{cf.} \cite[Lemmata 2.3, 2.4]{CHP10}). This, however, seems to be possible only  if the power-law index is time-independent.
    Since we allow for a time-dependence of the power-law index, we instead resort to Proposition \hspace*{-0.1mm}\ref{prop:non_conform_landes}, \hspace*{-0.1mm}which, \hspace*{-0.1mm}in \hspace*{-0.1mm}turn, \hspace*{-0.1mm}is \hspace*{-0.1mm}based \hspace*{-0.1mm}on 
    \hspace*{-0.1mm}the \hspace*{-0.1mm}Landes--Mustonen~\hspace*{-0.1mm}compactness~\hspace*{-0.1mm}principle~\hspace*{-0.1mm}(\textit{cf}.~\hspace*{-0.1mm}\mbox{\cite[Prop.~\hspace*{-0.1mm}1]{landes-87}}). To be able to extract the a.e.\ in time weak convergence in $Y$ (\textit{cf}.\ \eqref{prop:non_conform_landes.3}) from the approximative~scheme, the following result turns out to be crucial.\vspace*{-0.5mm}\enlargethispage{12mm}

    \begin{lemma}\label{lem:H_perp}
    	Let $r\in (1,\infty)$ and $(\tau_n)_{n\in \mathbb{N}}\subseteq (0,T)$ a sequence such that $\tau_n\to 0$ $(n\to \infty)$. Moreover, let $\boldsymbol{v}_n\in L^\infty (I;\Vo_{h_n,0})$, $n\in \mathbb{N}$,  be a sequence such that\vspace*{-0.25mm}
    	\begin{align}
    		\sup_{n\in \mathbb{N}}{\big[\|\boldsymbol{v}_n\|_{L^\infty(I;Y)}+\|\mathrm{div}_x \boldsymbol{v}_n\|_{L^1((\tau_n,T);L^r(\Omega))}\big]}<\infty\,.\label{lem:H_perp.1}
    	\end{align}
    	Then, there exists a subsequence $\boldsymbol{v}_{n_k}\in L^\infty(I;\Vo_{h_{n_k},0})$, $k\in \mathbb{N}$, and a (Lebesgue) measurable set $I_{\perp}\subseteq I $ with $\vert I\setminus I_{\perp}\vert =0$ such that\vspace*{-0.5mm}
    	\begin{align}
    		\smash{P_{H^{\perp}}(\boldsymbol{v}_{n_k}(t))\rightharpoonup \mathbf{0}\quad\text{ in }H^{\perp}\quad(n\to \infty)\quad\text{ for all }t\in I_{\perp}\,,}\label{lem:H_perp.2}
    	\end{align}
    	where $\smash{P_{H^{\perp}}}\colon \hspace*{-0.15em}Y\hspace*{-0.15em}\to\hspace*{-0.15em} \smash{H^{\perp}}$ denotes the orthogonal projection from $Y$ onto $\smash{H^{\perp}}$, the orthogonal~\mbox{complement}~of~$H$.
    \end{lemma}
    
    \begin{proof}
    	First, note that $\smash{H^\perp}=\{\nabla_{\! x}    z\mid z\in W^{1,2}(\Omega)\}$. Therefore, since $W^{1,2}(\Omega)$ is separable~and~$C^\infty(\overline{\Omega})$ is dense in $W^{1,2}(\Omega)$, there exists a countable set $\{\psi^{i}\hspace*{-0.1em}\in \hspace*{-0.1em} C^\infty(\overline{\Omega})\mid i\hspace*{-0.1em} \in\hspace*{-0.1em}  \mathbb{N}\}$ such that ${\{\nabla_{\! x}   \psi^{i}\hspace*{-0.1em} \in\hspace*{-0.1em}   C^\infty(\overline{\Omega};\mathbb{R}^d)\mid i\hspace*{-0.1em} \in\hspace*{-0.1em}  \mathbb{N}\}}$ lies densely in $H^\perp$. For every $i\hspace*{-0.1em}\in \hspace*{-0.1em}\mathbb{N}$, set 
    	$\psi_{h_n}^{i}\hspace*{-0.1em}\coloneqq \hspace*{-0.1em}\Pi_{h_n}^Q \psi^{i}\hspace*{-0.1em}\in\hspace*{-0.1em}  Q_{h_n}$, $n\hspace*{-0.1em}\in \hspace*{-0.1em}\mathbb{N}$.
    	Then, for every $\ell \hspace*{-0.1em}\in  \hspace*{-0.1em}\mathbb{N}$, using~\eqref{lem:H_perp.1}~and Lemma \ref{rmk:proj}(i), we~find~that\vspace*{-0.25mm}
    	\begin{align}\label{eq:H_perp.1} 
    		\begin{aligned}
    			\int_{\tau_n}^T{\sum_{i=1}^{\ell}{\vert(\boldsymbol{v}_n(s),\nabla_{\! x}    \psi^{i})_{\Omega}\vert}\,\mathrm{d}s}
    			&= 
    			 \int_{\tau_n}^T{\sum_{i=1}^{\ell}{\vert(\mathrm{div}_x \boldsymbol{v}_n(s),\psi^{i}-\psi_{h_n}^{i})_{\Omega}\vert}\,\mathrm{d}s}
    			\\[-1mm]&\leq \|\mathrm{div}_x \boldsymbol{v}_n\|_{L^1((\tau_n,T);L^r(\Omega))}T^{\frac{1}{r'}}\sum_{i=1}^{\ell}{\|\psi^{i}-\psi_{h_n}^{i}\|_{r',\Omega}}
    		 \to 0\quad (n\to\infty)\,.
    		\end{aligned}
    	\end{align}
    	
    	Next, we apply a diagonal sequence argumentation to derive from \eqref{eq:H_perp.1} that there exists~a~sub-sequence  $\boldsymbol{v}_{n_k}\in \smash{L^\infty(I;\Vo_{h_{n_k},0})}$, $k\in \mathbb{N}$, and a (Lebesgue) measurable set  $I_{\perp}\subseteq I $ with $\vert I\setminus I_{\perp}\vert =0$~such~that
    	\begin{align}\label{eq:H_perp.2} 
    		\smash{(\boldsymbol{v}_{\smash{n_k}}(t),\nabla_{\! x}    \psi^i)_{\Omega} \to 0\quad (k\to \infty)\quad\text{ for all }t\in I_\perp\text{ and }i\in \mathbb{N}\,.}
    	\end{align}
    	First, due to \eqref{eq:H_perp.1} for $\ell =1$, there exists a subsequence $\boldsymbol{v}_{\smash{n_k^1}}\in L^\infty(I;\Vo_{h_{\smash{n_k^1}}}(0))$, $k\in \mathbb{N}$, and a (Lebesgue) measurable set $I_\perp^1\subseteq I$ with $\vert I\setminus I_\perp^1\vert=0$~such~that
    	\begin{align*}
    		(\boldsymbol{v}_{\smash{n_k^1}}(t),\nabla_{\! x}    \psi^1)_{\Omega} \to 0\quad (k\to \infty)\quad\text{ for all }t\in I_\perp^1\,.
    	\end{align*}
    	Second, since \eqref{eq:H_perp.1} for $\ell =2$ also applies to the  subsequence $\boldsymbol{v}_{\smash{n_k^1}}\in L^\infty(I;\Vo_{h_{\smash{n_k^1}}}(0))$, $k\in \mathbb{N}$, we obtain~a further subsequence
    	$\boldsymbol{v}_{\smash{n_k^2}}\hspace*{-0.05em}\in\hspace*{-0.05em} L^\infty(I;\Vo_{h_{\smash{n_k^2}},0})$, $k\in \mathbb{N}$, where $(n_k^2)_{k\in \mathbb{N}}\hspace*{-0.05em}\subseteq \hspace*{-0.05em}(n_k^1)_{k\in \mathbb{N}}$,
    	and~a~(Lebesgue)~\mbox{measurable} set  $I_\perp^2\subseteq I$ with $\vert I\setminus I_\perp^2\vert=0$ such that 
    	\begin{align*}
    		\smash{	(\boldsymbol{v}_{\smash{n_k^2}}(t),\nabla_{\! x}    \psi^i)_{\Omega} \to 0\quad (k\to \infty)\quad\text{ for all }t\in I_\perp^2\text{ and }i=1,2\,.}
    	\end{align*}
    	Without loss of generality, we may assume that $\smash{I_\perp^2\subseteq I_\perp^1}$. Otherwise, we can just replace $\smash{I_\perp^2}$ by $\smash{I_\perp^2\cap I_\perp^1}$. 
    	Iterating this procedure, for every $\ell\in \mathbb{N}$, we obtain a 
    	subsequence $\boldsymbol{v}_{\smash{n_k^{\ell}}}\in L^\infty(I;\Vo_{h_{\smash{n_k^{\ell}}},0})$, $k\in \mathbb{N}$,  where $\smash{(n_k^{\ell})_{k\in \mathbb{N}}\subseteq (n_k^{\ell-1})_{k\in \mathbb{N}}}$,
    	and a (Lebesgue) measurable set $\smash{I_\perp^{\ell}\subseteq I_\perp^{\ell-1}\subseteq I}$ with $\smash{\vert I\setminus I_\perp^{\ell}\vert=0}$ such that 
    	\begin{align*}
    		\smash{	(\boldsymbol{v}_{\smash{n_k^{\ell}}}(t),\nabla_{\! x}    \psi^i)_{\Omega} \to 0\quad (k\to \infty)\quad\text{ for all }t\in I_\perp^{\ell}\text{ and }i=1,\ldots,\ell\,.}
    	\end{align*}
    	Thus, if we set $I_\perp\coloneqq \bigcap_{\ell\in \mathbb{N}}{I_\perp^{\ell}}$ and $n_k\coloneqq n_k^k$ for all $k\in \mathbb{N}$, then $\vert I\setminus I_\perp\vert=0$ and for every $i\in \mathbb{N}$~and~$t\in I_\perp$, due to $t\in I_\perp^{i}$ and since $(n_k)_{k\in \mathbb{N};k\ge i}\subseteq (n_k^{i})_{k\in \mathbb{N}}$, it holds that
    	\begin{align*}
    		\lim_{k\to \infty}{(\boldsymbol{v}_{\smash{n_k}}(t),\nabla_{\! x}    \psi^{i})_{\Omega}}= \lim_{k\to \infty}{(\boldsymbol{v}_{\smash{n_k^{i}}}(t),\nabla_{\! x}    \psi^{i})_{\Omega}}= 0\,.
    	\end{align*}
    	In other words, \eqref{eq:H_perp.2} applies.
    	Finally, since  $\{\nabla_{\! x}\psi_i\in C^\infty(\overline{\Omega};\mathbb{R}^d)\mid i\in \mathbb{N}\}$ lies densely~in~$H^\perp$~and~$(\boldsymbol{v}_{n_k})_{k\in\mathbb{N}}$ is bounded in $L^\infty(I;Y)$, from \eqref{eq:H_perp.2}, we conclude that
    	\begin{align*}
    		\begin{aligned}
    			\smash{	P_{H^{\perp}}(\boldsymbol{v}_{n_k}(t))\rightharpoonup \mathbf{0}\quad\text{ in }H^{\perp}\quad\text{ for all }t\in I_\perp\,,}
    		\end{aligned}
    	\end{align*}
    	which is the claimed a.e.\ in time weak convergence of the $H^{\perp}$-projections  \eqref{lem:H_perp.2}.
    \end{proof}

    \subsection{Time discretization}
    
	\hspace{5mm}In this section, we introduce the time discretization and prove some relevant
    properties.~In~doing~so, let $X$ be a Banach space, $K\in\mathbb{N}$,  $\tau\coloneqq \tau_K\coloneqq \frac{T}{K}$, $I_k\coloneqq \left((k-1)\tau,k\tau\right]$, $k=1,\ldots,K$, and 
    $\mathcal{I}_\tau \coloneqq \{I_k\}_{k=1,\ldots,K}$. Moreover, we denote by 
    \begin{align*}
    \mathbb{P}^0(\mathcal{I}_\tau;X)&\coloneqq \big\{\boldsymbol{u}\colon I\to X\mid \boldsymbol{u}(s)
    =\boldsymbol{u}(t)\text{ in }X\text{ for all }t,s\in I_k\,,\;k=1,\ldots,K\big\}\,,\\
    \mathbb{P}^1_c(\mathcal{I}_\tau;X)&\coloneqq \bigg\{\boldsymbol{u}\in W^{1,\infty}(I;X)\;\Big|\;  \frac{\mathrm{d}\boldsymbol{u}}{\mathrm{dt}}\in \mathbb{P}^0(\mathcal{I}_\tau;X)\bigg\}\,,
    \end{align*}
    the \textit{space of piece-wise constant functions
    with respect to $\mathcal{I}_\tau$} and \textit{space of continuous, piece-wise affine functions
    with respect to $\mathcal{I}_\tau$}. For   $(u^k)_{k=0,\ldots,K}\hspace*{-0.1em}\subseteq\hspace*{-0.1em} X$, the
    \textit{backward difference quotient} operator~is~defined~by
    \begin{align*}
    \mathrm{d}_\tau u^k\coloneqq \frac{1}{\tau}(u^k-u^{k-1})\quad\text{ in }X\quad\text{ for all } k=1,\ldots,K\,.
    \end{align*}
    In addition, we denote by
    $\overline{\boldsymbol{u}}^\tau\in \mathbb{P}^0(\mathcal{I}_\tau;X)$, 
    the \textit{piece-wise constant interpolant}, and by 
    $\widehat{\boldsymbol{u}}^\tau\in \mathbb{P}^1_c(\mathcal{I}_\tau;X)$,~the \textit{piece-wise affine interpolant}, 
    for every $t\in I_k$, $k=1,\ldots,K$, defined by
    \begin{align}\label{eq:polant}
    \overline{\boldsymbol{u}}^\tau(t)\coloneqq u^k\quad\text{and }\quad\widehat{\boldsymbol{u}}^\tau(t)
    \coloneqq \Big(\frac{t}{\tau}-(k-1)\Big)u^k+\Big(k-\frac{t}{\tau}\Big)u^{k-1}\quad\text{ in }X\,.
    \end{align}
    If $X$ is a Hilbert space with inner product $(\cdot,\cdot)_X$ and $(u^k)_{k=0,\ldots,K} \subseteq  X$, then, for every $k,\ell = 0,\ldots,K$, there holds a \textit{discrete integration-by-parts inequality}, \textit{i.e.},
    \begin{align}
    \int_{k\tau}^{\ell\tau}{\bigg( \frac{\mathrm{d}\widehat{\boldsymbol{u}}^\tau}{\mathrm{dt}}(t),
  \overline{\boldsymbol{u}}^\tau(t)\bigg)_X\,\mathrm{d}t}
    \ge \frac{1}{2}\|u^{\ell}\|_X^2-\frac{1}{2}\|u^k\|_X^2\,,\label{eq:4.2}
    \end{align}
    which follows from the identity
    $(\mathrm{d}_\tau u^k,u^k)_X
    =\frac{1}{2}\mathrm{d}_\tau\|u^k\|_X^2+\frac{\tau}{2}\|\mathrm{d}_\tau u^k\|_X^2$
    for all $k=1,\ldots,K$.
	 
 A convenient approach to discretize time-dependencies, \textit{e.g.}, of the right-hand side or time-dependent non-linear operator, is the application of the (local) $L^2$-projection~operator~from~$L^1(I;X)$~onto~$\mathbb{P}^0(\mathcal{I}_\tau;X)$:
  the \textit{(local) $L^2$-projection operator} $\Pi^{0,\mathrm{t}}_{\tau}\colon L^1(I;X)\to \mathbb{P}^0(\mathcal{I}_\tau;X)$, for every $\boldsymbol{u}\in L^1(I;X)$ is defined by
	\begin{align}\label{def:jtau}
	\Pi^{0,\mathrm{t}}_{\tau}[\boldsymbol{u}]\coloneqq \sum_{k=0}^K{\langle\boldsymbol{u}\rangle_k\chi_{I_k}}\quad\textup{ in }\mathbb{P}^0(\mathcal{I}_\tau;X)\,,\quad\text{ where }\langle\boldsymbol{u}\rangle_k\coloneqq\langle\boldsymbol{u}\rangle_{I_k}\coloneqq  \fint_{I_k}{\boldsymbol{u}(t)\,\mathrm{d}t}\in X\,,
	\end{align}
    and has the following projection, approximation, and stability properties.\enlargethispage{11mm}

	\begin{proposition}\label{rothe2}
		For $r\in [1,+\infty)$, the following statements apply:
		\begin{description}[noitemsep,topsep=2pt,leftmargin=!,labelwidth=\widthof{(iii)},font=\normalfont\itshape]
			\item[(i)] $\Pi^{0,\mathrm{t}}_{\tau}[\boldsymbol{u}_\tau]=\boldsymbol{u}_\tau$ in $\mathbb{P}^0(\mathcal{I}_\tau;X)$ for all $\boldsymbol{u}_\tau\in \mathbb{P}^0(\mathcal{I}_\tau;X)$ and $\tau>0$;
			\item[(ii)] $\Pi^{0,\mathrm{t}}_{\tau}[\boldsymbol{u}]\to\boldsymbol{u}$ in $L^r(I;X)$ $(\tau\to 0)$ for all $\boldsymbol{u}\in L^r(I;X)$, \textit{i.e.}, $\bigcup_{\tau>0}\mathbb{P}^0(\mathcal{I}_\tau;X)$~lies~densely~in~$L^r(I;X)$;
			\item[(iii)] $\|\Pi^{0,\mathrm{t}}_{\tau}[\boldsymbol{u}]\|_{L^r(I;X)}\leq \|\boldsymbol{u}\|_{L^r(I;X)}$ for all $\boldsymbol{u}\in L^r(I;X)$ and $\tau>0$.
		\end{description}
	\end{proposition}
	
	\begin{proof}
		See \cite[Rem.\ 8.15]{Rou13}.
	\end{proof}

	 Moreover, for $q,p\in \mathcal{P}^\infty(Q_T)$ with $q^-,p^->1$, let 
		$A(t)\colon  U^{q,p}_{\D}(t)\to (U^{q,p}_{\D}(t))^*$, $t\in I$, be a family~of operators with the followsing properties:
		\begin{description}[noitemsep,topsep=2pt,leftmargin=!,labelwidth=\widthof{(A.3)},font=\normalfont\itshape]
			\item[(A.1)] \hypertarget{A.1}{}  For a.e.\   $t\in I$, the operator $A(t)\colon U^{q,p}_{\D}(t)\to (U^{q,p}_{\D}(t))^*$ is  demi-continuous\footnote{For a Banach space $X$, an operator $A\colon X\to X^*$ is called \textit{demi-continuous} if for a sequence $(x_n)_{n\in \mathbb{N}}\subseteq X$, from $x_n\to x$ in $X$ $(n\to \infty)$, it follows that $Ax_n\rightharpoonup Ax$ in $X^*$ $(n\to \infty)$.};
			\item[(A.2)] \hypertarget{A.2}{} For every $\boldsymbol{u},\boldsymbol{z}\hspace*{-0.1em}\in \hspace*{-0.1em}\bXp$, the mapping $(t\hspace*{-0.1em}\mapsto\hspace*{-0.1em}\langle A(t)(\boldsymbol{u}(t)),\boldsymbol{z}(t)\rangle_{\smash{U^{q,p}_{\D}(t)}})\colon \hspace*{-0.1em}I\hspace*{-0.1em}\to\hspace*{-0.1em} \mathbb{R}$  is (Lebesgue)~\mbox{measurable};
			\item[(A.3)] \hypertarget{A.3}{} For a non-decreasing  mapping $\mathbb{B}\colon \mathbb{R}_{\ge 0}\to \mathbb{R}_{\ge 0}$, for a.e.\   $t\in I$ and every $\mathbf{u},\mathbf{z}\in U^{q,p}_{\D}(t)$,~it~holds~that
			\begin{align*}
			\vert\langle A(t)\mathbf{u},\mathbf{z}\rangle_{\smash{U^{q,p}_{\D}(t)}}\vert \leq \mathbb{B}(\|\mathbf{u}\|_{2,\Omega})\big(1+\rho_{q(t,\cdot),\Omega}(\mathbf{u})+\rho_{p(t,\cdot),\Omega}(\mathbf{D}_x\mathbf{u})+\rho_{q(t,\cdot),\Omega}(\mathbf{z})+\rho_{p(t,\cdot),\Omega}(\mathbf{D}_x\mathbf{z})\big)\,.
			\end{align*}
		\end{description}
		Then, the \textit{induced} operator $\mathbfcal{A}\colon \bXp\cap L^\infty(I;Y)\to (\bXp)^*$, for every  $\boldsymbol{u}\in \bXp\cap L^\infty(I;Y)$~and~$ \boldsymbol{z} \in \bXp$ defined by
		\begin{align}\label{def:induced}
			\langle \mathbfcal{A}\boldsymbol{u},\boldsymbol{z}\rangle_{ \bXp}\coloneqq \int_I{\langle A(t)(\boldsymbol{u}(t)),\boldsymbol{z}(t)\rangle_{\smash{U^{q,p}_{\D}(t)}}\,\mathrm{d}t}\,,
		\end{align} 
		 is well-defined, bounded, and demi-continuous (\textit{cf}.\ \cite[Prop.\  5.7]{alex-book}).
		 
		The \textit{k-th.\ temporal means} $\langle A\rangle_k \coloneqq \langle A\rangle_{I_k} \colon U^{q,p}_{\plus}\to (U^{q,p}_{\plus})^*$, $k=1,\ldots ,K$, of ${A(t)\colon U^{q,p}_{\D}(t)\to (U^{q,p}_{\D}(t))^*}$, ${ t\in I} $, for every $k=1,\ldots ,K$ and $\mathbf{u}\in U^{q,p}_{\plus}$ are defined by\vspace*{-0.5mm}
		\begin{align}\label{def:temp_mean}
		\langle A\rangle_k  \mathbf{u}\coloneqq \fint_{I_k}{(\mathrm{id}_{\smash{U^{q,p}_{\plus}}})^*A(t)\mathbf{u}\,\mathrm{d}t}\quad\textup{ in }(U^{q,p}_{\plus})^*\,.
		\end{align} 
		The \textit{(local) $L^2$-projection} $\Pi^{0,\mathrm{t}}_{\tau}[A](t)\colon \hspace*{-0.1em}U^{q,p}_{\plus}\hspace*{-0.1em}\to \hspace*{-0.1em}(U^{q,p}_{\plus})^*$, $t\in I$, of ${A(t)\colon\hspace*{-0.1em} U^{q,p}_{\D}(t)\hspace*{-0.1em}\to\hspace*{-0.1em} (U^{q,p}_{\D}(t))^*}$,~${t\hspace*{-0.1em}\in\hspace*{-0.1em} I}$,  for a.e.\   $t\in I$ and every  $\mathbf{u}\in U^{q,p}_{\plus}$ is defined by\vspace*{-1.5mm}
		\begin{align}\label{def:l2_proj1}
		\Pi^{0,\mathrm{t}}_{\tau}[A](t)\mathbf{u}\coloneqq \sum_{k=1}^{K}{\langle A\rangle_k \mathbf{u}\,\chi_{I_k}(t)}\quad\textup{ in }(U^{q,p}_{\plus})^*\,.
		\end{align} 
		The \textit{(local) $L^2$-projection} $\Pi^{0,\mathrm{t}}_{\tau}[\mathbfcal{A}]\colon \hspace*{-0.1em} \bXplus\cap  L^\infty(I;Y)\hspace*{-0.15em}\to\hspace*{-0.15em} (\bXplus)^*$ of 
		${\mathbfcal{A}\colon \hspace*{-0.1em}\bXp\hspace*{-0.15em}\cap\hspace*{-0.15em} L^\infty(I;Y)\hspace*{-0.15em}\to \hspace*{-0.15em} (\bXp)^*}$,  for every $\boldsymbol{u}\in \bXplus\cap  L^\infty(I;Y)$ and $ \boldsymbol{z} \in \bXplus$ is defined by\vspace*{-1mm}
		\begin{align}\label{def:l2_proj2}
		\langle \Pi^{0,\mathrm{t}}_{\tau}[\mathbfcal{A}]\boldsymbol{u},\boldsymbol{z}\rangle_{\bXplus}\coloneqq \int_I{\langle \Pi^{0,\mathrm{t}}_{\tau}[A](t)(\boldsymbol{u}(t)),\boldsymbol{z}(t)\rangle_{U^{q,p}_{\plus}}\,\mathrm{d}t}\,.
		\end{align}  
		
		\begin{proposition}\label{rothe3}
			Given the definitions \eqref{def:temp_mean}--\eqref{def:l2_proj2}, the following statements apply:
		\begin{description}[noitemsep,topsep=2pt,leftmargin=!,labelwidth=\widthof{(iii)},font=\normalfont\itshape]
			\item[(i)] For every $k=1,\ldots,K$, the operator $\langle A\rangle_k \colon U^{q,p}_{\plus}\to (U^{q,p}_{\plus})^*$ is well-defined, bounded,~and~demi-continuous;
			\item[(ii)]  The operators $\Pi^{0,\mathrm{t}}_{\tau}[A](t)\colon U^{q,p}_{\plus}\to (U^{q,p}_{\plus})^*$, $t\in I$, have the following properties:
			\begin{description}[noitemsep,topsep=2pt,leftmargin=!,labelwidth=\widthof{(ii.a)},font=\normalfont\itshape]
				\item[(ii.a)] For a.e.\   $t\in I$, the operator $\Pi^{0,\mathrm{t}}_{\tau}[A](t)\colon U^{q,p}_{\plus}\to (U^{q,p}_{\plus})^*$ is  demi-continuous;
				\item[(ii.b)] For every $\mathbf{u},\mathbf{z}\in U^{q,p}_{\plus}$, the mapping $(t\mapsto\langle \Pi^{0,\mathrm{t}}_{\tau}[A](t)\mathbf{u},\mathbf{z}\rangle_{\smash{U^{q,p}_{\plus}}})\colon U^{q,p}_{\plus}\to (U^{q,p}_{\plus})^*$ is  (Lebesgue) measurable;
				\item[(ii.c)] For a non-decreasing mapping $\mathbb{B}_{\plus}\colon \mathbb{R}_{\ge 0}\to \mathbb{R}_{\ge 0}$,  for a.e.\   $t\in I$ and every $\mathbf{u}\in U^{q,p}_{\plus}$,~it~holds~that $\|\Pi^{0,\mathrm{t}}_{\tau}[A](t)\mathbf{u}\|_{\smash{(U^{q,p}_{\plus})^*}} \leq \mathbb{B}_{\plus}(\|\mathbf{u}\|_{2,\Omega})\big(1+\|\mathbf{u}\|_{q^+,p^+,\Omega}^{p^+-1}\big)$. 
			\end{description}
			\item[(iii)] The operator $\Pi^{0,\mathrm{t}}_{\tau}[\mathbfcal{A}]\colon \bXplus\cap  L^\infty(I;Y)\to(\bXplus)^*$ is well-defined and has the following properties:
			\begin{description}[noitemsep,topsep=2pt,leftmargin=!,labelwidth=\widthof{(iii.a)},font=\normalfont\itshape]
				\item[(iii.a)] \hspace{-0.1em}If $\boldsymbol{u}_\tau\hspace{-0.15em}\in\hspace{-0.15em} \mathbb{P}^0(\mathcal{I}_\tau;U^{q,p}_{\plus})$, then $\langle \Pi^{0,\mathrm{t}}_{\tau}[\mathbfcal{A}]\boldsymbol{u}_\tau,\boldsymbol{z}\rangle_{\bXplus}\hspace{-0.15em}=\hspace{-0.15em}\langle \mathbfcal{A}\boldsymbol{u}_\tau, \hspace{-0.15em}\Pi^{0,\mathrm{t}}_{\tau}[\boldsymbol{z}]\rangle_{\bXp}$~for~all~${\boldsymbol{z}\hspace{-0.15em}\in \hspace{-0.15em}\bXplus}$;
				\item[(iii.b)] \hspace{-0.1em}If  $\boldsymbol{u}_\tau\in\mathbb{P}^0(\mathcal{I}_\tau;U^{q,p}_{\plus})$, then $\|\Pi^{0,\mathrm{t}}_{\tau}[\mathbfcal{A}]\boldsymbol{u}_\tau\|_{(\bXplus)^*}\leq 2\,(1+\vert Q_T\vert) \,\|\mathbfcal{A}\boldsymbol{u}_\tau\|_{(\bXp)^*}$.\enlargethispage{10mm}
			\end{description}
		\end{description}
	\end{proposition}

	\begin{proof}
		\textit{ad (i).}  \hspace*{-0.15mm}Let \hspace*{-0.15mm}$\mathbf{u}\hspace*{-0.175em}\in\hspace*{-0.175em}  U^{q,p}_{\plus}$ \hspace*{-0.15mm}be  \hspace*{-0.15mm}arbitrary. \hspace*{-0.15mm}Due \hspace*{-0.15mm}to \hspace*{-0.15mm}(\hyperlink{A.2}{A.2}),  \hspace*{-0.15mm}(\hyperlink{A.3}{A.3}),  \hspace*{-0.15mm}the \hspace*{-0.15mm}mapping \hspace*{-0.15mm}${(t\hspace*{-0.175em}\mapsto\hspace*{-0.175em}(\mathrm{id}_{\smash{U^{q,p}_{\plus}}})^*A(t)\mathbf{u})\colon\hspace*{-0.175em} I\hspace*{-0.175em}\to\hspace*{-0.175em}  (U^{q,p}_{\plus})^*}$ is Bochner  measurable with  $\|(\mathrm{id}_{\smash{U^{q,p}_{\plus}}})^*A(\cdot)\mathbf{u}\|_{(U^{q,p}_{\plus})^*}\in L^1(I)$ and, thus,   Bochner  integrable.  As a result, the Bochner  integrals  $\langle A\rangle_k \mathbf{u}\hspace*{-0.1em}\in \hspace*{-0.1em} (U^{q,p}_{\plus})^*$,  $k=1,\ldots,K$,  exist, \textit{i.e.}, for every $k=1,\ldots,K$,~the~\mbox{operator} $\langle A\rangle_k \colon \hspace*{-0.1em}U^{q,p}_{\plus}\hspace*{-0.1em}\to\hspace*{-0.1em} (U^{q,p}_{\plus})^*$  is well-defined.
		The boundedness and demi-continuity of the operator $\langle A\rangle_k \colon \hspace*{-0.1em}U^{q,p}_{\plus}\hspace*{-0.1em}\to\hspace*{-0.1em} (U^{q,p}_{\plus})^*$ for all $k=1,\ldots,K$  follows~from~(\hyperlink{A.1}{A.1}), (\hyperlink{A.3}{A.3}), and the boundedness of the Bochner integral.
		
		\textit{ad (ii).} Claim (ii.a) and claim (ii.b) follow from (i). Claim (ii.c) follows from  (\hyperlink{A.3}{A.3}).
		
		\textit{ad (iii).} The well-definedness results from (ii) by means of \cite[Prop.\ 3.13]{alex-rose-hirano}.
		
		\textit{ad (iii.a).} Let $\boldsymbol{u}_\tau\in \mathbb{P}^0(\mathcal{I}_\tau;U^{q,p}_{\plus})$ and $\boldsymbol{z}\in \bXplus$ be arbitrary. Then, using every $t,s\in I_k$,~${k=1,\ldots ,K}$, that $\langle A(s)(\boldsymbol{u}_\tau(t)),\boldsymbol{z}(t)\rangle_{\smash{U^{q,p}_{\D}(s)}}=\langle A(s)(\boldsymbol{u}_\tau(s)),\boldsymbol{z}(t)\rangle_{\smash{U^{q,p}_{\D}(s)}}$ and Fubini's theorem,~we~\mbox{deduce}~that\vspace*{-1mm}
		\begin{align*}
			\langle \Pi^{0,\mathrm{t}}_{\tau}[\mathbfcal{A}]\boldsymbol{u}_\tau,\boldsymbol{z}\rangle_{\bXplus}&=\sum_{k=1}^{K}{\int_{I_k}{\bigg\langle\fint_{I_k}{(\mathrm{id}_{\smash{U^{q,p}_{\plus}}})^* A(s)(\boldsymbol{u}_\tau(t))\,\mathrm{d}s},\boldsymbol{z}(t)\bigg\rangle_{\smash{U^{q,p}_{\plus}}}\,\mathrm{d}t}}
		\\&=\sum_{k=1}^{K}{\int_{I_k}{\bigg\langle  A(s)(\boldsymbol{u}_\tau(s)),\fint_{I_k}{\boldsymbol{z}(t)\,\mathrm{d}t}\bigg\rangle_{\smash{U^{q,p}_{\D}(s)}}\,\mathrm{d}s}}
		\\&=\langle  \mathbfcal{A}\boldsymbol{u}_\tau,\Pi^{0,\mathrm{t}}_{\tau}[\boldsymbol{z}]\rangle_{\bXp}\,,
		\end{align*}
		which is the claimed self-adjointness property.
		
		\textit{ad (iii.b).} By the self-adjointness property (iii.a) and $\|\boldsymbol{u}\|_{\smash{\bXp}}\leq 2\,(1+\vert Q_T\vert)\,\|\boldsymbol{u}\|_{\smash{\bXplus}}$ for all  $\boldsymbol{u}\in \bXplus$ (\textit{cf}.\ \eqref{eq:hoelder_embeddding}),~for~every~$ \boldsymbol{u}_\tau\in \mathbb{P}^0(\mathcal{I}_\tau;U^{q,p}_{\plus})$, we deduce that\vspace*{-1mm}
		\begin{align*}
	\|\Pi^{0,\mathrm{t}}_{\tau}[\mathbfcal{A}]\boldsymbol{u}_\tau\|_{\smash{(\bXplus)^*}}
 &= \sup_{\boldsymbol{z}\in \bXplus\,:\,\|\boldsymbol{z}\|_{\smash{\bXplus}\leq 1}}{\langle \mathbfcal{A}\boldsymbol{u}_\tau,\Pi^{0,\mathrm{t}}_{\tau}[\boldsymbol{z}]\rangle_{\smash{\bXp}}}
 \\&\leq 2\,(1+\vert Q_T\vert)\,	\|\mathbfcal{A}\boldsymbol{u}_\tau\|_{\smash{(\bXp)^*}}\,,
		\end{align*}
		which is the claimed stability estimate.
	\end{proof}
    \newpage

    \section{Non-conforming Bochner condition (M)}\label{sec:conditionM}\enlargethispage{5mm}

    \hspace{5mm}In this section, we introduce the non-conforming Bochner condition (M), a generalization of the Bochner \hspace*{-0.1mm}condition \hspace*{-0.1mm}(M),  
    \hspace*{-0.1mm}introduced  \hspace*{-0.1mm}in \hspace*{-0.1mm}the \hspace*{-0.1mm}framework \hspace*{-0.1mm}of \hspace*{-0.1mm}variable \hspace*{-0.1mm}Bochner--Lebesgue \hspace*{-0.1mm}spaces~\hspace*{-0.1mm}in~\hspace*{-0.1mm}\mbox{\cite[Sec.~\hspace*{-0.1mm}5.1]{alex-book}}.

    \begin{definition}[Non-conforming Bochner condition (M)]\label{def:non-conform_Bochner_cond_M}
    	A sequence of operators  $\mathbfcal{A}_n\colon \bXpn\cap L^\infty(I;Y)\to (\bXpn)^*$, $n\in  \mathbb{N}$,  is said to satisfy the \textup{non-conforming Bochner condition (M)}~with~respect to $(\Vo_{h_n,0})_{n\in \mathbb{N}}$ and $\mathbfcal{A}\colon\bVp\cap  L^\infty(I;H) \to (\bVp)^*$ if for every sequence $\boldsymbol{v}_n\in \bXpn\cap L^\infty(I;Y)$, $n\in\mathbb{N}$, with $\boldsymbol{v}_n(t)\in \Vo_{h_n,0}$ for a.e.\  $t\in I$~and~all ${n\in \mathbb{N}}$, from
    	\begin{align}
    		\hspace{-5.5mm}\sup_{n\in \mathbb{N}}{\big[\|\boldsymbol{v}_n\|_{\smash{\bXpn}}+\|\mathbfcal{A}_n\boldsymbol{v}_n\|_{(\bXpn)^*}\big]}<\infty\,, \label{def:non-conform_Bochner_cond_M.1}
    	\end{align}\vspace{-4mm}
    	\begin{alignat}{3}
    		\boldsymbol{v}_n&\overset{\ast}{\rightharpoondown} \boldsymbol{v}&&\quad \text{ in }L^\infty(I;Y)&&\quad (n\to \infty)\,, \label{def:non-conform_Bochner_cond_M.2}\\
    		\boldsymbol{v}_n(t)&\rightharpoonup \boldsymbol{v}(t)&&\quad \text{ in }Y&&\quad (n\to \infty)\quad \text{ for a.e.\  }t\in I\,, \label{def:non-conform_Bochner_cond_M.3}\\
    		(\mathrm{id}_{\mathbfcal{D}_T})^*\mathbfcal{A}_n\boldsymbol{v}_n&\to (\mathrm{id}_{\mathbfcal{D}_T})^*\boldsymbol{v}^*&&\quad \text{ in }(\mathbfcal{D}_T)^*&&\quad (n\to \infty)\,,\label{def:non-conform_Bochner_cond_M.4}
    	\end{alignat}
    	where $\boldsymbol{v}\in \bVp\cap  L^\infty(I;H)$ and $\boldsymbol{v}^*\in (\bVp)^*$, and
    	\begin{align}
    		\limsup_{n\to \infty}{\langle \mathbfcal{A}_n\boldsymbol{v}_n,\boldsymbol{v}_n\rangle_{\smash{\bXpn}}}\leq \langle  \boldsymbol{v}^*,\boldsymbol{v}\rangle_{\smash{\bVp}}\,,\label{def:non-conform_Bochner_cond_M.5}
    	\end{align}
    	it follows that $\mathbfcal{A}\boldsymbol{v}=\boldsymbol{v}^*$ in $(\bVp)^*$.
    \end{definition}
    
     An example of a sequence of operators that satisfies the non-conforming Bochner condition~(M), is given via the approximation of  $\mathbf{S}\colon Q_T\times \mathbb{R}_{\mathrm{sym}}^{d\times d}\to \mathbb{R}_{\mathrm{sym}}^{d\times d}$ having a weak $(p(\cdot,\cdot),\delta)$-structure (\textit{cf}.\ (\hyperlink{S.1}{S.1})--(\hyperlink{S.4}{S.4}))
    through mappings
    $\mathbf{S}_n\colon Q_T\times \mathbb{R}_{\mathrm{sym}}^{d\times d}\to \mathbb{R}_{\mathrm{sym}}^{d\times d}$, $n\in \mathbb{N}$,  having a weak $(p_n(\cdot,\cdot),\delta)$-structure for all $n\in \mathbb{N}$.

    \begin{proposition}\label{prop:stress}	
    	Let $\mathbf{S}_n\colon Q_T\times \mathbb{R}_{\mathrm{sym}}^{d\times d} \to \mathbb{R}_{\mathrm{sym}}^{d\times d}$, $n\in \mathbb{N}$, be mappings having the following properties:
    	\begin{description}[leftmargin=!,labelwidth=\widthof{(SN.4)},noitemsep,topsep=2pt,font=\normalfont\itshape]
    		\item[(SN.1)]\hypertarget{SN.1}{} $\mathbf{S},\mathbf{S}_n\colon Q_T\times\mathbb{R}_{\mathrm{sym}}^{d\times d}\to \mathbb{R}_{\mathrm{sym}}^{d\times d}$, $n\in \mathbb{N}$, are Carath\'eodory mappings and  for a.e.\  $(t,x)^{\top}\in Q_T$ and every  $\mathbf{A}\in \mathbb{R}_{\mathrm{sym}}^{d\times d}$, it holds that $\mathbf{S}_n(t,x,\mathbf{A})\to \mathbf{S}(t,x,\mathbf{A})$ in $\mathbb{R}_{\mathrm{sym}}^{d\times d}$ $(n\to \infty)$;
    		\item[(SN.2)]\hypertarget{SN.2}{} $\vert\mathbf{S}_n(t,x,\mathbf{A})\vert \leq \overline{\alpha} \,(\delta+\vert \mathbf{A}\vert)^{p_n(t,x)-2}\vert\mathbf{A}\vert+\overline{\beta}$  for all $n\in \mathbb{N}$, $\mathbf{A}\in \mathbb{R}_{\mathrm{sym}}^{d\times d}$, and a.e.\ ${(t,x)^\top\in Q_T}$, where $\overline{\alpha}\ge 1$ and $\overline{\beta} \ge 0$ are independent of $n\in \mathbb{N}$;
    		\item[(SN.3)]\hypertarget{SN.3}{} $\mathbf{S}_n(t,x,\mathbf{A}):\mathbf{A}\ge 
    		\overline{c}_0\,(\delta+\vert \mathbf{A}\vert)^{p_n(t,x)-2}\vert\mathbf{A}\vert^2-\overline{c}_1(t,x)$ for all $n\!\in\! \mathbb{N}$, $\mathbf{A}\in \mathbb{R}_{\mathrm{sym}}^{d\times d}$, and a.e.~${(t,x)^\top\!\in\! Q_T}$, where $\overline{c}_0>0$ and $\overline{c}_1\in L^1(Q_T)$ are independent of $n\in \mathbb{N}$;
    		\item[(SN.4)]\hypertarget{SN.4}{} $(\mathbf{S}_n(t,x,\mathbf{A})-\mathbf{S}_n(t,x,\mathbf{B})):
    		(\mathbf{A}-\mathbf{B})\ge 0$ for all $n\!\in\! \mathbb{N}$,
    		$\mathbf{A},\mathbf{B}\in \mathbb{R}_{\mathrm{sym}}^{d\times d}$, and a.e.\  $(t,x)^\top\in Q_T$.
    	\end{description}
    	Then, the operators $\mathbfcal{S}_n\colon \bXpn\to (\bXpn)^*$, $n\in \mathbb{N}$, for every $\boldsymbol{u}_n,\boldsymbol{z}_n\in\bXpn$ defined by
    	\begin{align*}
    		\langle\mathbfcal{S}_n\boldsymbol{u}_n,\boldsymbol{z}_n\rangle_{\smash{\bXpn}}\coloneqq (\mathbf{S}_n(\cdot,\cdot,\mathbf{D}_x\boldsymbol{u}_n),\mathbf{D}_x\boldsymbol{z}_n)_{Q_T}\,,
    	\end{align*}
    	are well-defined, bounded, continuous, and monotone. In addition, $\mathbfcal{S}_n\colon\bXpn\to (\bXpn)^*$,~${n\in  \mathbb{N}}$, satisfies the non-conforming Bochner condition (M) with respect to $(\Vo_{h_n,0})_{n\in \mathbb{N}}$ and $\mathbfcal{S}\colon \bVp\to (\bVp)^*$, for every $\boldsymbol{v},\boldsymbol{z}\in \bVp$
    	defined by
    	\begin{align*}
    		\langle\mathbfcal{S}_n\boldsymbol{v},\boldsymbol{z}\rangle_{\smash{\bVp}}\coloneqq (\mathbf{S}(\cdot,\cdot,\mathbf{D}_x\boldsymbol{v}),\mathbf{D}_x\boldsymbol{z})_{Q_T}\,.
    	\end{align*} 
    \end{proposition}
    
    \begin{remark}
    	\begin{itemize}[noitemsep,topsep=2pt,leftmargin=!,labelwidth=\widthof{(ii)}]
    		\item[(i)] If $\mathbf{S},\mathbf{S}_n\colon Q_T\times\mathbb{R}_{\mathrm{sym}}^{d\times d}\to \mathbb{R}_{\mathrm{sym}}^{d\times d}$, $n\in \mathbb{N}$, are mappings satisfying (\hyperlink{SN.1}{SN.1})--(\hyperlink{SN.4}{SN.4}), then $\mathbf{S}\colon Q_T\times\mathbb{R}_{\mathrm{sym}}^{d\times d}\to \mathbb{R}_{\mathrm{sym}}^{d\times d}$ automatically has
    		a weak $(p(\cdot,\cdot),\delta)$-structure (\textit{cf}.\ (\hyperlink{S.1}{S.1})--(\hyperlink{S.4}{S.4})); 
    		\item[(ii)] According to \cite[Prop.\ 3.31]{alex-book}, the operator $\mathbfcal{S}\colon \bVp\to (\bVp)^*$ is well-defined, bounded, continuous, and monotone.
    	\end{itemize}
    \end{remark}

    \begin{proof}[Proof (of Proposition \ref{prop:stress}).]
    	According to \cite[Prop.\ 3.31]{alex-book}, for every $n\hspace*{-0.15em}\in \hspace*{-0.15em}\mathbb{N}$, ${\mathbfcal{S}_n\colon \hspace*{-0.1em}\bXpn\hspace*{-0.15em}\to\hspace*{-0.15em} (\bXpn)^*}$ is 
    	well-defined, bounded, continuous, and monotone. To prove the non-conforming Bochner~\mbox{condition}~(M), we consider a sequence  $\boldsymbol{v}_n\in \bXpn\cap L^\infty(I;Y)$, $n\in\mathbb{N}$, with $\boldsymbol{v}_n(t)\in \Vo_{h_n,0}$ for a.e.\  $t\in I$ and~all~${n\in \mathbb{N}}$\linebreak as well as satisfying \eqref{def:non-conform_Bochner_cond_M.1}--\eqref{def:non-conform_Bochner_cond_M.5} with respect to $\mathbfcal{S}_n\colon \hspace*{-0.1em}\bXpn\hspace*{-0.1em}\to\hspace*{-0.1em} (\bXpn)^*$,~$n\hspace*{-0.1em}\in \hspace*{-0.1em}\mathbb{N}$, $\boldsymbol{v}\hspace*{-0.1em}\in\hspace*{-0.1em} \bVp$,~and~${\boldsymbol{v}^*\hspace*{-0.1em}\in\hspace*{-0.1em}(\bVp)^*}$. From the monotonicity of $\mathbfcal{S}_n\colon \bXpn\to (\bXpn)^*$, $ n\in\mathbb{N}$, 
    	for every $\boldsymbol{\varphi}\in \mathbfcal{D}_T$~and~${n\in \mathbb{N}}$,~we~obtain
    	\begin{align}
    		\langle \mathbfcal{S}_n\boldsymbol{v}_n-\mathbfcal{S}_n\boldsymbol{\varphi},\boldsymbol{v}_n-\boldsymbol{\varphi}\rangle_{\smash{\bXpn}}\ge 0\,.\label{eq:stress.1}
    	\end{align}
    	Rearranging \eqref{eq:stress.1}, for every $\boldsymbol{\varphi}\in \mathbfcal{D}_T$ and $n\in \mathbb{N}$, yields  that
    	\begin{align}
    		\langle \mathbfcal{S}_n\boldsymbol{v}_n,\boldsymbol{\varphi}\rangle_{\smash{\bXpn}}+\langle \mathbfcal{S}_n\boldsymbol{\varphi},\boldsymbol{v}_n-\boldsymbol{\varphi}\rangle_{\smash{\bXpn}}\leq \langle \mathbfcal{S}_n\boldsymbol{v}_n,\boldsymbol{v}_n\rangle_{\smash{\bXpn}}\,.\label{eq:stress.2}
    	\end{align}
    	Then, observing that $\mathbf{S}_n(\cdot,\cdot,\mathbf{D}_x\boldsymbol{\varphi})\to \mathbf{S}(\cdot,\cdot,\mathbf{D}_x\boldsymbol{\varphi})$ in $L^s(Q_T;\mathbb{R}_{\mathrm{sym}}^{d\times d})$ $(n\to \infty)$ for all $s\in \left[1,\infty\right)$, due to (\hyperlink{SN.1}{SN.1}), (\hyperlink{SN.2}{SN.2}), and Lebesgue's dominated convergence theorem, and that $\boldsymbol{v}_n\rightharpoonup \boldsymbol{v}$ in $\mathbfcal{U}^{r,s}_{\D}$~$(n\to \infty)$ for all $r,s\in \mathcal{P}^\infty(Q_T)$ with $r\prec q$ in $Q_T$ and $s\prec p$ in $Q_T$ (\textit{cf.}~Proposition~\ref{prop:Xqp_weak_compact}), taking the limit superior with respect to $n\to \infty$ in \eqref{eq:stress.2}, using \eqref{def:non-conform_Bochner_cond_M.4} and \eqref{def:non-conform_Bochner_cond_M.5} in doing so, for every $\boldsymbol{\varphi}\in \mathbfcal{D}_T$, we obtain
    	\begin{align}
    		\smash{\langle \boldsymbol{v}^*,\boldsymbol{\varphi}\rangle_{\smash{\bVp}}+\langle \mathbfcal{S}\boldsymbol{\varphi},\boldsymbol{v}-\boldsymbol{\varphi}\rangle_{\smash{\bVp}}\leq \langle \boldsymbol{v}^*,\boldsymbol{v}\rangle_{\smash{\bVp}}\,.}\label{eq:stress.3}
    	\end{align}
    	Rearranging \eqref{eq:stress.3}, in turn,  for every $\boldsymbol{\varphi}\in \mathbfcal{D}_T$, yields that
    	\begin{align}
    		\smash{\langle 
    		\boldsymbol{v}^*-\mathbfcal{S}\boldsymbol{\varphi},\boldsymbol{v}-\boldsymbol{\varphi}\rangle_{\smash{\bVp}}\ge 0\,.}\label{eq:stress.4}
    	\end{align}
    	Since $\mathbfcal{D}_T$ is dense in $\bVp$ (\textit{cf}.\ Proposition \ref{prop:density_in_Vqp}) and
    	$\mathbfcal{S}\colon\hspace*{-0.15em} \bVp\hspace*{-0.15em}\to\hspace*{-0.15em}(\bVp)^*$ continuous and maximal~\mbox{monotone}, we conclude from \eqref{eq:stress.4} that $\mathbfcal{S}\boldsymbol{v}=\boldsymbol{v}^*$ in $(\bVp)^*$.
    \end{proof}

    Similar to the classical Bochner condition (M), the non-conforming Bochner condition (M) is stable under certain strongly continuous perturbations.\enlargethispage{11mm}
    
    \begin{definition}[Non-conforming Bochner strong continuity]\label{def:non-conform_strong_continuity} 
    	A sequence of operators  $\mathbfcal{B}_n\colon \bXpn\cap L^\infty(I;Y)\to (\bXpn)^*$, $n\in  \mathbb{N}$, is said to be \textup{non-conforming Bochner strongly continuous}~with~respect~to  $(\Vo_{h_n,0})_{n\in \mathbb{N}}$ and $\mathbfcal{B}\colon\bVp\cap  L^\infty(I;H) \to (\bVp)^*$ if for every sequence $\boldsymbol{v}_n\in \bXpn\cap L^\infty(I;Y)$, $n\in\mathbb{N}$, with $\boldsymbol{v}_n(t)\in \Vo_{h_n,0}$ for a.e.\  $t\in I$ and all $n\in \mathbb{N}$, from \eqref{def:non-conform_Bochner_cond_M.1}--\eqref{def:non-conform_Bochner_cond_M.3}  with respect to $\boldsymbol{v}\in \bVp\cap  L^\infty(I;H)$, it follows that\vspace*{-1mm}
    	\begin{align}
    		\lim_{n\to \infty}{\langle \mathbfcal{B}_n\boldsymbol{v}_n,\boldsymbol{z}_n\rangle_{\smash{\bXpn}}}=\langle  \mathbfcal{B}\boldsymbol{v},\boldsymbol{z}\rangle_{\smash{\bVp}}\,,\label{def:non-conform_strong_continuity.5}
    	\end{align}
    	for  \hspace*{-0.1mm}each \hspace*{-0.1mm}sequence \hspace*{-0.1mm}$\boldsymbol{z}_n\hspace*{-0.1em}\in\hspace*{-0.1em}  \bXpn\cap L^\infty(I;Y)$, $n\hspace*{-0.1em}\in\hspace*{-0.1em} \mathbb{N}$, \hspace*{-0.1mm}satisfying \hspace*{-0.1mm}\eqref{def:non-conform_Bochner_cond_M.1}--\hspace*{-0.2mm}\eqref{def:non-conform_Bochner_cond_M.3} \hspace*{-0.1mm}with~\hspace*{-0.1mm}respect~\hspace*{-0.1mm}to~\hspace*{-0.1mm}${\boldsymbol{z}\hspace*{-0.1em}\in\hspace*{-0.1em} \bVp\hspace*{-0.1em}\cap\hspace*{-0.1em}  L^\infty(I;H)}$.
    \end{definition}
    
    \begin{remark}[Non-conforming Bochner strong continuity $\Rightarrow$ weak continuity property]\label{rem:1}
    	If $\mathbfcal{B}_n\colon \bXpn\cap L^\infty(I;Y)\to (\bXpn)^*$, $n\in  \mathbb{N}$, is \textit{non-conforming Bochner strongly continuous} with respect to $(\Vo_{h_n,0})_{n\in \mathbb{N}}$ and $\mathbfcal{B}\colon\bVp\cap  L^\infty(I;H) \to (\bVp)^*$ and if ${\boldsymbol{v}_n\in \bXpn\cap L^\infty(I;Y)}$, $n\in\mathbb{N}$, is a sequence with $\boldsymbol{v}_n(t)\in \Vo_{h_n,0}$ for a.e.\  $t\in I$ and all $n\in \mathbb{N}$, satisfying \eqref{def:non-conform_Bochner_cond_M.1}--\eqref{def:non-conform_Bochner_cond_M.3}  with respect to $\boldsymbol{v}\in \bVp\cap  L^\infty(I;H)$,~then 
    	\begin{align*}
    	{(\mathrm{id}_{\mathbfcal{D}_T})^*\mathbfcal{B}_n\boldsymbol{v}_n\to (\mathrm{id}_{\mathbfcal{D}_T})^*\mathbfcal{B}\boldsymbol{v}\quad \text{ in }(\mathbfcal{D}_T)^*\quad (n\to \infty)\,.}
    	\end{align*}

    \end{remark}
    
    \begin{proposition}\label{prop:perturbation}	
    	If $\mathbfcal{A}_n\colon\hspace*{-0.15em} \bXpn\cap L^\infty(I;Y)\hspace*{-0.15em}\to\hspace*{-0.15em} (\bXpn)^*$, $n\hspace*{-0.15em}\in \hspace*{-0.15em} \mathbb{N}$, satisfies the Bochner~condition~(M)~with respect to   $(\Vo_{h_n,0})_{n\in \mathbb{N}}$ and $\mathbfcal{A}\colon\hspace*{-0.16em}\bVp\cap  L^\infty(I;H) \hspace*{-0.16em}\to\hspace*{-0.16em} (\bVp)^*\hspace*{-0.16em}$, and if $\mathbfcal{B}_n\colon \hspace*{-0.16em}\bXpn\cap L^\infty(I;Y)\hspace*{-0.16em}\to\hspace*{-0.16em} (\bXpn)^*$,~${n\hspace*{-0.16em}\in\hspace*{-0.16em}  \mathbb{N}}$, \hspace*{-0.175mm}is \hspace*{-0.175mm}non-conforming \hspace*{-0.175mm}Bochner \hspace*{-0.175mm}strongly \hspace*{-0.175mm}continuous \hspace*{-0.175mm}with \hspace*{-0.175mm}respect \hspace*{-0.175mm}to \hspace*{-0.175mm}$(\Vo_{h_n,0})_{n\in \mathbb{N}}$~\hspace*{-0.175mm}and~\hspace*{-0.175mm}${\mathbfcal{B}\colon\hspace*{-0.25em} \bVp\hspace*{-0.225em} \cap\hspace*{-0.175em}   L^\infty(I;H) \hspace*{-0.25em} \to\hspace*{-0.25em}  (\bVp)^*}$,  then the sequence of operators $\mathbfcal{A}_n+\mathbfcal{B}_n\colon  \bXpn\cap L^\infty(I;Y)\to  (\bXpn)^*$, $n\in\mathbb{N}$, satisfies the Bochner condition (M)  with respect to $(\Vo_{h_n,0})_{n\in \mathbb{N}}$ and  $\mathbfcal{A}+\mathbfcal{B}\colon\bVp\cap  L^\infty(I;H) \to (\bVp)^*$.
    \end{proposition}
    
    \begin{proof}
    	Let $\boldsymbol{v}_n\in \bXpn\cap L^\infty(I;Y)$, $n\in\mathbb{N}$, be a sequence with  $\boldsymbol{v}_n(t)\in \Vo_{h_n,0}$ for a.e.\  $t\in I$~and~all~${n\in \mathbb{N}}$ as well as satisfying \eqref{def:non-conform_Bochner_cond_M.1}--\eqref{def:non-conform_Bochner_cond_M.5} with respect to  $\mathbfcal{A}+\mathbfcal{B}\colon\bVp\cap  L^\infty(I;H) \to (\bVp)^*$. Due~to~Remark~\ref{rem:1}, the non-conforming Bochner strong continuity  of  $\mathbfcal{B}_n\colon \hspace*{-0.1em}\bXpn\cap L^\infty(I;Y)\hspace*{-0.1em}\to\hspace*{-0.1em} (\bXpn)^*$,~${n\hspace*{-0.1em}\in\hspace*{-0.1em}  \mathbb{N}}$,~implies~that
    	\begin{align}\label{prop:perturbation.1}	
    	\smash{	(\mathrm{id}_{\mathbfcal{D}_T})^*\mathbfcal{B}_n\boldsymbol{v}_n\to (\mathrm{id}_{\mathbfcal{D}_T})^*\mathbfcal{B}\boldsymbol{v}\quad \text{ in }(\mathbfcal{D}_T)^*\quad (n\to \infty)\,,}
    	\end{align}
    	and 
    	\begin{align}\label{prop:perturbation.2}	
    		\lim_{n\to \infty}{\langle \mathbfcal{B}_n\boldsymbol{v}_n,\boldsymbol{v}_n\rangle_{\smash{\bXpn}}}=\langle  \mathbfcal{B}\boldsymbol{v},\boldsymbol{v}\rangle_{\smash{\bVp}}\,.
    	\end{align}
    	Therefore, combining \eqref{prop:perturbation.1} and \eqref{def:non-conform_Bochner_cond_M.4}, we find that
    	\begin{align}\label{prop:perturbation.3}	
    	\smash{	(\mathrm{id}_{\mathbfcal{D}_T})^*\mathbfcal{A}_n\boldsymbol{v}_n\to (\mathrm{id}_{\mathbfcal{D}_T})^*(\boldsymbol{v}^*-\mathbfcal{B}\boldsymbol{v})\quad \text{ in }(\mathbfcal{D}_T)^*\quad (n\to \infty)\,,}
    	\end{align}
    	and combining \eqref{prop:perturbation.2} and \eqref{def:non-conform_Bochner_cond_M.5}, we find that
    	\begin{align}\label{prop:perturbation.4}	
    		\limsup_{n\to \infty}{\langle \mathbfcal{A}_n\boldsymbol{v}_n,\boldsymbol{v}_n\rangle_{\smash{\bXpn}}}\leq \langle  \boldsymbol{v}^*-\mathbfcal{B}\boldsymbol{v},\boldsymbol{v}\rangle_{\smash{\bVp}}\,.
    	\end{align}
    	As a result, given \eqref{def:non-conform_Bochner_cond_M.1}--\eqref{def:non-conform_Bochner_cond_M.3} together with  \eqref{prop:perturbation.3}	and \eqref{prop:perturbation.4},	
    	the non-conforming Bochner condition~(M) implies that $\mathbfcal{A}\boldsymbol{v}=\boldsymbol{v}^*-\mathbfcal{B}\boldsymbol{v}$ in $(\bVp)^*$, \textit{i.e.},    $(\mathbfcal{A}+\mathbfcal{B})\boldsymbol{v}=\boldsymbol{v}^*$ in $(\bVp)^*$.
    \end{proof}
    
    An example of a sequence of operators that is  non-conforming Bochner strongly continuous,~is~given~via the approximation of the convective term through \textit{Temam's modification} (\textit{cf}.\ \cite{tem}).\enlargethispage{8.5mm}
    
    \begin{proposition}\label{prop:convective_term}	
    	Let  $p^->p_c\coloneqq\frac{3d+2}{d+2}$ and $q_n\ge (p_c)_*=2(p_c)'$ a.e.\ in $Q_T$ for all $n\in \mathbb{N}$. 
    	Then, the operators $\mathbfcal{C}_n\colon  \bXpn \to  (\bXpn)^*$, $n\in  \mathbb{N}$, for every $\boldsymbol{u}_n,\boldsymbol{z}_n\in \bXpn $~defined~by\vspace*{-0.5mm}
    	\begin{align*}
    		\langle\mathbfcal{C}_n\boldsymbol{u}_n,\boldsymbol{z}_n\rangle_{\smash{\bXpn}}\coloneqq \tfrac{1}{2}(\boldsymbol{z}_n\otimes \boldsymbol{u}_n,\nabla_{\! x}    \boldsymbol{u}_n)_{Q_T}-\tfrac{1}{2}(\boldsymbol{u}_n\otimes \boldsymbol{u}_n,\nabla_{\! x}    \boldsymbol{z}_n)_{Q_T}\,,
    	\end{align*}
    	are well-defined and bounded. In addition, the sequence of operators $\mathbfcal{C}_n\colon\bXpn \to (\bXpn)^*$, $n\in  \mathbb{N}$, is~non-conforming Bochner strongly continuous with respect to $(\Vo_{h_n,0})_{n\in \mathbb{N}}$ and $\mathbfcal{C}\colon \bVp\hspace*{-0.1em}\to\hspace*{-0.1em} (\bVp)^*$, for every $\boldsymbol{v},\boldsymbol{z}\in \bVp$ defined by\vspace*{-1mm}
    	\begin{align*}
    			\langle\mathbfcal{C}\boldsymbol{v},\boldsymbol{z}\rangle_{\smash{\bVp}}\coloneqq -(\boldsymbol{v}\otimes \boldsymbol{v},\nabla_{\! x}    \boldsymbol{z})_{Q_T}\,.
    	\end{align*} 
    \end{proposition}
    
    \begin{remark}
    	\begin{itemize}[noitemsep,topsep=2pt,leftmargin=!,labelwidth=\widthof{(ii)}]
    		\item[(i)] Proceeding \hspace*{-0.1mm}as \hspace*{-0.1mm}in \hspace*{-0.1mm}the \hspace*{-0.1mm}proof \hspace*{-0.1mm}of \hspace*{-0.1mm}\cite[Prop.\ \hspace*{-0.1mm}5.3]{alex-book}, \hspace*{-0.1mm}we \hspace*{-0.1mm}find \hspace*{-0.1mm}that \hspace*{-0.1mm}the \hspace*{-0.1mm}operator \hspace*{-0.1mm}${\mathbfcal{C}\colon\hspace*{-0.15em} \bVp\hspace*{-0.15em}\to\hspace*{-0.15em} (\bVp)^*}$ is well-defined and bounded.
    		\item[(ii)] The statement and proof of Proposition \ref{prop:convective_term} can readily be generalized to alternative discretizations of the operator   $\mathbfcal{C}\colon \bVp\hspace*{-0.1em}\to\hspace*{-0.1em} (\bVp)^*$.
    	\end{itemize}
    \end{remark}
    
    \begin{proof}[Proof (of Proposition \ref{prop:convective_term}).]
    	\textit{1. Well-definedness and boundedness.} Due to $q_n\ge (p_c)_*=2(p_c)'$ a.e.\ in $Q_T$ for all $n\hspace*{-0.15em}\in\hspace*{-0.15em}  \mathbb{N}$, \hspace*{-0.1mm}by \hspace*{-0.1mm}Korn's \hspace*{-0.1mm}inequality \hspace*{-0.1mm}and \hspace*{-0.1mm}the \hspace*{-0.1mm}variable \hspace*{-0.1mm}Hölder \hspace*{-0.1mm}inequality \hspace*{-0.1mm}\eqref{eq:gen_hoelder},~\hspace*{-0.1mm}for~\hspace*{-0.1mm}\mbox{every}~\hspace*{-0.1mm}${\boldsymbol{u}_n,\boldsymbol{z}_n\hspace*{-0.15em}\in\hspace*{-0.15em} \bXpn}$,~${n\hspace*{-0.15em}\in\hspace*{-0.15em}  \mathbb{N}}$, we have that\vspace*{-0.5mm}
    	\begin{align*}
    		\vert\langle \mathbfcal{C}_n\boldsymbol{u}_n,\boldsymbol{z}_n\rangle_{\smash{\bXpn}}\vert&\leq \|\boldsymbol{u}_n\|_{2 p_c',Q_T}^2\|\mathbf{D}_x\boldsymbol{z}_n\|_{p_c,Q_T}  +\|\boldsymbol{u}_n\|_{2p_c',Q_T}\|\boldsymbol{z}_n\|_{2p_c',Q_T}\|\mathbf{D}_x\boldsymbol{u}_n\|_{p_c,Q_T}
    		\\&\leq 8\,(1+\vert Q_T\vert)^3\,\|\boldsymbol{u}_n\|_{q_n(\cdot,\cdot),Q_T}^2\|\mathbf{D}_x\boldsymbol{z}_n\|_{p_n(\cdot,\cdot),Q_T}\\&\quad+ 8\,(1+\vert Q_T\vert)^3\,\|\boldsymbol{u}_n\|_{q_n(\cdot,\cdot),Q_T}\|\boldsymbol{z}_n\|_{q_n(\cdot,\cdot),Q_T}\|\mathbf{D}_x\boldsymbol{u}_n\|_{p_n(\cdot,\cdot),Q_T}
    		\\&\leq 16\,(1+\vert Q_T\vert)^3\,\|\boldsymbol{u}_n\|_{\smash{\bXpn}}^2\|\boldsymbol{z}_n\|_{\smash{\bXpn}}\,,
    	\end{align*}
    	which implies the well-definedness and boundedness of $\mathbfcal{C}_n\colon\bXpn\to (\bXpn)^*$~for~all~$n\in  \mathbb{N}$.
    	
    	\textit{2. Non-conforming Bochner strong continuity.} Let $\boldsymbol{v}_n\hspace*{-0.1em}\in\hspace*{-0.1em} \bXpn\cap L^\infty(I;Y)$, $n\hspace*{-0.1em}\in\hspace*{-0.1em}\mathbb{N}$, be a sequence~with $\boldsymbol{v}_n(t)\hspace*{-0.15em}\in\hspace*{-0.15em} \Vo_{h_n,0}$ for a.e.\  $t\hspace*{-0.15em}\in\hspace*{-0.15em} I$ \hspace*{-0.1mm}and \hspace*{-0.1mm}all \hspace*{-0.1mm}$n\hspace*{-0.15em}\in \hspace*{-0.15em}\mathbb{N}$ \hspace*{-0.1mm}as \hspace*{-0.1mm}well \hspace*{-0.1mm}as  \hspace*{-0.1mm}satisfying \hspace*{-0.1mm}\eqref{def:non-conform_Bochner_cond_M.1}--\eqref{def:non-conform_Bochner_cond_M.3} \hspace*{-0.1mm}with~\hspace*{-0.1mm}\mbox{respect}~\hspace*{-0.1mm}to~\hspace*{-0.1mm}${\boldsymbol{v}\hspace*{-0.15em}\in\hspace*{-0.15em} \bVp\hspace*{-0.15em}\cap\hspace*{-0.15em} L^\infty(I;H)}$. Due to Proposition \ref{prop:Xqp_weak_compact} and  Proposition \ref{prop:non_conform_landes}, 
    	for 
    	every $r,s\in \mathcal{P}^\infty(Q_T)$ with $r\prec q$ in $Q_T$~and~${s\prec p}$~in~$Q_T$ as well as for every  $\varepsilon\in(0,(p^-)^*-1]$, it holds that\vspace*{-1mm}
    	\begin{alignat}{3} 
    			\boldsymbol{v}_n &\rightharpoonup \boldsymbol{v}&&\quad\textup{ in }\mathbfcal{U}_{\D}^{r,s}&&\quad(n\to \infty )\,,\label{prop:convective_term.1}	\\[-0.5mm]
    			\boldsymbol{v}_n &\to \boldsymbol{v}&&\quad\textup{ in }L^{p_*(\cdot,\cdot)-\varepsilon}(Q_T;\mathbb{R}^d)&&\quad(n\to \infty )\,. \label{prop:convective_term.2}	
    	\end{alignat}
    	Then, due to $p^->p_c$, for $r=s=p_c$, by Korn's inequality and the variable Hölder inequality \eqref{eq:gen_hoelder}, from \eqref{prop:convective_term.1} and for $\varepsilon=(p^-)_*-(p_c)_*>0$, by the variable Hölder inequality \eqref{eq:gen_hoelder}, from \eqref{prop:convective_term.2}, it~follows~that
    	\begin{align}\label{prop:convective_term.3}
    		\begin{aligned}
    			\nabla_{\!x}\boldsymbol{v}_n &\rightharpoonup \nabla_{\!x}\boldsymbol{v}&&\quad\textup{ in }L^{p_c}(Q_T;\mathbb{R}^{d\times d})&&\quad(n\to \infty )\,,\\[-0.5mm]
    			\boldsymbol{v}_n& \to \boldsymbol{v}&&\quad\textup{ in }L^{\smash{2p_c'}}(Q_T;\mathbb{R}^d)&&\quad(n\to \infty )\,.
    		\end{aligned}
    	\end{align}
    	Next, let $\boldsymbol{z}_n\hspace*{-0.15em}\in\hspace*{-0.15em} \bXpn\cap L^\infty(I;Y)$, $n\hspace*{-0.15em}\in\hspace*{-0.15em}\mathbb{N}$, be sequence with $\boldsymbol{z}_n(t)\hspace*{-0.15em}\in\hspace*{-0.15em} \Vo_{h_n,0}$ for a.e.\  $t\hspace*{-0.15em}\in \hspace*{-0.15em}I$ and all $n\hspace*{-0.15em}\in\hspace*{-0.15em} \mathbb{N}$~as~well~as satisfying \eqref{def:non-conform_Bochner_cond_M.1}--\eqref{def:non-conform_Bochner_cond_M.3} with respect to  $\boldsymbol{z}\in \bVp\cap L^\infty(I;H)$. 
    	Then,~analogously~to~\eqref{prop:convective_term.1}--\eqref{prop:convective_term.3},~we~obtain
    	\begin{align}\label{prop:convective_term.6}
    		\begin{aligned}
    			\nabla_{\!x}\boldsymbol{z}_n &\rightharpoonup \nabla_{\!x}\boldsymbol{z}&&\quad\textup{ in }L^{p_c}(Q_T;\mathbb{R}^{d\times d})&&\quad(n\to \infty )\,,\\[-0.5mm]
    			\boldsymbol{z}_n& \to \boldsymbol{z}&&\quad\textup{ in }L^{\smash{2p_c'}}(Q_T;\mathbb{R}^d)&&\quad(n\to \infty )\,.
    		\end{aligned}
    	\end{align}
    	Eventually, from \eqref{prop:convective_term.3} and \eqref{prop:convective_term.6}, it follows that $	\lim_{n\to \infty}{\langle \mathbfcal{C}_n\boldsymbol{v}_n,\boldsymbol{z}_n\rangle_{\smash{\bXpn}}}=\langle  \mathbfcal{C}\boldsymbol{v},\boldsymbol{z}\rangle_{\smash{\bVp}}$. 
    \end{proof}
    
    Since the non-conforming Bochner condition (M) is stable under non-conforming Bochner strongly continuous perturbations (\textit{cf}.\ Proposition \ref{prop:perturbation}), 
    from Proposition \ref{prop:stress} and Proposition \ref{prop:convective_term}, we can derive a further example of a sequence of operators that satisfies the non-conforming Bochner condition (M).
    
    \begin{corollary}\label{prop:stress_plus_convective_term}	
    Let  $p^->p_c\coloneqq\frac{3d+2}{d+2}$ and $q_n\ge (p_c)_*$ a.e.\ in $Q_T$ for all $n\in \mathbb{N}$. 
    	Then, the sequence of   operators  $\mathbfcal{S}_n+\mathbfcal{C}_n\colon \bXpn\cap L^\infty(I;Y)\to (\bXpn)^*$, $n\in \mathbb{N}$, satisfies the non-conforming Bochner condition (M) with respect to $(\Vo_{h_n,0})_{n\in \mathbb{N}}$ and $\mathbfcal{S}+\mathbfcal{C}\colon \bVp\cap L^\infty(I;H)\to (\bVp)^*$.
    \end{corollary}
    
    \begin{proof}
    	Given Proposition \ref{prop:stress} and Proposition \ref{prop:convective_term}, the assertion follows by means of Proposition \ref{prop:perturbation}.
    \end{proof} \newpage

    \section{Fully discrete Rothe--Galerkin approximation}\label{sec:rothe_galerkin}
	\hspace{5mm}In this section, we specify the exact setup of the fully-discrete Rothe--Galerkin approximation~of~\eqref{eq:generalized_evolution_equation}. Moreover, we prove its well-posedness (\textit{i.e.}, the existence of iterates), its stability (\textit{i.e.},~the~boundedness of the corresponding double
	sequences of piece-wise constant and piece-wise affine interpolants), and its (weak) convergence (\textit{i.e.}, the weak convergence of diagonal subsequences towards
	a solution of \eqref{eq:generalized_evolution_equation}).\enlargethispage{11mm} 
	
	\begin{assumption} \label{assumption} In addition to the Assumptions \ref{assum:p_h}, \ref{ass:PiY}, and \ref{ass:proj-div}, 
		we make the following assumptions:
		\begin{description}[leftmargin=!,labelwidth=\widthof{\itshape(iii)},noitemsep,topsep=2pt,font=\normalfont\itshape]
			\item[(i)] \textup{Variable exponents:} Let $p\in \mathcal{P}^{\log}(Q_T)$ with $p^-\ge 2$ and $q\coloneqq p_*-\varepsilon\in \mathcal{P}^{\log}(Q_T)$~with~${\varepsilon\in (0,\frac{2}{d}p^-]}$, \textit{i.e.}, $q\ge p$ in $Q_T$. Moreover, let $p_n\ge p^-$ for all $n\in \mathbb{N}$ and  $q_n\coloneqq (p_n)_*-\varepsilon\in \mathcal{P}^{\infty}(Q_T)$~for~all~${n\in \mathbb{N}}$;
			\item[(ii)] \textup{Initial data:} Let $\mathbf{v}_0\in H$ and let  $\mathbf{v}_n^0\in \Vo_{h_n,0}$, $n\in\mathbb{N}$, be a sequence such that 
            \begin{align*}
                \sup_{n\in \mathbb{N}}{\|\mathbf{v}_n^0\|_{2,\Omega}}\leq \|\mathbf{v}_0\|_{2,\Omega}\quad\text{ and }\quad
                \mathbf{v}_n^0\to \mathbf{v}_0\quad\text{ in }Y\quad (n\to \infty)\,;
            \end{align*}
			\item[(iii)] \textup{Right-hand side:} Let $\boldsymbol{f}^*\in  (\bVp)^*$ and let $\boldsymbol{f}^*_n\in  (\bXpn)^*$, $n\in \mathbb{N}$, be a sequence~such~that~there~\mbox{exist} $(\boldsymbol{f}_n)_{n\in \mathbb{N}}\subseteq  L^{(p^-)'}(Q_T;\mathbb{R}^d)$ and  $(\boldsymbol{F}_n)_{n\in \mathbb{N}}\subseteq  L^{p_n'(\cdot,\cdot)}(Q_T;\mathbb{R}^{d\times d}_{\mathrm{sym}})$ with  $\boldsymbol{f}^*_n=\mathbfcal{J}_{\D}(\boldsymbol{f}_n,\boldsymbol{F}_n)$ in $(\bXpn)^*$ and\vspace*{-1.5mm} 
            \begin{align*}
                \sup_{n\in \mathbb{N}}{\big[\|\boldsymbol{f}_n\|_{(p^-)',Q_T}+\|\boldsymbol{F}_n\|_{p_n'(\cdot,\cdot),Q_T}\big]}<\infty\,,
            \end{align*}
            and for a sequence $\boldsymbol{u}_n\in\bXpn\cap L^\infty(I;Y) $, $n\in \mathbb{N}$, from \eqref{def:non-conform_Bochner_cond_M.1}--\eqref{def:non-conform_Bochner_cond_M.3}, it follows that\vspace*{-1mm}
            \begin{align*}
                \langle \boldsymbol{f}^*_n,\boldsymbol{u}_n\rangle_{\smash{\bXpn}}\to \langle \boldsymbol{f}^*,\boldsymbol{u}\rangle_{\smash{\bVp}}\quad (n\to \infty)\,;
            \end{align*}
   
			\item[(iv)] \textup{Operators:} Let $A_n(t)\colon U^{q_n,p_n}_{\D}(t)\to  (U^{q_n,p_n}_{\D}(t))^*$, $t\in  I$,  $n\in \mathbb{N}$, be operator families satisfying (\hyperlink{A.1}{A.1})--(\hyperlink{A.3}{A.3}), where the non-decreasing mapping $\mathbb{B}\colon\mathbb{R}_{\ge 0}\to \mathbb{R}_{\ge 0}$ in (\hyperlink{A.3}{A.3})  is the same for all $n\in \mathbb{N}$, with~respect~to $(q_n)_{n\in \mathbb{N}},(p_n)_{n\in \mathbb{N}}\subseteq \mathcal{P}^\infty(Q_T)$ and, in addition, the following~\mbox{semi-coercivity}~\mbox{condition}:
				\begin{description}[leftmargin=!,labelwidth=\widthof{\itshape(A.4)}, topsep=2pt,font=\normalfont\itshape]
			\item[(A.4)] \hypertarget{A.4}{} There exist constants $c_0>0$ and $c_1,c_2\ge 0$ such that for a.e.\   $t\in I$ and every $\mathbf{u}_n\in U^{q_n,p_n}_{\D}(t)$, it holds that\vspace*{-1.5mm}
			\begin{align*}
				\langle A_n(t)\mathbf{u}_n,\mathbf{u}_n\rangle_{\smash{U^{q_n,p_n}_{\D}(t)}}\geq
				c_0\,\rho_{p_n(t,\cdot),\Omega}(\mathbf{D}_x\mathbf{u}_n)-c_1\,\|\mathbf{u}_n\|_{2,\Omega}^2-c_2 \,;
			\end{align*}
				\end{description}
			
			 The sequence of  induced (\textit{cf}.\ \eqref{def:induced})   operators $\mathbfcal{A}_n\colon \hspace*{-0.15em}\bXpn\cap L^\infty(I;Y)\hspace*{-0.15em}\to \hspace*{-0.15em} (\bXpn)^*$, $n\hspace*{-0.15em}\in\hspace*{-0.15em} \mathbb{N}$,~\mbox{satisfies}~the non-conforming Bochner condition (M) with respect to $(\Vo_{h_n,0})_{n\in\mathbb{N}}$ and ${\mathbfcal{A}\colon\hspace*{-0.175em} \bVp\hspace*{-0.15em}\cap\hspace*{-0.15em}  L^\infty(I;H)\hspace*{-0.175em} \to\hspace*{-0.15em} (\bVp)^*}$.
		\end{description}
	\end{assumption}

	 \hspace*{-1mm}We \hspace*{-0.1mm}set \hspace*{-0.1mm}$H_{h_n,0}\hspace*{-0.15em}\coloneqq \hspace*{-0.15em}\Vo_{h_n,0}$, \hspace*{-0.1mm}if \hspace*{-0.1mm}equipped \hspace*{-0.1mm}with \hspace*{-0.1mm}$(\cdot,\cdot)_{\Omega}$, \hspace*{-0.1mm}and \hspace*{-0.1mm}denote \hspace*{-0.1mm}by \hspace*{-0.1mm}$R_{H_{h_n,0}}\hspace*{-0.15em}\colon \hspace*{-0.15em}H_{h_n,0}\hspace*{-0.15em}\to\hspace*{-0.15em}  (H_{h_n,0})^*$~\hspace*{-0.1mm}the~\hspace*{-0.1mm}\mbox{respective} Riesz isomorphism. Then, the triple $(\Vo_{h_n,0},H_{h_n,0},\mathrm{id}_{\smash{\Vo_{h_n,0}}})$ is an evolution triple and the \textit{canonical embedding} $e_n\coloneqq (\mathrm{id}_{\smash{\Vo_{h_n,0}}})^*\circ R_{H_{h_n,0}}\circ \mathrm{id}_{\smash{\Vo_{h_n,0}}}\colon \Vo_{h_n,0}\to  (\Vo_{h_n,0})^*$,   for every $\mathbf{u}_n,\mathbf{z}_n\in \Vo_{h_n,0}$, satisfies 
	\begin{align}
	\langle e_n\mathbf{u}_n,\mathbf{z}_n\rangle_{\smash{\Vo_{h_n,0}}}=(\mathbf{u}_n,\mathbf{z}_n)_\Omega\,.\label{eq:iden2}
	\end{align}

	Combining the time and space discretizations in Section \ref{sec:discrete_ptxNavierStokes} with the framework of Assumption \ref{assumption},
 leads to the following fully-discrete  Rothe--Galerkin approximation of \eqref{eq:generalized_evolution_equation}.\medskip
	
	\begin{algorithm}[Fully-discrete, implicit scheme]\label{scheme}
		For given $K,n\in \mathbb{N}$, the  iterates $(\mathbf{v}_n^k)_{k=1,\ldots ,K}\subseteq \Vo_{h_n,0}$, for every $k=1,\ldots, K$ and $\mathbf{z}_{h_n}\in \Vo_{h_n,0}$, are defined via\vspace*{-0.5mm}
		\begin{align}
		(\mathrm{d}_\tau \mathbf{v}_n^k,\mathbf{z}_{h_n})_\Omega+\langle \langle A_n\rangle_k   \mathbf{v}_n^k,\mathbf{z}_{h_n}\rangle_{\smash{U^{q_n,p_n}_{\plus}}}= \langle \langle\boldsymbol{f}^*_n\rangle_k,\mathbf{z}_{h_n}\rangle_{\smash{U^{q_n,p_n}_{\plus}}}\,.\label{eq:scheme}
		\end{align}
	\end{algorithm}\medskip
	
	For small step-sizes, Scheme  \ref{scheme} is well-posed, \textit{i.e.}, the existence of all iterates solving \eqref{eq:scheme}~is~ensured. 
	
	\begin{proposition}[Well-posedness]\label{prop:well-posed}
		For every $K,n\in \mathbb{N}$  with $\tau< \frac{1}{4c_1}$, there~exist~${(\mathbf{v}_n^k)_{k=1,\ldots ,K}\subseteq \Vo_{h_n,0}}$  solving \eqref{scheme}.
	\end{proposition}
	
	\begin{proof}
		Using \eqref{eq:iden2} and the adjoint  of the identity $\mathrm{id}_{\smash{\Vo_{h_n,0}}}\colon \Vo_{h_n,0}\to U^{q,p}_{\plus}$, we find that~\eqref{scheme}~is~\mbox{equivalent}~to
		\begin{align}
		(\mathrm{id}_{\smash{\Vo_{h_n,0}}})^*\langle\boldsymbol{f}^*_n\rangle_k +\tau^{-1}e_n\mathbf{v}_n^{k-1}\in\textrm{im}(F_n)\,,\label{prop:well-posed.1}
		\end{align}
		where $F_n\coloneqq \tau^{-1} e_n+(\mathrm{id}_{\smash{\Vo_{h_n,0}}})^*\circ \langle A_n\rangle_k  \circ\mathrm{id}_{\smash{\Vo_{h_n,0}}} \colon \Vo_{h_n,0}\to (\Vo_{h_n,0})^*$.
	 	Since $F_n \colon \Vo_{h_n,0}\to (\Vo_{h_n,0})^*$~is~pseudo-monotone, bounded (\textit{cf.}~Proposition~\ref{rothe3}(i)), and coercive, because, owing to \eqref{eq:iden2},  (\hyperlink{A.4}{A.4}), and $\tau<\frac{1}{4c_1}$, for~every~$\smash{\mathbf{z}_{h_n}\in \Vo_{h_n,0}}$, it holds that $	\langle F_n\mathbf{z}_{h_n},\mathbf{z}_{h_n}\rangle_{\smash{\Vo_{h_n,0}}}
	 	\ge 3c_1\|\mathbf{z}_{h_n}\|_{2,\Omega}^2-c_2$, 
		 the main theorem on pseudo-monotone operators (\textit{cf.}~\cite[Thm.\ 12.A]{zei-IIB})  proves the solvability of \eqref{prop:well-posed.1}.
	\end{proof}
	
	\newpage

	For small step-sizes, Scheme  \ref{scheme} is stable, \textit{i.e.}, the double sequences of piece-wise constant~and~piece-wise affine interpolants generated by iterates solving \eqref{eq:scheme} satisfy a priori estimates.\enlargethispage{6mm}

	\begin{proposition}[Stability]\label{apriori} There exists a constant $M>0$, not depending on $K,n\in \mathbb{N}$, such that the piece-wise constant~interpolants $\overline{\boldsymbol{v}}_n^\tau \in \mathbb{P}^0(\mathcal{I}_\tau;\Vo_{h_n,0})$, ${K,n\in  \mathbb{N}}$, and the piece-wise affine interpolants $\widehat{\boldsymbol{v}}_n^\tau\in \mathbb{P}^1_c(\mathcal{I}_\tau;\Vo_{h_n,0}) $, $K,n\in  \mathbb{N}$, (\textit{cf}.\ \eqref{eq:polant}) generated by iterates $(\mathbf{v}_n^k)_{k=1,\ldots ,K}\subseteq \Vo_{h_n,0}$, $K,n\in  \mathbb{N}$,~solving~\eqref{scheme}, for every $K,n\in \mathbb{N}$  with $\tau<\frac{1}{4c_1}$ satisfy the following a priori estimates:
		\begin{align}
		\|\overline{\boldsymbol{v}}_n^\tau\|_{\bXpn\cap L^\infty(I;Y)}&\leq M\,,\label{eq:ap.1}\\
	\|\widehat{\boldsymbol{v}}_n^\tau\|_{L^1((\tau,T);U^{q,p}_{\minus})}+\|\widehat{\boldsymbol{v}}_n^\tau\|_{L^\infty(I;Y)}&\leq M\,,\label{eq:ap.2}\\
		\|\mathbfcal{A}_n\overline{\boldsymbol{v}}_n^\tau\|_{(\bXpn)^*}&\leq M\,,\label{eq:ap.3}\\
		\|e_n[\widehat{\boldsymbol{v}}_n^\tau-\overline{\boldsymbol{v}}_n^\tau]\|_{L^{\min\{(q_n^+)',(p_n^+)'\}}(I;(\Vo_{h_n,0})^*)}&\leq  \tau\, M\,,\label{eq:ap.4}
		\end{align}
		where, for every $n\in \mathbb{N}$, we equip the space $\Vo_{h_n,0}$ with the norm $\|\cdot\|_{q^+_n,p^+_n,\Omega}$ in \eqref{eq:ap.4}.\enlargethispage{6mm}
	\end{proposition}
	
	\begin{proof}
        \textit{ad \eqref{eq:ap.1}.}
		In \eqref{scheme}, for every 
		$k\hspace*{-0.1em}=\hspace*{-0.1em}1,\ldots ,\ell $, $\ell \hspace*{-0.1em}=\hspace*{-0.1em}1,\ldots ,K$,
		we~choose~$\mathbf{z}_{h_n}\hspace*{-0.1em}=\hspace*{-0.1em}\mathbf{v}_n^k\hspace*{-0.1em}\in\hspace*{-0.1em} \Vo_{h_n,0}$, multiply~the resulting equation by $\tau > 0$, sum with respect to $k=1,\ldots ,\ell $, use the discrete~\mbox{integration-by-parts}~inequality \eqref{eq:4.2}, and $\sup_{n\in \mathbb{N}}{\|\mathbf{v}^0_n\|_{2,\Omega}\leq\|\mathbf{v}_0\|_{2,\Omega}}$ (\textit{cf}.\ Assumption \ref{assumption}(i)), to find that for every $\ell =1,\ldots ,K$, it holds that
		\begin{align}\begin{split}
		\frac{1}{2}\,\|\mathbf{v}_n^{\ell}\|_{2,\Omega}^2+\sum_{k=1}^{\ell}{\tau\,\langle\langle A_n\rangle_k  \mathbf{v}_n^k,\mathbf{v}_n^k\rangle_{\smash{U^{q_n,p_n}_{\plus}}}}\leq \frac{1}{2}\,\|\mathbf{v}_0\|_{2,\Omega}^2+
		\sum_{k=1}^{\ell}{\tau\,\langle\langle \boldsymbol{f}^*_n\rangle_k ,\mathbf{v}_n^k\rangle_{\smash{U^{q_n,p_n}_{\plus}}}}\,.
		\end{split}
		\label{eq:apriori.1}
		\end{align}
		Using in \eqref{eq:apriori.1}, definitions \eqref{def:temp_mean}--\eqref{def:l2_proj2}, Proposition~\ref{rothe3}(iii.a), Proposition~\ref{rothe2}(i), the 
		$\kappa$-Young inequality (\textit{cf}.\ \hspace*{-0.1mm}\cite[Prop.\ \hspace*{-0.1mm}2.8]{alex-book}), \hspace*{-0.1mm}and \hspace*{-0.1mm}that, \hspace*{-0.1mm}due \hspace*{-0.1mm}to \hspace*{-0.1mm}Poincar\'e's \hspace*{-0.1mm}and \hspace*{-0.1mm}Korn's \hspace*{-0.1mm}inequality, \hspace*{-0.1mm}for \hspace*{-0.1mm}a.e.\ \hspace*{-0.1mm}$t\hspace*{-0.15em}\in\hspace*{-0.15em} I$~\hspace*{-0.1mm}and~\hspace*{-0.1mm}\mbox{every}~\hspace*{-0.1mm}${\mathbf{u}_n\hspace*{-0.15em}\in\hspace*{-0.15em} U^{q_n,p_n}_{\D}(t)}$, $n\in\mathbb{N}$, due to $p_n\geq p^-$ a.e.\  in $Q_T$, it holds that
		$	\rho_{p^-,\Omega}(\mathbf{u}_n)\leq c_p\,\rho_{p^-,\Omega}(\mathbf{D}_x\mathbf{u}_n)
		\leq c_p\,(1+\rho_{p_n(t,\cdot),\Omega}(\mathbf{D}_x\mathbf{u}_n))$, 
		for every $\ell =1,\ldots ,K$, we deduce that
		\begin{align}\label{eq:apriori.2}
            \begin{aligned}
		\sum_{k=1}^{\ell}{\tau\,\langle\langle\boldsymbol{f}_n^*\rangle_k ,\mathbf{v}_n^k\rangle_{\smash{U^{q_n,p_n}_{\plus}}}}
&=\langle\Pi^{0,\mathrm{t}}_{\tau}[\boldsymbol{f}_n^*],\overline{\boldsymbol{v}}_n^\tau\chi_{\left[0,\ell\tau\right]}\rangle_{\mathbfcal{U}^{q_n,p_n}_{\plus}}\\[-3mm]&=\langle \boldsymbol{f}^*_n,\Pi^{0,\mathrm{t}}_{\tau}[\overline{\boldsymbol{v}}_n^\tau\chi_{\left[0,\ell\tau\right]}]\rangle_{\smash{\bXpn}}\\&=\langle \boldsymbol{f}^*_n,\overline{\boldsymbol{v}}_n^\tau\chi_{\left[0,\ell\tau\right]}\rangle_{\smash{\bXpn}}
		\\&\leq c_\kappa\,\big(\rho_{(p^-)',Q_T}(\boldsymbol{f}_n\chi_{\left[0,\ell\tau\right]})+\rho_{p_n'(\cdot,\cdot),Q_T}(\boldsymbol{F}_n\chi_{\left[0,\ell\tau\right]})\big)\\&\quad+\kappa\,\big(\rho_{p^-,Q_T}(\overline{\boldsymbol{v}}^\tau_n\chi_{\left[0,\ell\tau\right]})
  +\rho_{p_n(\cdot,\cdot),Q_T}(\mathbf{D}_x\overline{\boldsymbol{v}}^\tau_n\chi_{\left[0,\ell\tau\right]})\big)
		\\&\leq c_\kappa\,\sup_{n\in \mathbb{N}}{\big[\rho_{(p^-)',Q_T}(\boldsymbol{f}_n)+\rho_{p_n'(\cdot,\cdot),Q_T}(\boldsymbol{F}_n)\big]}\\[-1mm]&\quad+\kappa\,c_p\,\big(1+\rho_{p_n(\cdot,\cdot),Q_T}(\mathbf{D}_x\overline{\boldsymbol{v}}^\tau_n\chi_{\left[0,\ell\tau\right]})\big)\,.
            \end{aligned}
		\end{align}
		In addition, using definition \eqref{def:temp_mean} and (\hyperlink{A.4}{A.4}) (\textit{cf}.\ Assumption \ref{assumption}(iv)), for every $\ell =1,\ldots ,K$,~we~find~that 
		\begin{align}\label{eq:apriori.3}
			\begin{aligned}
				\sum_{k=1}^{\ell}{\tau\,\langle\langle A_n\rangle_k  \mathbf{v}_n^k,\mathbf{v}_n^k\rangle_{\smash{U^{q_n,p_n}_{\plus}}}}&=\sum_{k=1}^{\ell}{\int_{I_k}{\langle A_n(t)  \mathbf{v}_n^k,\mathbf{v}_n^k\rangle_{\smash{U^{q_n,p_n}_{\D}(t)}}\,\mathrm{d}t}}\\&\ge
				c_0\,\rho_{p_n(\cdot,\cdot),Q_T}(\mathbf{D}_x\overline{\boldsymbol{v}}^\tau_n\chi_{\left[0,\ell\tau\right]})-\sum_{k=1}^{\ell}{\tau\, c_1\,\|\mathbf{v}_n^k\|_{2,\Omega}^2}-c_2\,\ell\,\tau\,.
			\end{aligned}
		\end{align}
		As a result, combining \eqref{eq:apriori.2} and \eqref{eq:apriori.3} in \eqref{eq:apriori.1}, for every  $\ell =1,\ldots ,K$,  we~arrive~at~the~inequality
		\begin{align}\label{eq:apriori.3.2}
			\begin{aligned} 
			\Big(\frac{1}{2}-c_1\,\tau\Big)\,\|\mathbf{v}_n^{\ell}\|_{2,\Omega}^2&+(c_0-\kappa\,c_p)\,\rho_{p_n(\cdot,\cdot),Q_T}(\mathbf{D}_x\overline{\boldsymbol{v}}^\tau_n\chi_{\left[0,\ell\tau\right]})\\&\leq \frac{1}{2}\,\|\mathbf{v}_0\|_{2,\Omega}^2+c_\kappa\,\sup_{n\in \mathbb{N}}{\big[\rho_{(p^-)',Q_T}(\boldsymbol{f}_n)+\rho_{p_n'(\cdot,\cdot),Q_T}(\boldsymbol{F}_n)\big]}
			\\&\quad +\sum_{k=1}^{\ell-1}{c_1\,\tau\, \|\mathbf{v}_n^k\|_{2,\Omega}^2}+c_2\,\ell\,\tau+\kappa\,c_p\,.
				\end{aligned}
		\end{align} 
		We set 
        \begin{align*}
           \kappa&\coloneqq  \frac{c_0}{2c_p}\,,\\
            \alpha&\coloneqq \frac{1}{2}\,\|\mathbf{v}_0\|_H^2+ c_\kappa\,\sup_{n\in\mathbb{N}}\big[\rho_{(p^-)',Q_T}(\boldsymbol{f}_n)+\rho_{p_n'(\cdot,\cdot),Q_T}(\boldsymbol{F}_n)\big]+c_2\,T+\frac{c_0}{2}\,,\\
            \beta&\coloneqq 4\,\tau\, c_1\,,
        \end{align*}
         and 
         \begin{align*}
             y^k_n\coloneqq \frac{1}{4}\,\|\mathbf{v}^k_n\|_{2,\Omega}^2+\frac{c_0}{2}\,\rho_{p_n(\cdot,\cdot),Q_T}(\mathbf{D}_x\overline{\boldsymbol{v}}^\tau_n\chi_{\left[0,k\tau\right]})\quad\text{ for all }k=1,\ldots,K\,.
         \end{align*}
        Then,  from \eqref{eq:apriori.3.2}, for every $\ell =1,\ldots,K$, we infer   that
		\begin{align}
		y^{\ell}_n\leq \alpha+\beta\,\sum_{k=1}^{\ell-1}{y^k_n}\,.\label{eq:apriori.4}
		\end{align}
		As a result, the discrete Gr\"onwall lemma (\textit{cf.}~\cite[Lem.\ 2.2]{Ba15}) applied to \eqref{eq:apriori.4} yields that
		\begin{align}
			\begin{aligned}
				\frac{1}{4}\,\|\overline{\boldsymbol{v}}_n^\tau\|_{L^\infty(I;Y)}^2+\frac{c_0}{2}\,\rho_{p_n(\cdot,\cdot),Q_T}(\mathbf{D}_x\overline{\boldsymbol{v}}^\tau_n)&\leq 2\, \max_{k=1,\ldots ,K}{y^k_n}\\&\leq \alpha\,\exp(K\,\beta)=\alpha\,\exp(4\,T\,c_1)\,.
			\end{aligned}\label{eq:apriori.5}
		\end{align}
		In addition, Proposition~\ref{prop:non_conform_poincare} yields a constant $c_\varepsilon=c_\varepsilon(p,\Omega)>0$  such that
		\begin{align}
			\|\overline{\boldsymbol{v}}_n^\tau\|_{q_n(\cdot,\cdot),Q_T}\leq c_\varepsilon\,\big(\|\mathbf{D}_x\overline{\boldsymbol{v}}_n^\tau\|_{p_n(\cdot,\cdot),Q_T}+\|\overline{\boldsymbol{v}}_n^\tau\|_{L^\infty(I;Y)}\big)\,.\label{eq:apriori.7}
		\end{align}
		Eventually, from \eqref{eq:apriori.5} and \eqref{eq:apriori.7}, we conclude that the claimed a priori estimate  \eqref{eq:ap.1} applies.
  
        \textit{ad \eqref{eq:ap.2}.} According to \cite[ineq.\ (8.29)]{Rou13}, we have that
        \begin{align*}
            \|\widehat{\boldsymbol{v}}_n^\tau\|_{L^\infty(I;Y)}&\leq \|\overline{\boldsymbol{v}}_n^\tau\|_{L^\infty(I;Y)}\,,\\
             \|\widehat{\boldsymbol{v}}_n^\tau\|_{L^1((\tau,T);U^{q,p}_{\minus})}&\leq \|\overline{\boldsymbol{v}}_n^\tau\|_{L^1(I;U^{q,p}_{\minus})}\,,
        \end{align*}
        so that  the claimed a priori estimate  \eqref{eq:ap.2} follows from the a priori estimate  \eqref{eq:ap.1}.

        \textit{ad \eqref{eq:ap.3}.} Following along the lines of the proof of \cite[Prop.\  5.13, \textit{i.e.}, (5.20) and below]{alex-book},~we~find~that
        \begin{align*}
        \|\mathbfcal{A}_n\overline{\boldsymbol{v}}_n^{\tau}\|_{(\bXpn)^*}\leq \mathbb{B}(\|\overline{\boldsymbol{v}}_n^\tau\|_{L^\infty(I;Y)})\,\big(T+\rho_{q_n(\cdot,\cdot),Q_T}(\overline{\boldsymbol{v}}_n^\tau)+\rho_{p_n(\cdot,\cdot),Q_T}(\mathbf{D}_x\overline{\boldsymbol{v}}_n^\tau)+2\big)\,,
        \end{align*}
        so that the claimed a priori estimate  \eqref{eq:ap.3} follows from the a priori estimate  \eqref{eq:ap.1}.

        \textit{ad \eqref{eq:ap.4}.} Since for every $t\in I_k$, $k=1,\ldots ,K$, we have that
        \begin{align*}
            \left.\begin{aligned}
            e_n(\widehat{\boldsymbol{v}}_n^\tau(t)-\overline{\boldsymbol{v}}_n^\tau(t))&=(t-k\tau)\, e_n\Big(\frac{\mathrm{d}\widehat{\boldsymbol{v}}_n^\tau}{\mathrm{dt}}(t)\Big)\\&=(t-k\tau)\,\big(
           \Pi^{0,\mathrm{t}}_{\tau}[\boldsymbol{f}^*_n](t)- \Pi^{0,\mathrm{t}}_{\tau}[A_n](t)\big)
             \end{aligned}\quad\right\} \quad\text{ in }(\Vo_{h_n,0})^*\,,
        \end{align*}
        as well as  $\vert t-k\tau\vert \leq \tau$, using Proposition \ref{rothe3}(iii.b), 
        we conclude that
		\begin{align*}
            \begin{aligned}
		\|e_n(\widehat{\boldsymbol{v}}_n^\tau-\overline{\boldsymbol{v}}_n^\tau)\|_{L^{\min\{(q^+_n)',(p^+_n)'\}}(I;(\Vo_{h_n,0})^*)}&\leq\tau\,\Big\|e_n\Big(\frac{\mathrm{d}\widehat{\boldsymbol{v}}_n^\tau}{\mathrm{d}t}\Big)\Big\|_{L^{\min\{(q^+_n)',(p^+_n)'\}}(I;(\Vo_{h_n,0})^*)}
		\\&=\sup_{\substack{\boldsymbol{z}_n\in L^{\max\{q^+_n,p^+_n\}}(I;\Vo_{h_n,0})\\\|\boldsymbol{z}_n\|_{\mathbfcal{U}_{\plus}^{q_n,p_n}}\leq 1}}{\tau\,
		\langle \Pi^{0,\mathrm{t}}_{\tau}[\boldsymbol{f}^*_n]-\Pi^{0,\mathrm{t}}_{\tau}[\mathbfcal{A}_n]\overline{\boldsymbol{v}}_n^\tau,\boldsymbol{z}_n\rangle_{\mathbfcal{U}_{\plus}^{q_n,p_n}}}
		\\&\leq  \tau\,\|\Pi^{0,\mathrm{t}}_{\tau}[\boldsymbol{f}^*_n]-\Pi^{0,\mathrm{t}}_{\tau}[\mathbfcal{A}_n]\overline{\boldsymbol{v}}_n^\tau\|_{(\mathbfcal{U}_{\plus}^{q_n,p_n})^*}
			\\&\leq  \tau\,2\,(1+\vert Q_T\vert)\|\boldsymbol{f}^*_n-\mathbfcal{A}_n\overline{\boldsymbol{v}}_n^\tau\|_{(\mathbfcal{U}_{\D}^{q_n,p_n})^*}\,,
            \end{aligned}
		\end{align*}
		so that the claimed a priori estimate  \eqref{eq:ap.4} follows from  the a priori estimate \eqref{eq:ap.3} and Assumption~\ref{assumption}(iii). 
	\end{proof}

    \newpage

    Scheme \ref{scheme} is (at least) weakly convergent, \textit{i.e.}, the double sequences of piece-wise constant and piece-wise affine interpolants generated by iterates solving \eqref{eq:scheme} (at least)  weakly converge~(up~to~a~\mbox{subsequence})
    to a solution of the  generalized evolution equation \eqref{eq:generalized_evolution_equation}.\enlargethispage{7mm}
	
	\begin{theorem}[Weak convergence]\label{thm:main}
		Let $\overline{\boldsymbol{v}}_n\coloneqq \overline{\boldsymbol{v}}^{\tau_n}_{m_n}\in \mathbb{P}^0(\mathcal{I}_{\tau_n},\Vo_{h_{m_n},0})$, $n\in \mathbb{N}$, where $\tau_n=\frac{T}{K_n}$ and $K_n,m_n\to\infty$  $(n\to\infty)$, be an arbitrary diagonal sequence of the piece-wise constant interpolants $\overline{\boldsymbol{v}}_n^\tau\in \mathbb{P}^0(\mathcal{I}_\tau;\Vo_{h_n,0})$, $K,n\in \mathbb{N}$, from~Proposition \ref{apriori}. Then, there exists a not re-labeled~subsequence~and a limit  $\overline{\boldsymbol{v}}\in \bVp\cap  L^\infty(I;H) $ such that for every $r,s\in \mathcal{P}^{\infty}(Q_T)$ with  $r\prec q$ in $Q_T$ and $s\prec p$~in~$Q_T$,\linebreak it holds that
		\begin{align*}
			\begin{aligned} 
		\overline{\boldsymbol{v}}_n&\rightharpoonup\overline{\boldsymbol{v}}&&\quad\text{ in }\mathbfcal{U}^{r,s}_{\D}&&\quad(n\to\infty) \,,\\
		\overline{\boldsymbol{v}}_n&\overset{\ast}{\rightharpoondown}\overline{\boldsymbol{v}}&&\quad\text{ in }L^\infty(I;Y)&&\quad(n\to\infty)\,.
			\end{aligned}
		\end{align*}
		Furthermore, it follows that $\overline{\boldsymbol{v}}\in \mathbfcal{W}^{q,p}_{\D}$, with continuous representation $\overline{\boldsymbol{v}}_c\in  C^0(\overline{I};H)$, and 
		\begin{alignat*}{2}
		\frac{\mathbf{d}_\sigma\overline{\boldsymbol{v}}}{\mathbf{dt}}+\mathbfcal{A}\overline{\boldsymbol{v}}&=\boldsymbol{f}^*&&\quad\text{ in }(\bVp)^*\,,\\
		\overline{\boldsymbol{v}}_c(0)&=\mathbf{v}_0&&\quad\text{ in }H\,. 
		\end{alignat*}
	\end{theorem}

	\begin{proof}
		For ease of presentation, without loss of generality, we assume  that $m_n=n$ for all $n\in \mathbb{N}$.
		
		\textit{1.\ (Weak) convergences.} From \eqref{eq:ap.1}--\eqref{eq:ap.4}, using Proposition~\ref{prop:Xqp_prime_weak_compact}, Proposition \ref{prop:approx_closedness},~and~the~Banach--Alaoglu theorem and that $L^\infty(I;Y)$ has a separable pre-dual space,  we obtain not re-labeled
		subsequences  
		as well as (weak) limits ${\overline{\boldsymbol{v}}\in \bVp\cap  L^\infty(I;H) }$, $\widehat{\boldsymbol{v}}\in L^\infty(I;Y)$, and
		$\overline{\boldsymbol{v}}^*\in(\bVp)^*$ such that for every $r_0,r_{\infty},s_0,s_{\infty}\in \mathcal{P}^\infty(Q_T)$ with $r_0\prec q\prec r_{\infty}$ in $Q_T$ and  $s_0\prec p\prec s_{\infty}$ in $Q_T$, it holds that
		\begin{align}\label{eq:main.1}
		\begin{aligned}
		\overline{\boldsymbol{v}}_n&\rightharpoonup\overline{\boldsymbol{v}}&&\quad\text{ in }\mathbfcal{U}^{r_0,s_0}_{\D}&&\quad (n\to \infty)\,,\\
		\overline{\boldsymbol{v}}_n&\overset{\ast}{\rightharpoondown}\overline{\boldsymbol{v}}&&\quad\text{
			in }L^\infty(I;Y)&&\quad (n\to \infty)\,,\\
		\widehat{\boldsymbol{v}}_n&\overset{\ast}{\rightharpoondown}\widehat{\boldsymbol{v}}&&\quad\textup{
			in }L^\infty(I;Y)&&\quad (n\to \infty)\,,\\[0.5mm]
		(\mathrm{id}_{\smash{\mathbfcal{U}^{r_{\infty},s_{\infty}}_{\D}}})^*\mathbfcal{A}_n\overline{\boldsymbol{v}}_n&\rightharpoonup(\mathrm{id}_{\smash{\mathbfcal{U}^{r_{\infty},s_{\infty}}_{\D}}})^*\overline{\boldsymbol{v}}^*&&\quad\textup{
			in }(\mathbfcal{U}^{r_{\infty},s_{\infty}}_{\D})^*&&\quad (n\to \infty)\,.
		\end{aligned}
		\end{align}
		Equipping $\Vo_{h_n,0}$ with the norm $\|\cdot\|_{q^+_n,p^+_n,\Omega}$, \eqref{eq:ap.4} yields a (Lebesgue) measurable set $\widehat{I}\subseteq I$~with~$\vert I\setminus\smash{\widehat{I}}\vert = 0$ and a not re-labeled
		subsequence such that 
		\begin{align}
		\|e_n(\widehat{\boldsymbol{v}}_n(t)-\overline{\boldsymbol{v}}_n(t))\|_{\smash{(\Vo_{h_n,0})^*}}\to 0\quad (n\to \infty)\quad\text{ for all }t\in\smash{\widehat{I}}\,.\label{eq:main.2}
		\end{align}
        Next, let $\mathbf{z}\in \mathcal{V}$ be fixed, but arbitrary. Then, $\mathbf{z}_{h_n}\coloneqq \Pi_{h_n}^V\mathbf{z}\in \Vo_{h_n,0}$, $n\in\mathbb{N}$, satisfies $\|\mathbf{z}_{h_n}-\mathbf{z}\|_{q^+_n,p^+_n,\Omega}\to 0 $ $(n\to\infty)$ (\textit{cf}.\ Lemma \ref{rmk:proj}(ii)). Hence, using \eqref{eq:iden2},  the variable Hölder inequality~\eqref{eq:gen_hoelder}, \eqref{eq:ap.1}, \eqref{eq:ap.2}, and \eqref{eq:main.2},  for every $t\in \smash{\widehat{I}}$, we infer that
		\begin{align}\label{eq:main.3} 
        \begin{aligned}
		\vert(P_H(\widehat{\boldsymbol{v}}_n(t)-\overline{\boldsymbol{v}}_n(t)),\mathbf{z})_{\Omega}\vert 
		&\leq \vert\langle e_n(\widehat{\boldsymbol{v}}_n(t)-\overline{\boldsymbol{v}}_n(t)),\mathbf{z}_{h_n}\rangle_{\smash{\Vo_{h_n,0}}}\vert + \vert(\widehat{\boldsymbol{v}}_n(t)-\overline{\boldsymbol{v}}_n(t),\mathbf{z}-\mathbf{z}_{h_n})_{\Omega}\vert
		\\&\leq 	
		\|e_n[\widehat{\boldsymbol{v}}_n(t)-\overline{\boldsymbol{v}}_n(t)]\|_{\smash{(\Vo_{h_n,0})^*}}\|\mathbf{z}_{h_n}\|_{q^+_n,p^+_n,\Omega}\\&\quad+ 2\,M\,\|\mathbf{z}-\mathbf{z}_{h_n}\|_{2,\Omega}
		\to 0\quad (n\to \infty)\,.
  \end{aligned}
		\end{align}
		 Since $\mathbf{z}\in \mathcal{V}$ is arbitrary in \eqref{eq:main.3} and  $\mathcal{V}$ is dense in $H$, due to \eqref{eq:main.1}$_{2,3}$, from \eqref{eq:main.3},~we~conclude~that
		\begin{align}
		P_H(\widehat{\boldsymbol{v}}_n(t)-\overline{\boldsymbol{v}}_n(t))\rightharpoonup \mathbf{0}\quad\text{ in }H\quad (n\to \infty)\quad\text{ for all }t\in\smash{\widehat{I}}\,,\label{eq:main.4}
		\end{align}
		where $P_H\colon Y\to H$ denotes the orthogonal projection from $Y$ onto $H$.
        Due to  \eqref{eq:ap.1}, using  Lemma \ref{lem:H_perp}, we obtain a (Lebesgue) measurable set $I_{\perp}\subseteq I$ with $\vert I\setminus I_{\perp}\vert=0$ and a not re-labeled subsequence~such~that
        \begin{align}\label{eq:main.7} 
            \begin{aligned}
                 P_H(\widehat{\boldsymbol{v}}_n(t)),
                 P_{H^{\perp}}(\overline{\boldsymbol{v}}_n(t))\rightharpoonup \mathbf{0}\quad\text{ in }H^{\perp}\quad (n\to \infty)\quad\text{ for all }t\in I_\perp
                 \,.
            \end{aligned}
        \end{align}
        Eventually, combining \eqref{eq:main.7} and \eqref{eq:main.4},  we arrive at
        \begin{align}\label{eq:main.8} 
		\widehat{\boldsymbol{v}}_n(t)-\overline{\boldsymbol{v}}_n(t)\rightharpoonup \mathbf{0}\quad\text{ in }Y\quad (n\to \infty)\quad\text{ for all }t\in\smash{\widehat{I}}\cap I_\perp\,.
		\end{align}
		Since the sequences $(\overline{\boldsymbol{v}}_n)_{n\in\mathbb{N}},(\widehat{\boldsymbol{v}}_n)_{n\in\mathbb{N}}$ are bounded in $ L^\infty(I;Y) $  (\textit{cf}.\ \eqref{eq:ap.1} \& \eqref{eq:ap.2}), from \eqref{eq:main.8}, it follows that $\overline{\boldsymbol{v}}_n-\widehat{\boldsymbol{v}}_n\rightharpoonup \boldsymbol{0}$ in $L^{\infty}(I;Y)$ $(n\to \infty)$. 
       From $\eqref{eq:main.1}_{2,3}$, we deduce that $\overline{\boldsymbol{v}}_n-\widehat{\boldsymbol{v}}_n\rightharpoonup \overline{\boldsymbol{v}}-\widehat{\boldsymbol{v}}$ in $L^{\infty}(I;Y)$ $(n\to \infty)$. Therefore, it holds that $\widehat{\boldsymbol{v}}=\overline{\boldsymbol{v}}$ in $ L^\infty(I;Y) $, where we used that $\overline{\boldsymbol{v}}\in \bVp\cap L^\infty(I;H) $. 
		
		\textit{2.\ Regularity and trace of the (weak) limit.}
		Let $\mathbf{z}\hspace*{-0.15em}\in\hspace*{-0.15em} \mathcal{V}$ be fixed, but arbitrary~and~let~${\mathbf{z}_{h_n}\hspace*{-0.15em}\coloneqq \hspace*{-0.15em}\Pi_{h_n}^V\mathbf{z}\hspace*{-0.15em}\in\hspace*{-0.15em} \Vo_{h_n,0}}$ for all $n\hspace*{-0.1em}\in\hspace*{-0.1em}\mathbb{N}$. If we test
		\eqref{scheme} for every $n\in \mathbb{N}$ with
		$\mathbf{z}_{h_n}\in  \Vo_{h_n,0}$, multiply by $\varphi\in C^\infty(\overline{I})$ with $\varphi(T)=0$, integrate with respect to $t\in I$, and  integrate-by-parts in time,  for every $n\in \mathbb{N}$, we find that\vspace*{-0.5mm}
		\begin{align}
			\begin{aligned}
		\langle \Pi_{\tau_n}^{0,\mathrm{t}}[\mathbfcal{A}_n]\overline{\boldsymbol{v}}_n-\Pi_{\tau_n}^{0,\mathrm{t}}[\boldsymbol{f}^*_n],\mathbf{z}_{h_n}\varphi\rangle_{\smash{\mathbfcal{U}^{q_n,p_n}_{\plus}}}&=(\widehat{\boldsymbol{v}}_n,\varphi^\prime\mathbf{z}_{h_n})_{Q_T}
		+(\mathbf{v}_n^0,\mathbf{z}_{h_n})_{\Omega}\,\varphi(0)\,.
	\end{aligned}\label{eq:main.9} 
		\end{align}
		Using Proposition \ref{rothe3}(iii.a), $\eqref{eq:main.1}_4$, $\mathbf{z}_{h_n}\Pi_{\tau_n}^{0,\mathrm{t}}[\varphi]\to \mathbf{z}\varphi$ in $\smash{L^s(I;W^{1,s}_0(\Omega;\mathbb{R}^d))}$ $(n\to \infty)$~for~all~${s\in (1,+\infty)}$  (\textit{cf}.\ Lemma \ref{rmk:proj}(ii) and Proposition~\ref{rothe2}(ii)), and Assumption \ref{assumption}(iii), we obtain\vspace*{-0.5mm}
		\begin{align}\label{eq:main.10}
			\begin{aligned}
		\langle \Pi_{\tau_n}^{0,\mathrm{t}}[\mathbfcal{A}_n]\overline{\boldsymbol{v}}_n- \Pi_{\tau_n}^{0,\mathrm{t}}[\boldsymbol{f}^*_n],\mathbf{z}_{h_n}\varphi\rangle_{\smash{\mathbfcal{U}^{q_n,p_n}_{\plus}}}&=\langle \mathbfcal{A}_n\overline{\boldsymbol{v}}_n-\boldsymbol{f}^*_n,\mathbf{z}_{h_n}\Pi_{\tau_n}^{0,\mathrm{t}}[\varphi]\rangle_{\smash{\bXpn}}\\&\to \langle\overline{\boldsymbol{v}}^*-\boldsymbol{f}^*,\mathbf{z}\varphi\rangle_{\smash{\bVp}}\quad (n\to \infty)\,.
			\end{aligned}
		\end{align}
		By passing in \eqref{eq:main.9} for $n\to \infty$, using \eqref{eq:main.1}$_3$, $\mathbf{z}_{h_n}\to \mathbf{z}$ in $\smash{W^{1,s}_0(\Omega;\mathbb{R}^d)}$ $(n\to \infty)$ for all $s\in (1,+\infty)$ (\textit{cf}.\ Lemma \ref{rmk:proj}(ii)),
		Assumption \ref{assumption}(i), and \eqref{eq:main.10},  for every $\mathbf{z}\in\mathcal{V}$ and
		$\varphi\in C^\infty(\overline{I})$~with~${\varphi(T)=0}$,~we~obtain\vspace*{-0.5mm}
		\begin{align}
		\begin{split}
		\langle \overline{\boldsymbol{v}}^*-\boldsymbol{f}^*,\mathbf{z}\varphi\rangle_{\smash{\bVp}}
  =(\overline{\boldsymbol{v}},\varphi^\prime\mathbf{z})_{Q_T}
		+(\mathbf{v}_0,\mathbf{z})_{\Omega}\,\varphi(0)\,.
		\end{split}\label{eq:main.12} 
		\end{align}
		If we consider
		$\varphi\in C_c^\infty(I)$ in \eqref{eq:main.12}, then, using Proposition \ref{prop:equivalent_characterization} and Proposition \ref{prop:integration-by-parts}(i),~we~\mbox{conclude}~that $\overline{\boldsymbol{v}}\in \mathbfcal{W}^{q,p}_{\D}$, with  a continuous representation $\overline{\boldsymbol{v}}_c \in C^0(\overline{I};H)$, and\vspace*{-0.5mm}
		\begin{align}
		\frac{\mathbf{d}_\sigma\overline{\boldsymbol{v}}}{\mathbf{dt}}=\boldsymbol{f}^*-\overline{\boldsymbol{v}}^*\quad\text{ in }
		(\bVp)^*\,.\label{eq:main.13}
		\end{align}
		Therefore, using Proposition \ref{prop:integration-by-parts}(ii) and \eqref{eq:main.13} in
		\eqref{eq:main.12} for $\varphi\in C^\infty(\overline{I})$ with
		$\varphi(T)=0$ and $\varphi(0)=1$, for every $\mathbf{z}\in	\mathcal{V}$, 
		we find that\vspace*{-1mm}
		\begin{align}
		(\overline{\boldsymbol{v}}_c(0)-\mathbf{v}_0,\mathbf{z})_{\Omega}=0\,. \label{eq:main.14}
		\end{align}
		Since $\mathcal{V}$ is dense in $H$ and $\overline{\boldsymbol{v}}_c(0)\in H$,  from \eqref{eq:main.14}, we deduce that $\overline{\boldsymbol{v}}_c(0)=\mathbf{v}_0$ in $H$.
		
		\textit{3.\ Point-wise weak convergence.}
		Next, we show that
		$\widehat{\boldsymbol{v}}_n(t)\rightharpoonup 
	\overline{\boldsymbol{v}}(t)$ in $Y$ $(n\to\infty)$ for a.e.\  $t \in I$,~which, due to \eqref{eq:main.8},  yields that
		 $\overline{\boldsymbol{v}}_n(t)
		\rightharpoonup  \overline{\boldsymbol{v}}(t)$ in $Y$ $(n\to\infty)$ for a.e.\   $t \in I$. To this end,  fix~an~arbitrary~$t\in I_\perp$. From the a priori estimate
		$\|\widehat{\boldsymbol{v}}_n(t)\|_{2,\Omega}\leq M$ for all
		$t\in I_\perp$ and $n\in\mathbb{N}$ (\textit{cf.}~\eqref{eq:ap.2}), we
		obtain  a subsequence
		$(\widehat{\boldsymbol{v}}_{n_k^t}(t))_{k\in\mathbb{N}}\subseteq
		Y$ (initially
		depending on this fixed $t$) as well as a (weak) limit
		$\widehat{\mathbf{w}}_t\in Y$  such that\vspace*{-0.5mm}
		\begin{align}
		\smash{	\widehat{\boldsymbol{v}}_{n_k^t}(t)} \rightharpoonup
			\widehat{\mathbf{w}}_t\quad\text{ in } Y\quad(
		 k\to \infty)\,.\label{eq:main.15}
		\end{align}
		Let $\mathbf{z}\in \mathcal{V}$ be fixed, but arbitrary and let $\mathbf{z}_{h_n}\coloneqq\Pi_{h_n}^V\mathbf{z}\in \Vo_{h_n,0}$ for all $n\in \mathbb{N}$. Testing \eqref{eq:scheme} for every $k\in \mathbb{N}$ with
		$\mathbf{z}_{h_{\smash{n_k^t}}}\in\Vo_{h_{\smash{n_k^t}},0}$, multiplying~by~${\varphi\hspace*{-0.1em}\in\hspace*{-0.1em} C^\infty(\overline{I})}$ with $\varphi(0)=0$ and
		$\varphi(t)=1$, integrating over $\left[0,t\right]$, and integration-by-parts in time, for every $k\in \mathbb{N}$,~we~find~that
		\begin{align}
			\begin{aligned}
		\langle \Pi_{\tau_{\smash{n_k^t}}}^{0,\mathrm{t}}[\mathbfcal{A}_{\smash{n_k^t}}]\overline{\boldsymbol{v}}_{\smash{n_k^t}}-\Pi_{\tau_{\smash{n_k^t}}}^{0,\mathrm{t}}[\boldsymbol{f}^*_{\smash{n_k^t}}],\mathbf{z}_{h_{\smash{n_k^t}}}\varphi\chi_{\left[0,t\right]}\rangle_{\smash{\mathbfcal{U}^{q_{\smash{n_k^t}},p_{\smash{n_k^t}}}_{\plus}}}
		=(\widehat{\boldsymbol{v}}_{\smash{n_k^t}},\varphi^\prime\mathbf{z}_{h_{\smash{n_k^t}}}\chi_{\left[0,t\right]})_{Q_T}
		-(\widehat{\boldsymbol{v}}_{\smash{n_k^t}}(t),\mathbf{z}_{h_{\smash{n_k^t}}})_{\Omega}\,.
	\end{aligned}\label{eq:main.16}
		\end{align}
		By passing in \eqref{eq:main.16} for  $k\to\infty$, using the same argumentation as for \eqref{eq:main.10}, but in this case~\eqref{eq:main.15} instead of Assumption \ref{assumption}(i),
		 for every $\mathbf{z}\in \mathcal{V}$, we obtain\vspace*{-0.5mm}
		\begin{align}
		\begin{split}\langle\overline{\boldsymbol{v}}^*-\boldsymbol{f}^*,\mathbf{z}\varphi\chi_{\left[0,t\right]}\rangle_{\smash{\bVp}}
		=(\overline{\boldsymbol{v}},\varphi^\prime\mathbf{z}\chi_{\left[0,t\right]})_{Q_T}
		-(\widehat{\mathbf{w}}_t,\mathbf{z})_{\Omega}\,.
		\end{split} \label{eq:main.17}
		\end{align}
		Therefore, using Proposition \ref{prop:integration-by-parts}(ii) and \eqref{eq:main.13} in 
		\eqref{eq:main.17}, for every $\mathbf{z}\in \mathcal{V}$, we find that\vspace*{-0.5mm}
		\begin{align}
		(\overline{\boldsymbol{v}}_c(t)-
		P_H\widehat{\mathbf{w}}_t,\mathbf{z})_{\Omega}=(\overline{\boldsymbol{v}}_c(t)-
		\widehat{\mathbf{w}}_t,\mathbf{z})_{\Omega}=0\,.\label{eq:main.18} 
		\end{align}
		Thanks to the density of $\mathcal{V}$ in $H$, \eqref{eq:main.18} implies that
		$\overline{\boldsymbol{v}}_c(t)=P_H\widehat{\mathbf{w}}_t$
		in $H$, \textit{i.e.}, looking back to \eqref{eq:main.15}, we just proved that\vspace*{-1mm}\enlargethispage{9 mm}
		\begin{align}
		\smash{P_H(\widehat{\boldsymbol{v}}_{\smash{n_k^t}}(t))}\rightharpoonup\overline{\boldsymbol{v}}_c(t)\quad\text{
			in }H\quad(k\to \infty)\,.\label{eq:main.19} 
		\end{align}
        Combining \eqref{eq:main.7} and \eqref{eq:main.19}, we arrive at\vspace*{-0.5mm}
        \begin{align}\label{eq:main.22} 
            \smash{\widehat{\boldsymbol{v}}_{\smash{n_k^t}}(t)}\rightharpoonup\overline{\boldsymbol{v}}_c(t)\quad\text{
			in }Y\quad( k\to \infty)\,.
        \end{align}
		As this argumentation is true for each subsequence of
		$(\widehat{\boldsymbol{v}}_n(t) )_{n\in\mathbb{N}}\subseteq Y$,
		$\overline{\boldsymbol{v}}_c(t)\in H$ is weak accumulation point of
		each subsequence of
		$(\widehat{\boldsymbol{v}}_n(t))_{n\in\mathbb{N}}\subseteq
		Y$. The standard convergence principle (\textit{cf.}~\cite[Kap.\ I, Lem.\ 5.4]{GGZ}) yields that $n_k^t=k$ for all $k\in \mathbb{N}$
		 in \eqref{eq:main.22}. Thus, using \eqref{eq:main.22} and that $\overline{\boldsymbol{v}}_c(t)=\overline{\boldsymbol{v}}(t)$~in~$H$~for~a.e.~$t\in I$, we obtain 
		$\widehat{\boldsymbol{v}}_n(t)\rightharpoonup \overline{\boldsymbol{v}}(t)$ in $Y$ $ (n\to \infty)$ for a.e.\  $t\in I$, 
        which, due to \eqref{eq:main.8}, yields that\vspace*{-0.5mm}
        \begin{align}\label{eq:main.23.1}
		\overline{\boldsymbol{v}}_n(t)\rightharpoonup \overline{\boldsymbol{v}}(t)\quad\text{
			in }Y\quad (n\to \infty)\quad\text{ for a.e.\  }t\in I\,.
		\end{align}
		
		\textit{4.\ \hspace*{-0.15mm}Identification \hspace*{-0.15mm}of \hspace*{-0.15mm}$\mathbfcal{A}\overline{\boldsymbol{v}}$ \hspace*{-0.15mm}and \hspace*{-0.1mm}$\overline{\boldsymbol{v}}^*\hspace*{-0.15em}$.}
		\hspace*{-0.15mm}Inequality \hspace*{-0.15mm}\eqref{eq:apriori.1} \hspace*{-0.15mm}for  \hspace*{-0.15mm}$\tau\hspace*{-0.15em}=\hspace*{-0.15em}\tau_n$ \hspace*{-0.15mm}and \hspace*{-0.15mm}$\ell \hspace*{-0.15em}=\hspace*{-0.15em}K_n$, \hspace*{-0.15mm}using~\hspace*{-0.15mm}the~\hspace*{-0.15mm}\mbox{definitions}~\hspace*{-0.15mm}\mbox{\eqref{def:temp_mean}--\eqref{def:l2_proj2}} in conjunction with $\langle\Pi_{\tau_n}^{0,\mathrm{t}}[\mathbfcal{A}_n]	\overline{\boldsymbol{v}}_n,\overline{\boldsymbol{v}}_n\rangle_{\smash{\mathbfcal{U}^{q_n,p_n}_{\plus}}}\hspace*{-0.15em}=\hspace*{-0.15em}\langle	\mathbfcal{A}_n\overline{\boldsymbol{v}}_n,\overline{\boldsymbol{v}}_n\rangle_{\smash{\bXpn}}$ and  ${\langle\Pi_{\tau_n}^{0,\mathrm{t}}[\boldsymbol{f}^*_n],\overline{\boldsymbol{v}}_n\rangle_{\smash{\mathbfcal{U}^{q_n,p_n}_{\plus}}}\hspace*{-0.15em}=\hspace*{-0.15em}\langle\boldsymbol{f}^*_n,\overline{\boldsymbol{v}}_n\rangle_{\smash{\bXpn}}}$ for all $n\in \mathbb{N}$ (\textit{cf}.\ Proposition \ref{rothe3}(iii.a) and Proposition \ref{rothe2}(i)),
		that $\overline{\boldsymbol{v}}_c(0)=\mathbf{v}_0$ in $H$,~and~that $\|P_H(\widehat{\boldsymbol{v}}_n(T))\|_{2,\Omega}\leq\|\widehat{\boldsymbol{v}}_n(T)\|_{2,\Omega}=\|\overline{\boldsymbol{v}}_n(T)\|_{2,\Omega}$ for all $n\in \mathbb{N}$,   for every ${n\in \mathbb{N}}$, yields that
		\begin{align}
		\begin{split}
		\langle
		\mathbfcal{A}_n\overline{\boldsymbol{v}}_n,\overline{\boldsymbol{v}}_n\rangle_{\smash{\bXpn}}&\leq 
		-\tfrac{1}{2}\|P_H(\widehat{\boldsymbol{v}}_n(T))\|_{2,\Omega}^2+\tfrac{1}{2}\|\overline{\boldsymbol{v}}_c(0)\|_{2,\Omega}^2+\langle \boldsymbol{f}^*_n,\overline{\boldsymbol{v}}_n\rangle_{\smash{\bXpn}}.
		\end{split}\label{eq:main.24}
		\end{align}
		The limit superior with respect to $n\to \infty$  in \eqref{eq:main.24}, \eqref{eq:main.22} with
		$n_k^t=k$ for all $k\in \mathbb{N}$ in the case $t=T$, the weak lower
		semi-continuity of $\|\cdot\|_{2,\Omega}$, and Proposition~\ref{prop:integration-by-parts}(ii) yield that
		\begin{align}
		\begin{split}
		\limsup_{n\to \infty}{ \langle
			\mathbfcal{A}_n\overline{\boldsymbol{v}}_n,\overline{\boldsymbol{v}}_n\rangle_{\smash{\bXpn}}}
		&\leq
		-\tfrac{1}{2}\|\overline{\boldsymbol{v}}_c(T)\|_{2,\Omega}^2+\tfrac{1}{2}\|\overline{\boldsymbol{v}}_c(0)\|_{2,\Omega}^2+\langle\boldsymbol{f}^*,\overline{\boldsymbol{v}}\rangle_{\smash{\bVp}}\\[-1mm]&=
		\bigg\langle\boldsymbol{f}^*-\frac{\mathbf{d}_\sigma\overline{\boldsymbol{v}}}{\mathbf{dt}},\overline{\boldsymbol{v}}\bigg\rangle_{\smash{\bVp}}\\&=\langle\overline{\boldsymbol{v}}^*,\overline{\boldsymbol{v}}\rangle_{\smash{\bVp}}\,.
		\end{split}\label{eq:main.25}
		\end{align}
		As a consequence of \eqref{eq:main.25}, \eqref{eq:main.1},
		\eqref{eq:main.23.1}, and the non-conforming Bochner condition (M) of~the~sequence
		$\mathbfcal{A}_n\colon \hspace*{-0.15em}\bXpn\cap L^\infty(I;Y)\hspace*{-0.15em}\to\hspace*{-0.15em}
		(\bXpn)^*$, $n\hspace*{-0.15em}\in\hspace*{-0.15em}  \mathbb{N}$, with respect to $(\Vo_{h_n,0})_{n\in \mathbb{N}}$ and $\mathbfcal{A}\colon \hspace*{-0.15em}\bVp\cap  L^\infty(I;H)\hspace*{-0.15em}\to\hspace*{-0.15em}
		(\bVp)^*$, we conclude that $\mathbfcal{A}\overline{\boldsymbol{v}}=\overline{\boldsymbol{v}}^*$ in $(\bVp)^*$. This completes the proof of Theorem~\ref{thm:main}.
	\end{proof}

\section{Practical realization and applications}\label{sec:application}\vspace*{-1mm}
\qquad In this section, we apply the general framework of the previous sections to prove the well-posedness, stability, and (weak) convergence of a fully-discrete Rothe--Galerkin approximation~of~the~unsteady~$p(\cdot,\cdot)$-Stokes equations (\textit{i.e.}, \eqref{eq:p-NS-scheme} without convective term) and of the unsteady $p(\cdot,\cdot)$-Navier--Stokes~equations \eqref{eq:p-NS-scheme}. In doing so,
the general  framework of the previous sections will reduce to fully-discrete Rothe--Galerkin approximations that are prototypical and  straightforward to  implement.\vspace*{-1mm}\enlargethispage{12mm}

\subsection{Approximation of the power-law index}\vspace*{-1mm}

\hspace*{5mm}As approximation of the power-law~index~$\smash{p\in \mathcal{P}^{\log}(Q_T)}$ with $p^-\ge 2$, for sequences $(\tau_n)_{n\in\mathbb{N}},(h_n)_{n\in\mathbb{N}}\subseteq \left(0,\infty\right)$ with $\tau_n,h_n\to 0$ $(n\to \infty)$, 
where $\tau_n\coloneqq  \frac{T}{K_n}$ for all $n\in \mathbb{N}$ for a sequence $(K_n)_{n\in\mathbb{N}}\subseteq \mathbb{N}$~such~that ${K_n\to \infty}$ $(n\to \infty)$, 
we employ a sequence of Riemann sums,~\textit{i.e.}, for every $n\in \mathbb{N}$,~we~define 
\begin{align}
	p_n\coloneqq \sum_{k=1}^{K_n}{\sum_{T\in \mathcal{T}_{h_n}}{p(t_k^*,x_T^*)\chi_{I_k\times T}}}\in \mathbb{P}^0(\mathcal{I}_{\tau_n} ;\mathbb{P}^0(\mathcal{T}_{h_n}))\,,\label{def:p_n}
\end{align}
where $\smash{(t_k^*,x_T^*)^\top}\in I_k \times T$ for each tuple $\smash{(k,T)^\top}  \in \{1,\ldots,K_n\} \times \mathcal{T}_{h_n}$  is an arbitrary quadrature~point,~\textit{e.g.}, we can set\vspace*{-1mm}
\begin{align}
	(t_k^*,x_T^*)^\top\coloneqq (t_k,x_T)^\top\in I_k \times T\,,\label{eq:explicit_quadpoints}
\end{align}
where $x_T\in T$ denotes the barycenter of $T$ for all $T\in \mathcal{T}_{h_n}$. Moreover, for some $\varepsilon\in (0,\frac{2}{d}p^-]$,~we~define $q \coloneqq p_*-\varepsilon\in \mathcal{P}^{\log}(Q_T)$ and 
$q_n\coloneqq (p_n)_*-\varepsilon\in \mathcal{P}^\infty(Q_T)$, $n\in \mathbb{N}$. Due to the uniform continuity~of~${p\in \mathcal{P}^{\log}(Q_T)}$, we~have~that\vspace*{-1mm}
\begin{align}\label{eq:conv_pn}
	\begin{aligned}
	q_n&\to q&&\quad \text{ in }L^\infty(Q_T)&&\quad (n\to \infty)\,,\\
	p_n&\to p&&\quad \text{ in }L^\infty(Q_T)&&\quad (n\to \infty)\,,
\end{aligned}
\end{align}
so that both Assumption \ref{assum:p_h} and Assumption \ref{assumption}(i) are satisfied.

\subsection{Approximation of the extra-stress tensor}\vspace*{-1mm}

\hspace*{5mm}
Let  the extra-stress tensor $\mathbf{S}\colon Q_T\times\mathbb{R}^{d\times d}_{\textup{sym}}\to \mathbb{R}^{d\times d}_{\textup{sym}}$~be~\mbox{defined}~by~\eqref{eq:example-stress}~for~constants
%
   $\mu_0>0$~and~$\delta\ge 0$. As approximation of the extra-stress tensor, we employ  the sequence of mappings $\mathbf{S}_n\colon Q_T\times\mathbb{R}^{d\times d}_{\textup{sym}}\to \mathbb{R}^{d\times d}_{\textup{sym}}$, $n\in \mathbb{N}$, for a.e.\  $(t,x)^\top\in Q_T$ and every~${\mathbf{A}\in \mathbb{R}^{d\times d}_{\textup{sym}}}$~defined~by\vspace*{-0.5mm}
\begin{align}\label{eq:explicit_Sn}
		\mathbf{S}_n(t,x,\mathbf{A})\coloneqq \mu_0\,(\delta +\vert \mathbf{A}\vert)^{p_n(t,x)-2}\mathbf{A}\,.
\end{align}
From the uniform convergence \eqref{eq:conv_pn}$_1$, it follows that the mappings $\mathbf{S}_n,\mathbf{S}\colon Q_T\hspace*{-0.1em}\times\hspace*{-0.1em}\mathbb{R}^{d\times d}_{\textup{sym}}\hspace*{-0.1em}\to\hspace*{-0.1em} \mathbb{R}^{d\times d}_{\textup{sym}}$, $n\in \mathbb{N}$,   satisfy   (\hyperlink{SN.1}{SN.1})--(\hyperlink{SN.4}{SN.4}) with respect to $(p_n)_{n\in\mathbb{N}}\subseteq L^\infty(Q_T)$ and ${p\in \mathcal{P}^{\log}(Q_T)}$. 

\subsection{Approximation of the initial velocity vector field and right-hand side data}

\hspace*{5mm}\textit{Initial velocity vector field.} As approximation of the initial velocity vector~field~${\mathbf{v}_0\in H}$, 
we employ 
\begin{align*} 
		\smash{\mathbf{v}_n^0\coloneqq \Pi_{h_n}^{0,V}	\mathbf{v}\in \Vo_{h_n,0}\,, \quad n\in \mathbb{N}\,,}
\end{align*}
where \hspace*{-0.1mm}$\Pi_{h_n}^{0,V}\hspace*{-0.15em}	\colon \hspace*{-0.15em}Y\hspace*{-0.15em}\to\hspace*{-0.15em} \Vo_{h_n,0}$ \hspace*{-0.1mm}denotes \hspace*{-0.1mm}the \hspace*{-0.1mm}(global) \hspace*{-0.1mm}$L^2$-projection from \hspace*{-0.1mm}$Y\hspace*{-0.2em}$ \hspace*{-0.1mm}onto \hspace*{-0.1mm}$\Vo_{h_n,0}$, \hspace*{-0.1mm}which~\hspace*{-0.1mm}for~\hspace*{-0.1mm}every~\hspace*{-0.1mm}${\mathbf{z}\hspace*{-0.15em}\in\hspace*{-0.15em} Y\hspace*{-0.2em}}$~\hspace*{-0.1mm}\mbox{satisfies}
\begin{align*}
	\sup_{n\in\mathbb{N}}{\big\|\Pi_{h_n}^{0,V}\mathbf{z}\big\|_{2,\Omega}}\leq \|\mathbf{z}\|_{2,\Omega}\quad\text{ and }\quad
	\Pi_{h_n}^{0,V}\mathbf{z}\to \mathbf{z}\quad\text{ in }Y\quad(n\to \infty)\,,
\end{align*}
so that Assumption \ref{assumption}(ii) is satisfied.\enlargethispage{7mm}

\textit{Right-hand side data.} As approximation of the right-hand side data $\boldsymbol{f}^*\hspace*{-0.1em}\coloneqq \hspace*{-0.1em}\mathbfcal{J}_{\mathbf{D}_x}(\boldsymbol{f},\boldsymbol{F})\hspace*{-0.1em}\in\hspace*{-0.1em} (\bVp)^*$,~where $\boldsymbol{f}\in L^{(p^-)'}(Q_T;\mathbb{R}^d)$ and $\boldsymbol{F}\in L^{(p^-)'}(Q_T;\mathbb{R}^{d\times d}_{\textup{sym}})$, we employ 
$\boldsymbol{f}_n^*\coloneqq \mathbfcal{J}_{\mathbf{D}_x}(\boldsymbol{f}_n,\boldsymbol{F}_n)\in  (\bXpn)^*$, $n\in \mathbb{N}$,~with
 \begin{alignat*}{2}
 		\boldsymbol{f}_n&\coloneqq   \Pi_{h_n}^{0,\mathrm{x}}\boldsymbol{f}\in L^{(p^-)'}(I;(\mathbb{P}^0(\mathcal{T}_{h_n}))^d)\,,&&\quad n\in \mathbb{N}\,,\\
 	\boldsymbol{F}_n&\coloneqq  \Pi_{h_n}^{0,\mathrm{x}}\boldsymbol{F}\in  L^{(p^-)'}(I;(\mathbb{P}^0(\mathcal{T}_{h_n}))^{d\times d}_{\textup{sym}})\,,&&\quad n\in \mathbb{N}\,,
 \end{alignat*}
 where 
  $ \Pi_{h_n}^{0,\mathrm{x}}	\colon L^1(\Omega;\mathbb{R}^{\ell})\to (\mathbb{P}^0(\mathcal{T}_{h_n}))^{\ell}$, $\ell\in \mathbb{N}$, denotes the (local) $L^2$-projection from $L^1(\Omega;\mathbb{R}^{\ell})$~onto $(\mathbb{P}^0(\mathcal{T}_{h_n}))^{\ell}$ for all ${n\in \mathbb{N}}$,
   which for every $\mathbf{z}\in  L^{\smash{(p^-)'}}(\Omega;\mathbb{R}^{\ell})$ satisfies
 \begin{align*}
 	\sup_{n\in\mathbb{N}}{\big\|\Pi_{h_n}^{0,\mathrm{x}}\mathbf{z}\big\|_{(p^-)',\Omega}}\leq \|\mathbf{z}\|_{(p^-)',\Omega}\quad\text{ and }\quad
 	\Pi_{h_n}^{0,\mathrm{x}}\mathbf{z}\to \mathbf{z}\quad\text{ in }L^{(p^-)'}(\Omega;\mathbb{R}^{\ell})\quad(n\to \infty)\,,
 \end{align*}
 so that Assumption \ref{assumption}(iii) is satisfied. 

\subsection{Unsteady $p(\cdot,\cdot)$-Stokes equations}

\hspace*{5mm}To distribute the difficulties, we first apply the general framework of the previous sections to establish the well-posedness, stability, and (weak) convergence of the following fully-discrete Rothe-Galerkin approximation of the unsteady $p(\cdot,\cdot)$-Stokes equations, \textit{i.e.}, we first focus on the treatment~of~the~extra-stress tensor and separate the difficulties coming from the convective term.\medskip 

\begin{algorithm}[Fully discrete, implicit scheme]\label{alg:p-S-scheme}
	For given $n\in \mathbb{N}$, the  iterates $(\mathbf{v}_n^k)_{k=1,\ldots ,K_n}\subseteq \Vo_{h_n,0}$, for every $k=1,\ldots ,K_n$ and $\mathbf{z}_{h_n}\in \Vo_{h_n,0}$, are defined via
	\begin{align}
		(\mathrm{d}_{\tau_n} \mathbf{v}_n^k,\mathbf{z}_n)_\Omega+( \mathbf{S}_n(t_k^*,\cdot,\mathbf{D}_x\mathbf{v}_n^k),\mathbf{D}_x\mathbf{z}_{h_n})_{\Omega}= (\langle \boldsymbol{f}_n\rangle_k,\mathbf{z}_{h_n})_{\Omega}+(\langle \boldsymbol{F}_n\rangle_k,\mathbf{D}_x\mathbf{z}_{h_n})_{\Omega}
		\,.\label{eq:p-S-scheme}
	\end{align}
\end{algorithm}\medskip

\begin{remark}\label{rem:S_no_t}
	Since, by construction \eqref{def:p_n}, it holds that $p_n(t_k^*,x)\hspace*{-0.1em}=\hspace*{-0.1em}p_n(t,x)$ for a.e.\ $t\hspace*{-0.1em}\in\hspace*{-0.1em} I_k$,~${k\hspace*{-0.1em}=\hspace*{-0.1em}1,\ldots,K_n}$, and $x\in\Omega$, we have that
	$\mathbf{S}_n(t_k^*,x,\mathbf{A})\hspace*{-0.1em}=\hspace*{-0.1em}\mathbf{S}_n(t,x,\mathbf{A})$ for a.e.\ $t\hspace*{-0.1em}\in\hspace*{-0.1em} I_k$, $k\hspace*{-0.1em}=\hspace*{-0.1em}1,\ldots,K_n$, $x\hspace*{-0.1em}\in\hspace*{-0.1em}\Omega$, and all $\mathbf{A}\hspace*{-0.1em}\in\hspace*{-0.1em} \mathbb{R}^{d\times d}_{\textup{sym}}$.
\end{remark}
 
Using the preceding sections, we arrive at the first main result of the paper, \textit{i.e.}, the well-posedness, stability, and (weak) convergence of Scheme \ref{alg:p-S-scheme}.

\begin{theorem}[Well-posedness, stability, weak convergence]\label{thm:S}
 The following statements \mbox{apply}:
	\begin{description}[noitemsep,topsep=2pt,leftmargin=!,labelwidth=\widthof{\itshape(iii)},font=\normalfont\itshape]
		\item[(i)] \textup{Well-posedness.}  For every $n\in \mathbb{N}$,  there exist iterates $(\mathbf{v}_n^k)_{k=0,\ldots ,K_n}\subseteq \Vo_{h_n,0}$~solving~\eqref{eq:p-S-scheme};
		\item[(ii)] \textup{Stability.}  The piece-wise constant interpolants $\overline{\boldsymbol{v}}_n\coloneqq\overline{\boldsymbol{v}}_n^{\tau_n} \in \mathbb{P}^0(\mathcal{I}_{\tau_n};\Vo_{h_n,0})$, $n\in \mathbb{N}$, satisfy
		\begin{align*}
			\sup_{n\in \mathbb{N}}{\big[\|\overline{\boldsymbol{v}}_n\|_{\bXpn\cap L^\infty(I;Y)}+\|\mathbfcal{S}_n\overline{\boldsymbol{v}}_n\|_{(\bXpn)^*}\big]}<\infty\,;
		\end{align*}
		\item[(iii)] \textup{Weak convergence.} 
		There exists  a  limit $\boldsymbol{v}\in  \bVp \cap L^\infty(I;H)$ such that  
		for~every~${r,s\in \mathcal{P}^{\infty}(Q_T)}$ with  $r\prec q$ in $Q_T$ and $s\prec p$ in $Q_T$,  it holds that
		\begin{align*}
			\begin{aligned}
				\overline{\boldsymbol{v}}_n&\rightharpoonup\overline{\boldsymbol{v}}&&\quad\textup{ in }\mathbfcal{U}^{r,s}_{\D}&&\quad (n\to \infty)\,,\\
				\overline{\boldsymbol{v}}_n&\overset{\ast}{\rightharpoondown}\overline{\boldsymbol{v}}&&\quad\textup{ in }L^\infty(I;Y)&&\quad(n\to\infty)\,.
			\end{aligned}
		\end{align*}
		Furthermore, it follows that $\overline{\boldsymbol{v}}\in \mathbfcal{W}_{\D}^{q,p}$, with continuous representation $\overline{\boldsymbol{v}}_c\in C^0(\overline{I};H)$, satisfies  $\overline{\boldsymbol{v}}_c(0)=\mathbf{v}_0$ in $H$ and for every $\boldsymbol{\varphi}\in \mathbfcal{D}_T$, it holds that
		\begin{align*}
			-(\overline{\boldsymbol{v}},\partial_t\boldsymbol{\varphi})_{Q_T}+(\mathbf{S} (\cdot,\cdot,\mathbf{D}_x\overline{\boldsymbol{v}}),\mathbf{D}_x\boldsymbol{\varphi})_{Q_T} =(\boldsymbol{f},\boldsymbol{\varphi})_{Q_T}+(\boldsymbol{F},\mathbf{D}_x\boldsymbol{\varphi})_{Q_T}\,.
		\end{align*}
	\end{description}
\end{theorem}

\begin{proof}
	For every $n\in \mathbb{N}$ and a.e.\ $t\in I$, let $S_n(t)\colon U^{q_n,p_n}_{\D}(t)\to (U^{q_n,p_n}_{\D}(t))^*$ for every $\mathbf{u}_n,\mathbf{z}_n\in U^{q_n,p_n}_{\D}(t)$ be defined by 
	\begin{align}\label{def:S_n_t}
		\langle S_n(t)\mathbf{u}_n,\mathbf{z}_n\rangle_{\smash{U^{q_n,p_n}_{\D}(t)}}\coloneqq (\mathbf{S}_n(t,\cdot,\mathbf{D}_x\mathbf{u}_n),\mathbf{D}_x\mathbf{z}_n)_{\Omega}\,.
	\end{align}
	We  prove that the operator families $S_n(t)\colon  U^{q_n,p_n}_{\D}(t)\to  (U^{q_n,p_n}_{\D}(t))^*$, $t\in  I$, $n\in \mathbb{N}$,~satisfy~\mbox{(\hyperlink{A.1}{A.1})--(\hyperlink{A.4}{A.4})}:
	
	\textit{ad  (\hyperlink{A.1}{A.1}).} Following the proof of \cite[Prop.\ 5.10(C.1)\&(C.4)]{alex-book}, for a.e.\ $t\hspace*{-0.1em}\in\hspace*{-0.1em} I$~and~every~${n\hspace*{-0.1em}\in \hspace*{-0.1em}\mathbb{N}}$,~we~find~that $S_n(t)\colon U^{q_n,p_n}_{\D}(t)\to (U^{q_n,p_n}_{\D}(t))^*$ is well-defined, continuous, and, thus, demi-continuous, so that the operator families $S_n(t)\colon U^{q_n,p_n}_{\D}(t)\to (U^{q_n,p_n}_{\D}(t))^*$, $t\in I$, $n\in \mathbb{N}$, satisfy (\hyperlink{A.1}{A.1}).
	

		\textit{ad  (\hyperlink{A.2}{A.2}).} For every $\boldsymbol{u}_n,\boldsymbol{z}_n\in \mathbfcal{U}_{\D}^{q_n,p_n}$, thanks to the variable Hölder inequality \eqref{eq:gen_hoelder}, we have that 
		$\mathbf{S}_n(\cdot,\cdot,\mathbf{D}_x\boldsymbol{u}_n):\mathbf{D}_x\boldsymbol{z}_n\in L^1(Q_T)$. By the Fubini-Tonelli theorem, we find that
		\begin{align*}
			(t\mapsto \langle S_n(t)(\boldsymbol{u}_n(t)), \boldsymbol{z}_n(t)\rangle_{\smash{U^{q_n,p_n}_{\D}(t)}})= (t\mapsto \mathbf{S}_n(t,\cdot,\mathbf{D}_x\boldsymbol{u}_n(t)),\mathbf{D}_x\boldsymbol{z}_n(t))_{\Omega})\in L^1(I)\,,
		\end{align*}
		  so that the operator families $S_n(t)\colon U^{q_n,p_n}_{\D}(t)\to  (U^{q_n,p_n}_{\D}(t))^*$, $t\in I$, $n\in \mathbb{N}$, satisfy (\hyperlink{A.2}{A.2}).
	
	\textit{ad  (\hyperlink{A.3}{A.3}).} Following the proof of \cite[Prop.\ 5.10(C.6)\&(C.3)]{alex-book}, for a.e.\ $t\hspace*{-0.1em}\in\hspace*{-0.1em} I$ and~every~${\mathbf{u}_n,\mathbf{z}_n\hspace*{-0.1em}\in\hspace*{-0.1em} U^{q_n,p_n}_{\D}(t)}$, $n\in \mathbb{N}$, we find that
		\begin{align}\label{thm:S.3}
			\begin{aligned}
			\vert \langle S_n(t)\mathbf{u}_n,\mathbf{z}_n\rangle_{\smash{U^{q_n,p_n}_{\D}(t)}}\vert &\leq 2^{(p_n^-)'}\mu_0^{(p_n^-)'}2^{p_n^+}\big(\rho_{p_n(t,\cdot),\Omega}(\delta)+\rho_{p_n(t,\cdot),\Omega}(\mathbf{D}_x\mathbf{u}_n)\big)\\&\quad+\smash{((p^+_n)')^{1-p_n^+}(p_n^-)^{-1}}
			\,\rho_{p_n(t,\cdot),\Omega}(\mathbf{D}_x\mathbf{z}_n)\,.
			\end{aligned}
		\end{align}
		Since, by construction \eqref{def:p_n}, we have that $p^-\hspace*{-0.1em}\leq\hspace*{-0.1em} p_n\hspace*{-0.1em}\leq\hspace*{-0.1em} p^+$ a.e.\ in $Q_T$ for all $n\hspace*{-0.1em}\in \hspace*{-0.1em}\mathbb{N}$, from~\eqref{thm:S.3},~\textit{e.g.},~\mbox{using}~that $\rho_{p_n(t,\cdot),\Omega}(\delta)\hspace*{-0.1em}\leq\hspace*{-0.1em} \vert \Omega\vert(\max\{1,\delta\})^{p^+}\hspace*{-0.25em}$ for a.e.\ $t\hspace*{-0.12em}\in\hspace*{-0.12em} I$ and all $n\hspace*{-0.12em}\in\hspace*{-0.12em} \mathbb{N}$, for a.e.\ $t\hspace*{-0.12em}\in\hspace*{-0.12em} I$ and every ${\mathbf{u}_n,\mathbf{z}_n\hspace*{-0.12em}\in\hspace*{-0.12em} U^{q_n,p_n}_{\D}(t)}$,~${n\hspace*{-0.12em}\in\hspace*{-0.12em} \mathbb{N}}$, it follows that
		\begin{align}\label{thm:S.4} \begin{aligned}\vert \langle S_n(t)\mathbf{u}_n,\mathbf{z}_n\rangle_{\smash{U^{q_n,p_n}_{\D}(t)}}\vert &\leq 2^{(p^-)'}(\max\{1,\mu_0\})^{(p^-)'}2^{p^+}\big(\vert \Omega\vert(\max\{1,\delta\})^{p^+}+\rho_{p_n(t,\cdot),\Omega}(\mathbf{D}_x\mathbf{u}_n)\big)\\&\quad+\smash{((p^+)')^{1-p^+}(p^-)^{-1}}
				\,\rho_{p_n(t,\cdot),\Omega}(\mathbf{D}_x\mathbf{z}_n)\,,
			\end{aligned}
		\end{align} 
		so that the operator families $S_n(t)\colon U^{q_n,p_n}_{\D}(t)\to  (U^{q_n,p_n}_{\D}(t))^*$, $t\in I$, $n\in \mathbb{N}$, satisfy (\hyperlink{A.3}{A.3}).
		
		\textit{ad  (\hyperlink{A.4}{A.4}).} Following the proof of \cite[Prop.\ 5.10(C.5)]{alex-book}, for a.e.\ $t\hspace*{-0.1em}\in\hspace*{-0.1em} I$ and every $\mathbf{u}_n\hspace*{-0.1em}\in\hspace*{-0.1em} U^{q_n,p_n}_{\D}(t)$,~$n\hspace*{-0.1em}\in\hspace*{-0.1em} \mathbb{N}$, using, again, that $\rho_{p_n(t,\cdot),\Omega}(\delta)\leq \vert \Omega\vert(\max\{1,\delta\})^{p^+}$ for a.e.\ $t\in I$ and all $n\in \mathbb{N}$, we find that
		\begin{align}\label{thm:S.5} \begin{aligned}
				\langle S_n(t)\mathbf{u}_n,\mathbf{u}_n\rangle_{\smash{U^{q_n,p_n}_{\D}(t)}} &\geq \tfrac{\mu_0}{2}\rho_{p_n(t,\cdot),\Omega}(\mathbf{D}_x\mathbf{u}_n)-\mu_0\,\rho_{p_n(t,\cdot),\Omega}(\delta)
				\\&\ge \tfrac{\mu_0}{2}\rho_{p_n(t,\cdot),\Omega}(\mathbf{D}_x\mathbf{u}_n)- \vert \Omega\vert(\max\{1,\delta\})^{p^+}\,,
			\end{aligned}
		\end{align}
		so that the operator families $S_n(t)\colon U^{q_n,p_n}_{\D}(t)\to  (U^{q_n,p_n}_{\D}(t))^*$, $t\in I$, $n\in \mathbb{N}$, satisfy (\hyperlink{A.4}{A.4}).
		
		In summary,  the operator families $S_n(t)\colon  U^{q_n,p_n}_{\D}(t)\to (U^{q_n,p_n}_{\D}(t))^*$,  $t\in  I$, $n\in  \mathbb{N}$,~\mbox{satisfy}~\mbox{(\hyperlink{A.1}{A.1})--(\hyperlink{A.4}{A.4})}.
		By Proposition \ref{prop:stress},  the induced sequence of
operators $\mathbfcal{S}_n\colon\hspace*{-0.15em}\mathbfcal{U}_{\D}^{q_n,p_n}\hspace*{-0.15em}\cap\hspace*{-0.15em} L^\infty(I;Y)\hspace*{-0.15em}\to\hspace*{-0.15em} (\mathbfcal{U}_{\D}^{q_n,p_n})^*$,~${n\hspace*{-0.15em}\in\hspace*{-0.15em} \mathbb{N}}$,~satisfies the non-conforming Bochner condition (M) with respect to $(\Vo_{h_n,0})_{n\in \mathbb{N}}$ and 
${\mathbfcal{S}\colon\hspace*{-0.15em} \mathbfcal{V}_{\D}^{q,p}\hspace*{-0.15em}\cap \hspace*{-0.15em}L^\infty(I;H)\hspace*{-0.15em}\to\hspace*{-0.15em} (\mathbfcal{V}_{\D}^{q,p})^*}$. As a result, the claims (i)--(iii) follow from Proposition \ref{prop:well-posed}, Proposition~\ref{apriori}, and Theorem \ref{thm:main} since Assumption \ref{assumption} is satisfied and  $\langle S_n\rangle_k=S_n(t_k^*)$ for all $k=1,\ldots,K_n$, $n\in \mathbb{N}$,~(\textit{cf}.~Remark~\ref{rem:S_no_t}).
\end{proof}

\subsection{Unsteady $p(\cdot,\cdot)$-Navier--Stokes equations}

\hspace*{5mm}Next, we apply the general framework of the previous  sections to establish the well-posedness, stability and (weak) convergence of the following fully-discrete Rothe--Galerkin approximation of the unsteady $p(\cdot,\cdot)$-Navier--Stokes equations \eqref{eq:ptxNavierStokes}, \textit{i.e.}, we include the difficulties coming~from~the~convective~term.\medskip

\begin{algorithm}[Fully discrete, implicit scheme]\label{alg:p-NS-scheme}
	For given $n\in \mathbb{N}$, the iterates $(\mathbf{v}_n^k)_{k=1,\ldots ,K_n}\subseteq \Vo_{h_n,0}$, for every $k=1,\ldots ,K_n$ and $\mathbf{z}_{h_n}\in \Vo_{h_n,0}$, are defined via
	\begin{align}
		\begin{aligned}
		(\mathrm{d}_{\tau_n} \mathbf{v}_n^k,\mathbf{z}_n)_\Omega+( \mathbf{S}_n(t_k^*,\cdot,\mathbf{D}_x\mathbf{v}_n^k),\mathbf{D}_x\mathbf{z}_{h_n})_{\Omega}
		&+\tfrac{1}{2}(\mathbf{z}_{h_n}\otimes \mathbf{v}_n^k,\nabla_{\!x}\mathbf{v}_n^k)_{\Omega}-\tfrac{1}{2}(\mathbf{v}_n^k\otimes \mathbf{v}_n^k,\nabla_{\!x}\mathbf{z}_{h_n})_{\Omega}\\
		&= (\langle \boldsymbol{f}_n\rangle_k,\mathbf{z}_{h_n})_{\Omega}+(\langle \boldsymbol{F}_n\rangle_k,\mathbf{D}_x\mathbf{z}_{h_n})_{\Omega}
		\,.
	\end{aligned}\label{eq:p-NS-scheme}
	\end{align}
\end{algorithm}\newpage

Using the preceding sections, we arrive at the second main result of the paper, \textit{i.e.}, the well-posedness, stability, and (weak) convergence of Scheme \ref{alg:p-NS-scheme}.\enlargethispage{5mm}

\begin{theorem}[Well-posedness, stability, weak convergence]\label{prop:p-NS-scheme}
	If $p^->p_c\coloneqq\frac{3d+2}{d+2}$, then for every $\varepsilon\in (0,\min\{(p^-)^*-(p_c)^*,\frac{2}{d}p^-\}]$, the following  statements apply:
	\begin{description}[noitemsep,topsep=2pt,leftmargin=!,labelwidth=\widthof{\itshape(iii)},font=\normalfont\itshape]
		\item[(i)] \textup{Well-posedness.}  For every $n\in \mathbb{N}$,  there exist iterates $(\mathbf{v}_n^k)_{k=0,\ldots ,K_n}\subseteq \Vo_{h_n,0}$ solving~\eqref{eq:p-NS-scheme};
		\item[(ii)] \textup{Stability.}  The piece-wise constant interpolants $\overline{\boldsymbol{v}}_n\coloneqq \overline{\boldsymbol{v}}_n^{\tau_n} \in \mathbb{P}^0(\mathcal{I}_{\tau_n};\Vo_{h_n,0})$, $n\in \mathbb{N}$, satisfy
		\begin{align*}
			\sup_{n\in \mathbb{N}}{\big[\|\overline{\boldsymbol{v}}_n\|_{\smash{\bXpn\cap L^\infty(I;Y)}}+\|\mathbfcal{S}_n\overline{\boldsymbol{v}}_n\|_{\smash{(\bXpn)^*}}+\|\mathbfcal{C}_n\overline{\boldsymbol{v}}_n\|_{\smash{(\bXpn)^*}}\big]}<\infty\,;
		\end{align*}
		\item[(iii)] \textup{Weak \hspace*{-0.1mm}convergence.}  
		There exists a not re-labeled subsequence and a  limit  $\overline{\boldsymbol{v}}\in \bVp\cap  L^\infty(I;H) $ such that for every $r,s\in \mathcal{P}^{\infty}(Q_T)$ with $r\prec q$ in $Q_T$ and $s\prec p$ in $Q_T$,  it holds that
		\begin{align*}
			\begin{aligned}
				\overline{\boldsymbol{v}}_n&\rightharpoonup\overline{\boldsymbol{v}}&&\quad\textup{ in }\mathbfcal{U}^{r,s}_{\D}&&\quad (n\to \infty)\,,\\
				\overline{\boldsymbol{v}}_n&\overset{\ast}{\rightharpoondown}\overline{\boldsymbol{v}}&&\quad\textup{ in }L^\infty(I;Y)&&\quad(n\to\infty)\,.
			\end{aligned}
		\end{align*}
		Furthermore, it follows that $\overline{\boldsymbol{v}}\in \mathbfcal{W}_{\D}^{q,p}$, with continuous representation $\overline{\boldsymbol{v}}_c\in C^0(\overline{I};H)$, satisfies  $\overline{\boldsymbol{v}}_c(0)=\mathbf{v}_0$ in $H$ and for every $\boldsymbol{\varphi}\in \mathbfcal{D}_T$, it holds that
		\begin{align*}
			-(\overline{\boldsymbol{v}},\partial_t\boldsymbol{\varphi})_{Q_T}+(\mathbf{S} (\cdot,\cdot,\mathbf{D}_x\overline{\boldsymbol{v}})-\overline{\boldsymbol{v}}\otimes\overline{\boldsymbol{v}},\mathbf{D}_x\boldsymbol{\varphi})_{Q_T} =(\boldsymbol{f},\boldsymbol{\varphi})_{Q_T}+(\boldsymbol{F},\mathbf{D}_x\boldsymbol{\varphi})_{Q_T}\,.
		\end{align*}
	\end{description}
\end{theorem}

\begin{proof} 
	For every $n\hspace*{-0.1em}\in \hspace*{-0.1em}\mathbb{N}$ and a.e.\ $t\hspace*{-0.1em}\in\hspace*{-0.1em} I$, we define the operator ${A_n(t)\hspace*{-0.1em}\coloneqq \hspace*{-0.1em} S_n(t)+C\hspace*{-0.1em}\colon \hspace*{-0.1em}U^{q_n,p_n}_{\D}(t)\hspace*{-0.1em}\to \hspace*{-0.1em}(U^{q_n,p_n}_{\D}(t))^*}$, where $S_n(t)\colon U^{q_n,p_n}_{\D}(t)\to (U^{q_n,p_n}_{\D}(t))^*$ is, again, defined by \eqref{def:S_n_t} and $C\colon U^{q_n,p_n}_{\D}(t)\to (U^{q_n,p_n}_{\D}(t))^*$~for every $\mathbf{u}_n,\mathbf{z}_n\in U^{q_n,p_n}_{\D}(t)$ is defined by 
	\begin{align}\label{def:C_n}
		\langle C\mathbf{u}_n,\mathbf{z}_n\rangle_{\smash{U^{q_n,p_n}_{\D}(t)}}\coloneqq \tfrac{1}{2}(\mathbf{z}_n\otimes\mathbf{u}_n,\nabla_{\!x}\mathbf{u}_n)_{\Omega}-\tfrac{1}{2}(\mathbf{u}_n\otimes\mathbf{u}_n,\nabla_{\!x}\mathbf{z}_n)_{\Omega}\,.
	\end{align}
	We prove  that the operator families $A_n(t)\colon U^{q_n,p_n}_{\D}(t)\in  (U^{q_n,p_n}_{\D}(t))^*$, $t\in  I$, $n\in  \mathbb{N}$,~satisfy~\mbox{(\hyperlink{A.1}{A.1})--(\hyperlink{A.4}{A.4})}:
	
	\textit{ad (\hyperlink{A.1}{A.1}).} Similar to the proof of \cite[Prop.\ 5.11(C.1),(C.2),(C.4)]{alex-book},  from the compact embedding  $W^{1,p_c}(\Omega;\mathbb{R}^d)\hspace*{-0.15em}\hookrightarrow\hookrightarrow \hspace*{-0.15em} L^s(\Omega;\mathbb{R}^d)$ valid for all $s\hspace*{-0.15em}\in\hspace*{-0.15em}\left[1,(p_c)^*\right)$, where $(p_c)^*\hspace*{-0.15em}\coloneqq\hspace*{-0.15em} \frac{dp_c}{d-p_c}$ if $d\hspace*{-0.15em}\ge\hspace*{-0.15em} 3$ and $(p_c)^*\hspace*{-0.15em}\coloneqq\hspace*{-0.15em} \infty$~if~${d\hspace*{-0.15em}=\hspace*{-0.15em}2}$, for a.e.\ $t\in I$ and every $n\in \mathbb{N}$, 
	it follows that $C\colon U^{q_n,p_n}_{\D}(t)\to (U^{q_n,p_n}_{\D}(t))^*$~is~\mbox{well-defined}~and~compact, so that, 
	 appealing to the proof of Proposition \ref{thm:S}, $A_n(t)\colon U^{q_n,p_n}_{\D}(t)\to (U^{q_n,p_n}_{\D}(t))^*$ is continuous.
	
	\textit{ad (\hyperlink{A.2}{A.2}).}   For every $\boldsymbol{u}_n,\boldsymbol{z}_n\in \mathbfcal{U}_{\D}^{q_n,p_n}$, due to Hölder's inequality \eqref{eq:gen_hoelder}, we have that 
	${\boldsymbol{z}_n\otimes\boldsymbol{u}_n:\nabla_{\!x}\boldsymbol{u}_n},$ $\boldsymbol{u}_n\otimes\boldsymbol{u}_n:\nabla_{\!x}\boldsymbol{z}_n\hspace*{-0.1em}\in\hspace*{-0.1em} L^1(Q_T)$. By the Fubini-Tonelli theorem, we find that 
	$(t\hspace*{-0.1em}\mapsto \hspace*{-0.1em}\langle C(\boldsymbol{u}_n(t)), \boldsymbol{z}_n(t)\rangle_{\smash{U^{q_n,p_n}_{\D}(t)}})\hspace*{-0.1em}= (t\mapsto \tfrac{1}{2}(\boldsymbol{z}_n(t)\otimes \boldsymbol{u}_n(t),\nabla_{\!x}\boldsymbol{u}_n(t))_{\Omega}-\tfrac{1}{2}(\boldsymbol{u}_n(t)\otimes\boldsymbol{u}_n(t),\nabla_{\!x}\boldsymbol{z}_n(t))_{\Omega})\in L^1(I)$,
	so that the operator families $A_n(t)\colon U^{q_n,p_n}_{\D}(t)\to  (U^{q_n,p_n}_{\D}(t))^*$, $t\in I$, $n\in \mathbb{N}$, satisfy (\hyperlink{A.2}{A.2}).
	
	
	\textit{ad (\hyperlink{A.3}{A.3}).} Using Hölder's inequality, Young's inequality, and that, owing to $2p_c'= (p_c)_*\leq q_n\leq q^+$ a.e.\ in $Q_T$, since $\varepsilon\hspace*{-0.1em}\in\hspace*{-0.1em} (0,(p^-)_*-(p_c)_*]$,  and $ p_c\hspace*{-0.1em}\leq\hspace*{-0.1em} p_n\hspace*{-0.1em}\leq\hspace*{-0.1em} p^+$ a.e.\ in $Q_T$ for all $n\hspace*{-0.1em}\in\hspace*{-0.1em} \mathbb{N}$, for a.e.\ $t\hspace*{-0.1em}\in\hspace*{-0.1em} I$ and~every~${n\hspace*{-0.1em}\in\hspace*{-0.1em}\mathbb{N}}$, it~holds~that 
	$\rho_{2 p_c',\Omega}(\mathbf{u}_n)\leq 2^{q^+}\,(\vert \Omega\vert +\rho_{q_n(t,\cdot),\Omega}(\mathbf{u}_n))$ as well as $\rho_{p_c,\Omega}(\mathbf{D}_x\mathbf{u}_n)\leq 2^{p^+}\,(\vert \Omega\vert+\rho_{p_n(t,\cdot),\Omega}(\mathbf{D}_x\mathbf{u}_n))$ for all $\mathbf{u}_n\in U^{q_n,p_n}_{\D}(t)$, $n\in \mathbb{N}$, for a.e.\ $t\in I$ and every $\textbf{u}_n,\textbf{z}_n\in U^{q_n,p_n}_{\D}(t)$, $n\in\mathbb{N}$, we find that
	\begin{align*}
		\vert\langle C\mathbf{u}_n,\mathbf{z}_n\rangle_{\smash{U^{q_n,p_n}_{\D}(t)}}\vert&\leq \|\mathbf{u}_n\|_{2 p_c',\Omega}^2\|\mathbf{D}_x\mathbf{z}_n\|_{p_c,\Omega}+\|\mathbf{u}_n\|_{2p_c',\Omega}\|\mathbf{z}_n\|_{2p_c',\Omega}\|\mathbf{D}_x\mathbf{u}_n\|_{p_c,\Omega}\\&
		\leq \rho_{2 p_c',\Omega}(\mathbf{u}_n)+\rho_{p_c,\Omega}(\mathbf{D}_x\mathbf{z}_n)+\rho_{2 p_c',\Omega}(\mathbf{z}_n)+\rho_{p_c,\Omega}(\mathbf{D}_x\mathbf{u}_n)\\
		&\leq 2^{q^+}\big(4\,\vert \Omega\vert +\rho_{q_n(t,\cdot),\Omega}(\mathbf{u}_n)+\rho_{q_n(t,\cdot),\Omega}(\mathbf{z}_n)+\rho_{p_n(t,\cdot),\Omega}(\mathbf{D}_x\mathbf{z}_n)+\rho_{p_n(t,\cdot),\Omega}(\mathbf{D}_x\mathbf{u}_n)\big)\,.
	\end{align*} 
	Due \hspace*{-0.1mm}to \hspace*{-0.1mm}\eqref{thm:S.4}, \hspace*{-0.1mm}it \hspace*{-0.1mm}follows \hspace*{-0.1mm}that \hspace*{-0.1mm}the \hspace*{-0.1mm}operator \hspace*{-0.1mm}families \hspace*{-0.1mm}$A_n(t)\colon \hspace*{-0.15em}U^{q_n,p_n}_{\D}(t)\hspace*{-0.15em}\to\hspace*{-0.15em} (U^{q_n,p_n}_{\D}(t))^*$, \hspace*{-0.1mm}$t\hspace*{-0.15em}\in\hspace*{-0.15em} I$, \hspace*{-0.1mm}$n\hspace*{-0.15em}\in\hspace*{-0.15em} \mathbb{N}$,~\hspace*{-0.1mm}\mbox{satisfy}~\hspace*{-0.1mm}(\hyperlink{A.3}{A.3}).

	\textit{ad (\hyperlink{A.4}{A.4}).} By construction \eqref{def:C_n}, for a.e.\ $t\in I$ and every $\mathbf{u}_n\in U^{q_n,p_n}_{\D}(t)$, $n\in \mathbb{N}$, we have that
	\begin{align*}
		\langle C\mathbf{u}_n,\mathbf{u}_n\rangle_{\smash{U^{q_n,p_n}_{\D}(t)}}=0\,.
	\end{align*}
	Due \hspace*{-0.15mm}to \hspace*{-0.15mm}\eqref{thm:S.5}, \hspace*{-0.15mm}it \hspace*{-0.15mm}follows \hspace*{-0.15mm}that \hspace*{-0.15mm}the \hspace*{-0.15mm}operator \hspace*{-0.15mm}families \hspace*{-0.15mm}$A_n(t)\colon \hspace*{-0.175em}U^{q_n,p_n}_{\D}(t)\hspace*{-0.175em}\to\hspace*{-0.175em} (U^{q_n,p_n}_{\D}(t))^*$, $t\hspace*{-0.175em}\in \hspace*{-0.175em}I$,~${n\hspace*{-0.175em}\in \hspace*{-0.175em}\mathbb{N}}$,~\hspace*{-0.15mm}\mbox{satisfy}~\hspace*{-0.15mm}(\hyperlink{A.4}{A.4}).  
	
In summary, the operator families $A_n(t)\colon U^{q_n,p_n}_{\D}(t)\to  (U^{q_n,p_n}_{\D}(t))^*$, $t\in  I$, $n\in  \mathbb{N}$,~satisfy~\mbox{(\hyperlink{A.1}{A.1})--(\hyperlink{A.4}{A.4})}.
By Proposition \ref{prop:stress},  the induced sequence of
operators ${\mathbfcal{A}_n\colon\hspace*{-0.15em}\mathbfcal{U}_{\D}^{q_n,p_n}\hspace*{-0.15em} \cap\hspace*{-0.15em}  L^\infty(I;Y)\hspace*{-0.15em}\to\hspace*{-0.15em} (\mathbfcal{U}_{\D}^{q_n,p_n})^*}$,~${n\hspace*{-0.15em}\in\hspace*{-0.15em} \mathbb{N}}$,~\mbox{satisfies} the non-conforming Bochner condition (M) with respect to $(\Vo_{h_n,0})_{n\in \mathbb{N}}$ and 
${\mathbfcal{A}\colon \hspace*{-0.15em}\mathbfcal{V}_{\D}^{q,p}\hspace*{-0.15em}\cap\hspace*{-0.15em} L^\infty(I;H)\hspace*{-0.15em}\to\hspace*{-0.15em} (\mathbfcal{V}_{\D}^{q,p})^*}$.
As a result, the claims (i)--(iii) follow from Proposition \ref{prop:well-posed}, Proposition~\ref{apriori}, and Theorem \ref{thm:main} since Assumption \ref{assumption} is satisfied and  $\langle A_n\rangle_k=A_n(t_k^*)$ for all $k=1,\ldots,K_n$, $n\in \mathbb{N}$,~(\textit{cf}.~Remark~\ref{rem:S_no_t}).
\end{proof}

	\section{Numerical experiments}\label{sec:experiments}

\hspace*{5mm}In this section, we complement the theoretical findings  of Section \ref{sec:application} with numerical experiments:~first, we examine a purely academic example that is meant to indicate the 
(weak) convergence~of~Scheme~\ref{alg:p-NS-scheme} predicted by Theorem \ref{prop:p-NS-scheme}, for low regularity data, in the two-dimensional case (\textit{i.e.}, $d=2$); second, we consider a less academic example describing an unsteady incompressible electro-rheological fluid flow~in~the three-dimensional case (\textit{i.e.}, $d=3$).

\subsection{Implementation details} 

\hspace*{5mm}All experiments were carried out employing the finite element software package \texttt{FEniCS} (version 2019.1.0, \textit{cf}.\  \cite{LW10}). All triangulations were generated using  the \texttt{MatLab} (version R2022b, \textit{cf}.\  \cite{matlab}) library \texttt{DistMesh} (version 1.1, \textit{cf}.\  \cite{distmesh}).
All graphics were generated using the \texttt{Matplotlib} library~(version~3.5.1, \textit{cf.}~\cite{Hun07}), the \texttt{Vedo} library (version 2023.4.4, \textit{cf}.\ \cite{vedo}), or the \texttt{ParaView}~\mbox{engine}~(\mbox{version}~\mbox{5.12.0-RC2},~\textit{cf}.~\cite{ParaView}). The communication between \texttt{FEniCS} (\textit{i.e.}, \texttt{Numpy} (version 1.24.3, \textit{cf}.\  \cite{numpy})) and \texttt{MatLab}~relied~on~the~\texttt{mat4py} library (version 3.1.4, \textit{cf}.\ \cite{mat4py}). 

As quadrature points of the one-point~quadrature~rule~used~to discretize $p \in \mathcal{P}^{\log}(Q_T)$ (\textit{cf}.\ \eqref{def:p_n}), 
we employ tuples consisting of right limits of local time intervals and barycenters~of~elements~(\textit{i.e.},~\eqref{eq:explicit_quadpoints}).\enlargethispage{6mm}

For given sequences $(\tau_n)_{n\in \mathbb{N}},(h_n)_{n\in \mathbb{N}}\hspace*{-0.12em}\subseteq\hspace*{-0.12em} (0,+\infty)$ with $\tau_n,h_n\hspace*{-0.12em}\to\hspace*{-0.12em} 0$ $(n\hspace*{-0.12em}\to\hspace*{-0.12em} \infty)$, where~${\tau_n\hspace*{-0.12em} \coloneqq \hspace*{-0.12em}\frac{T}{K_n}}$~for~all~${n\hspace*{-0.12em}\in \hspace*{-0.12em}\mathbb{N}}$ for a sequence $(K_n)_{n\in \mathbb{N}}\subseteq \mathbb{N}$ with $K_n\to \infty$ $(n\to \infty)$, for every $k=1,\dots, K_n$,~\eqref{eq:p-NS-scheme}~can~\mbox{equivalently}~be re-written as a non-linear saddle point problem seeking for iterates
$\smash{(\mathbf{v}_n^k,\pi_n^k)^{\top}\hspace*{-0.05em}\in\hspace*{-0.05em} \Vo_{h_n}\times \Qo_{h_n}}$,~${k\hspace*{-0.05em}=\hspace*{-0.05em}1,\ldots,K_n}$, 
such that for every $(\mathbf{z}_{h_n},z_{h_n})^\top\in\Vo_{h_n}\times Q_{h_n}$,~it~holds~that
\begin{align}\label{eq:primal1}
	\begin{aligned}
		(\mathrm{d}_{\tau_n}\mathbf{v}_n^k, \mathbf{z}_{h_n})_{\Omega} 
		&+(\mathbf{S}_n(t_k,\cdot,\mathbf{D}_x\mathbf{v}_n^k),\mathbf{D}_x\mathbf{z}_{h_n})_{\Omega}\\&
		\quad+\tfrac{1}{2}(\mathbf{z}_{h_n}\otimes \mathbf{v}_n^k,\nabla_{\! x}\mathbf{v}_n^k)_{\Omega}-\tfrac{1}{2}(\mathbf{v}_n^k\otimes \mathbf{v}_n^k,\nabla_{\! x}\mathbf{z}_{h_n})_{\Omega} 
		\\&\quad-(\pi _n^{k},\textup{div}_x\mathbf{z}_{h_n})_{\Omega}
		\\
		&=  (\langle \boldsymbol{f}_n\rangle_k,\mathbf{z}_{h_n})_{\Omega}+(\langle \boldsymbol{F}_n\rangle_k,\mathbf{D}_x\mathbf{z}_{h_n})_{\Omega}\,,\\
		(\textup{div}_x \mathbf{v}_n^k,z_{h_n})_{\Omega}&=0\,.
	\end{aligned}
\end{align}  

We approximate the iterates $(\mathbf{v}_n^k,\pi_n^k)^{\top}\in \Vo_{h_n,0}\times \Qo_{h_n}$, $k=1,\ldots,K_n$, $n\in \mathbb{N}$, solving the non-linear saddle point problem \eqref{eq:primal1} employing the Newton solver from \mbox{\texttt{PETSc}} (version~3.17.3, \textit{cf.}~\cite{LW10}), with an absolute tolerance~of $\tau_{\textup{abs}}= 1.0\times 10^{-8}$ and a relative tolerance of $\tau_{\textup{abs}}=1.0\times 10^{-10}$. The linear system emerging~in~each~Newton iteration is solved using a sparse direct solver from \texttt{MUMPS} (version~5.5.0,~\textit{cf.}~\cite{mumps}). In~the~\mbox{implementation}, the uniqueness of the pressure is enforced by adding a zero mean condition.

\subsection{General choice of the extra-stress tensor}

\hspace*{5mm}In the non-linear saddle point problem \eqref{eq:primal1}, 
the extra-stress tensor $\mathbf{S}\colon Q_T\times \mathbb{R}^{d\times d}_{\textup{sym}}\to\mathbb{R}^{d\times d}_{\mathrm{sym}}$,~${d\in \{2,3\}}$,  is \hspace*{-0.1mm}always \hspace*{-0.1mm}defined \hspace*{-0.1mm}by \hspace*{-0.1mm}\eqref{eq:example-stress}, \hspace*{-0.1mm}where \hspace*{-0.1mm}$\mu_0\hspace*{-0.15em}  \coloneqq \hspace*{-0.15em} \frac{1}{2}$, \hspace*{-0.1mm}$\delta\hspace*{-0.15em} \coloneqq\hspace*{-0.15em} 1.0\hspace*{-0.1em} \times\hspace*{-0.1em}  10^{-5}$, \hspace*{-0.1mm}and \hspace*{-0.1mm}the \hspace*{-0.1mm}power-law \hspace*{-0.1mm}index \hspace*{-0.1mm}is \hspace*{-0.1mm}at \hspace*{-0.1mm}least~\hspace*{-0.1mm}\mbox{$\log$-Hölder}~\hspace*{-0.1mm}con-tinuous (\textit{i.e.}, $p\hspace*{-0.1em}\in \hspace*{-0.1em}\mathcal{P}^{\log}(Q_T)$) with $p^-\hspace*{-0.1em}>\hspace*{-0.1em} \frac{3d+2}{d+2}$ (\textit{i.e.}, $p^-\hspace*{-0.1em}>\hspace*{-0.1em} 2$ in the case $d\hspace*{-0.1em}=\hspace*{-0.1em}2$ and $p^-\hspace*{-0.1em}>\hspace*{-0.1em}2.2$~in~the~case~${d\hspace*{-0.1em}=\hspace*{-0.1em}3}$).

\subsection{(Weak) convergence of the velocity vector field} 

\hspace*{5mm}In this subsection, we carry out numerical experiments that are meant to indicate (weak) convergence of Scheme \ref{alg:p-NS-scheme} predicted by Proposition \ref{prop:p-NS-scheme}.
To this end, we construct manufactured solutions with fractional regularity, gradually reduce the fractional regularity parameters, and measure~convergence~rates that reduce with decreasing (but are stable 
for fixed) value of the fractional regularity parameters.

\textit{$\bullet$  Power-law index}. The power-law index $p\in C^{0,\alpha}(\overline{Q_T})$, where $T=0.1$ and $\alpha\in  (0,1]$,  for every $(t,x)^\top\in \overline{Q_T}$ is defined by 
\begin{align*}
	p(t,x)\coloneqq \Big(1-\frac{\vert x\vert^{\alpha}}{2^{\alpha/2}}\Big)\, p^++\frac{\vert x\vert^{\alpha}}{2^{\alpha/2}}\,(p^-+t)\,,
\end{align*}
where $p^-,p^+>1$.\newpage

\textit{$\bullet$ \hspace*{-0.1mm}Manufactured \hspace*{-0.1mm}solutions}. \hspace*{-0.1mm}As \hspace*{-0.1mm}manufactured solutions \hspace*{-0.1mm}serve \hspace*{-0.1mm}the \hspace*{-0.1mm}velocity \hspace*{-0.1mm}vector \hspace*{-0.1mm}field \hspace*{-0.1mm}${\overline{\boldsymbol{v}}\hspace*{-0.15em}\colon\hspace*{-0.15em}\overline{Q_T}\hspace*{-0.15em}\to\hspace*{-0.15em} \mathbb{R}^2}$~\hspace*{-0.1mm}and~\hspace*{-0.1mm}the kinematic pressure $\overline{\pi} \colon \overline{Q_T}\to \mathbb{R}$, given $\rho_{\overline{\boldsymbol{v}}},\rho_{\overline{\pi}}\hspace*{-0.1em}\in \hspace*{-0.1em}C^{0,\alpha}(\overline{Q_T})$, 
for every $(t,x)^\top\hspace*{-0.1em}\coloneqq\hspace*{-0.1em}(t,x_1,x_2)^\top\hspace*{-0.1em}\in\hspace*{-0.1em} \overline{Q_T}$~defined~by
\begin{align}\label{solutions}
\smash{\overline{\boldsymbol{v}}(t,x)\coloneqq t\,\vert x\vert^{\rho_{\overline{\boldsymbol{v}}}(t,x)} (x_2,-x_1)^\top\,, \qquad \overline{\pi}(t,x)\coloneqq 25\,t\,(\vert x\vert^{\rho_{\overline{\pi}}(t,x)}-\langle\,\vert \cdot\vert^{\rho_{\overline{\pi}}(t,\cdot)}\,\rangle_\Omega)\,,}
\end{align}
\textit{i.e.}, we choose the right-hand side data $\boldsymbol{f}^*\hspace*{-0.1em}\in\hspace*{-0.1em} (\mathbfcal{V}^{q,p}_{\mathbf{D}_x})^*$ (\textit{i.e.}, $\boldsymbol{f}\hspace*{-0.1em}\in\hspace*{-0.1em} L^{(p^-)'}(Q_T;\mathbb{R}^2)$~and~$\boldsymbol{F}\hspace*{-0.1em}\in\hspace*{-0.1em} L^{p'(\cdot,\cdot)}(Q_T;\mathbb{R}^{2\times 2}_{\textup{sym}})$),  the \hspace*{-0.1mm}Dirichlet \hspace*{-0.1mm}boundary \hspace*{-0.1mm}data \hspace*{-0.1mm}$\boldsymbol{v}_{\Gamma_T}\hspace*{-0.15em}\in\hspace*{-0.15em}{L^{p^-}(I;W^{\smash{1-\frac{1}{p^-},p^-}}\hspace*{-0.1em}(\partial\Omega;\mathbb{R}^2))}$, \hspace*{-0.1mm}and \hspace*{-0.1mm}the \hspace*{-0.1mm}initial \hspace*{-0.1mm}velocity~\hspace*{-0.1mm}$\mathbf{v}_0\hspace*{-0.15em}\in\hspace*{-0.15em} Y$~\hspace*{-0.1mm}\mbox{accordingly}.

\textit{$\bullet$  Regularity of velocity vector field.} Concerning the fractional  regularity of the velocity vector field, for a fractional regularity parameter  $\beta\in (0,1]$, we choose\vspace*{-1mm}
\begin{align}\label{rho_v}
	\rho_{\overline{\boldsymbol{v}}}\coloneqq 2\frac{\beta-1}{p}+\delta \in C^{0,\alpha}(\overline{Q_T})\,,
\end{align}
which just yields that\vspace*{-0.5mm}
\begin{align}\label{regurarity_v}
	\begin{aligned}
	\mathbf{F}(\cdot,\cdot,\mathbf{D}_x\overline{\boldsymbol{v}})&\in L^2(I;W^{\beta,2}(\Omega;\mathbb{R}^{2\times 2}_{\textup{sym}}))\cap W^{\beta,2}(I;L^2(\Omega;\mathbb{R}^{2\times 2}_{\textup{sym}}))\,,\\ 
	\overline{\boldsymbol{v}}&\in L^\infty(I;W^{\beta,2}(\Omega;\mathbb{R}^2))\,.
\end{aligned}
\end{align}

\textit{$\bullet$ Regularity of kinematic pressure.} Concerning the fractional   regularity of the kinematic pressure, for a fractional regularity parameter    $\gamma \in (0,1]$, we choose\vspace*{-1mm}
\begin{align}\label{rho_pi}
	\rho_{\overline{\pi}}=\gamma-\frac{2}{p'} +\delta\in   C^{0,\alpha}(\overline{Q_T})\,,
\end{align}
which just yields that\vspace*{-0.5mm}
\begin{align}\label{regurarity_pi} 
\overline{\pi}\in \mathcal{Q}_{\nabla_{\! x}}^p\coloneqq \left\{ \xi\in \smash{L^{p'(\cdot,\cdot)}(Q_T)}\;\bigg|\; \begin{aligned}&\vert \nabla_{\! x} ^{\gamma}\xi\vert \in \smash{L^{p'(\cdot,\cdot)}(Q_T)}\,,\;\\& \xi(t)\in \smash{H^{\gamma,p'(t,\cdot)}(\Omega)\cap L^{p'(t,\cdot)}_0(\Omega)} \text{ for a.e.\ }t\in I\end{aligned}\right\}\,,
\end{align} 
where the space $ \smash{H^{\gamma,p'(t,\cdot)}(\Omega)}$, for a.e.\ $t\in I$, denotes the fractional variable Haj\l asz--Sobolev space and $\vert \nabla_{\! x}^{\gamma}(\cdot)\vert$ the minimal $\gamma$-order Haj\l asz gradient with respect to the space variable (\textit{cf}.\ \cite{ZS22}).

\textit{$\bullet$ Comment on the choice of $p$, $\rho_{\overline{\boldsymbol{v}}}$, $\rho_{\overline{\pi}}$}. The power-law index $p\hspace*{-0.1em}\in\hspace*{-0.1em} C^{0,\alpha}(\overline{Q_T})$ is constructed~in~such~a~way that $p(t,x)\hspace*{-0.1em}\approx\hspace*{-0.1em} p^+$ for all $t\hspace*{-0.1em}\in\hspace*{-0.1em} I$ close to the origin (\textit{i.e.}, when $\vert x\vert\hspace*{-0.1em}\approx\hspace*{-0.1em}0$), where both~the~\mbox{velocity}~\mbox{vector}~field~and the \hspace*{-0.1mm}kinematic \hspace*{-0.1mm}pressure \hspace*{-0.1mm}(\textit{cf}.\ \hspace*{-0.1mm}\eqref{solutions})
\hspace*{-0.1mm}have \hspace*{-0.1mm}singularities \hspace*{-0.1mm}enforcing \hspace*{-0.1mm}the \hspace*{-0.1mm}fractional~\hspace*{-0.1mm}\mbox{regularity}~\hspace*{-0.1mm}(\textit{i.e.},~\hspace*{-0.1mm}\eqref{regurarity_v}~\hspace*{-0.1mm}and~\hspace*{-0.1mm}\eqref{regurarity_pi}). This is needed to create a critical setup that allows us to measure fractional convergence rates to be expec-ted~by~the contributions \cite{breit-mensah,BK24}. It proved important~to~use~time-~and~\mbox{space-dependent}~power~\mbox{functions} $\rho_{\overline{\boldsymbol{v}}},\rho_{\overline{\pi}}\in C^{0,\alpha}(\overline{Q_T})$ in \eqref{solutions}, since, then, the asymptotic behavior of the power-law~index~$p\in C^{0,\alpha}(\overline{Q_T})$ (\textit{i.e.}, $p(t,x)\approx p^+$  for all $t\in I$ when $\vert x\vert \approx 0$) transfers to these power~\mbox{functions},~\textit{i.e.}, due~to~\eqref{rho_v}~and~\eqref{rho_pi}, it holds that $\rho_{\overline{\boldsymbol{v}}}(t,x)\approx 2\frac{\beta-1}{p^+}+\delta$ for all $t\in I$ and $\rho_{\overline{\pi}}(t,x)\approx \gamma-
\smash{\frac{2}{(p^+)'}}+\delta$~for~all~$t\in  I$~when~$\vert x\vert \approx0$. 

\textit{$\bullet$ Discretization of time-space cylinder}.
First, we construct an initial triangulation $\mathcal
T_{h_0}$, where $h_0=1$, by subdividing the domain $\Omega=(0,1)^2$ along its diagonals into four triangles  with different orientations.  Finer triangulations~$\mathcal T_{h_n}$, $n\hspace*{-0.1em}=\hspace*{-0.1em}1,\ldots,6$, where $h_{n+1}\hspace*{-0.1em}=\hspace*{-0.1em}\frac{h_n}{2}$ for all $n\hspace*{-0.1em}=\hspace*{-0.1em}0,\ldots,6$, are, then,
obtained~by~\mbox{regular} subdivision of the previous triangulation, \textit{i.e.}, each triangle is subdivided
into four triangles by connecting the midpoints of the edges. Moreover, we employ the~time~\mbox{step-sizes}~${\tau_n\coloneqq 2^{-n-2}} $,~\textit{i.e.},~${K_n\coloneqq 2^{n+2}}$,~${n\in \mathbb{N}}$.\enlargethispage{9mm}

\textit{$\bullet$ \hspace*{-0.1mm}Alternative \hspace*{-0.1mm}time \hspace*{-0.1mm}discretization \hspace*{-0.1mm}of \hspace*{-0.1mm}right-hand \hspace*{-0.1mm}side}.
\hspace*{-0.1mm}For \hspace*{-0.1mm}a \hspace*{-0.1mm}simple \hspace*{-0.1mm}implementation,~\hspace*{-0.1mm}for~\hspace*{-0.1mm}each~\hspace*{-0.1mm}${k\hspace*{-0.15em}=\hspace*{-0.15em}1,\dots ,K}$, we \hspace*{-0.1mm}replace \hspace*{-0.1mm}the \hspace*{-0.1mm}$k$-th.\ \hspace*{-0.1mm}temporal \hspace*{-0.1mm}means \hspace*{-0.1mm}$\langle\boldsymbol{f}_n\rangle_k\hspace*{-0.175em}\in \hspace*{-0.175em}(\mathbb{P}^0(\mathcal{T}_{h_n}))^2$ \hspace*{-0.1mm}and \hspace*{-0.1mm}$\langle\boldsymbol{F}_n\rangle_k\hspace*{-0.175em}\in\hspace*{-0.175em} (\mathbb{P}^0(\mathcal{T}_{h_n}))^{2\times 2}_{\textup{sym}}$~\hspace*{-0.1mm}by~\hspace*{-0.1mm}${\boldsymbol{f}_n(t_k)\hspace*{-0.175em}\in \hspace*{-0.175em}(\mathbb{P}^0(\mathcal{T}_{h_n}))^2}$ and $\boldsymbol{F}_n(t_k)\in (\mathbb{P}^0(\mathcal{T}_{h_n}))^{2\times 2}_{\textup{sym}}$,~\mbox{respectively}.
Since the manufactured solutions (\textit{cf}.\ \eqref{solutions}) are smooth in  time and since 
\eqref{regurarity_v}, \hspace*{-0.1mm}\eqref{regurarity_pi} \hspace*{-0.1mm}imply \hspace*{-0.1mm}higher \hspace*{-0.1mm}integrability \hspace*{-0.1mm}of \hspace*{-0.1mm}the \hspace*{-0.1mm}right-hand~\hspace*{-0.1mm}side,~\hspace*{-0.1mm}we~\hspace*{-0.1mm}have~\hspace*{-0.1mm}that
 $\partial_t\boldsymbol{f}\hspace*{-0.1em}\in \hspace*{-0.1em} \smash{L^{(p^-)'}(Q_T;\mathbb{R}^2)}$ and $\boldsymbol{F},\partial_t\boldsymbol{F}\hspace*{-0.1em}\in\hspace*{-0.1em} L^{p'(\cdot,\cdot)+\eta}(Q_T;\mathbb{R}^{2\times 2}_{\textup{sym}})$ for some $\eta>0$, which, 
by  the (local)~\mbox{$L^2$-stability} of $\Pi_{\tau_n}^{0,\mathrm{t}}$ and $ \Pi_{h_n}^{0,\mathrm{x}}$, the estimates
 $\vert \langle \boldsymbol{f}\rangle_k-\boldsymbol{f}(t_k)\vert \leq \langle\vert\partial_t \boldsymbol{f}\vert\rangle_k$ and $\vert \langle \boldsymbol{F}\rangle_k-\boldsymbol{F}(t_k)\vert \leq \langle\vert\partial_t \boldsymbol{F}\vert\rangle_k$ for all $k=1,\ldots,K_n$, $n\in\mathbb{N}$,  
 and the variable Hölder inequality \eqref{eq:gen_hoelder},
 for every $k=1,\ldots,K_n$ and $T\in \mathcal{T}_{h_n}$, $n\in \mathbb{N}$, implies that 
\begin{align}
	\rho_{(p^-)',T}(\langle \boldsymbol{f}_n\rangle_k-\boldsymbol{f}_n(t_k)) 
		&	\leq  \smash{\tau_n^{(p^-)'}} \, 	\rho_{(p^-)',T}(\partial_t \boldsymbol{f})\,,\hspace{4cm}\label{eq:local1}\\
			\label{eq:local2}
		\rho_{p'_n(t_k,x_T),T}
		(\langle\boldsymbol{F}_n\rangle_k-\boldsymbol{F}_n(t_k)) 
		&\leq 
	\smash{\tau_n^{(p^+)'}}\,2^{p^++\eta}\big(\vert T\vert +\tau_n^{-1}\rho_{p'(\cdot,\cdot)+\eta,I_k\times T}
		(\partial_t \boldsymbol{F})\big) \,.
\end{align} 
Multiplying \hspace*{-0.1mm}\eqref{eq:local1}, \hspace*{-0.1mm}\eqref{eq:local2} \hspace*{-0.1mm}with \hspace*{-0.1mm}$\tau_n$ \hspace*{-0.1mm}and \hspace*{-0.1mm}summing \hspace*{-0.1mm}with \hspace*{-0.1mm}respect \hspace*{-0.1mm}to \hspace*{-0.1mm}$(k,T)^\top\hspace*{-0.2em}\in\hspace*{-0.15em} \{1,\ldots,K_n\}\times \mathcal{T}_{h_n}$,~\hspace*{-0.1mm}${n\hspace*{-0.15em}\in \hspace*{-0.15em}\mathbb{N}}$,~\hspace*{-0.1mm}shows~\hspace*{-0.1mm}that 
replacing \hspace*{-0.1mm}$\langle\boldsymbol{f}_n\rangle_k\hspace*{-0.175em}\in\hspace*{-0.175em} (\mathbb{P}^0(\mathcal{T}_{h_n}))^2$ \hspace*{-0.1mm}and \hspace*{-0.1mm}$\langle\boldsymbol{F}_n\rangle_k\hspace*{-0.175em}\in\hspace*{-0.15em} (\mathbb{P}^0(\mathcal{T}_{h_n}))^{2\times 2}_{\textup{sym}}$ 
\hspace*{-0.1mm}by \hspace*{-0.1mm}$\boldsymbol{f}_n (t_k)\hspace*{-0.175em}\in\hspace*{-0.175em} (\mathbb{P}^0(\mathcal{T}_{h_n}))^2$~\hspace*{-0.1mm}and~\hspace*{-0.1mm}${\boldsymbol{F}_n(t_k)\hspace*{-0.175em}\in\hspace*{-0.175em} (\mathbb{P}^0(\mathcal{T}_{h_n}))^{2\times 2}_{\textup{sym}}}$, respectively, for all $k\hspace*{-0.1em}=\hspace*{-0.1em} 1,\dots ,K_n$, $n\hspace*{-0.1em}\in\hspace*{-0.1em}\mathbb{N}$, does  not have an  effect on the best possible 
convergence rate $\min\{\alpha,\beta,\gamma(p^+)'/2\}$ to be expected by the contributions  \cite{breit-mensah,BK24} for the natural error quantities~in~\eqref{eq:errors}.\newpage

Then, we compute the iterates $(\mathbf{v}_n^k,\pi_n^k)^\top\in \Vo_{h_n,0}\times \Qo_{h_n}$, $k=0,\ldots,K_n$,~${n=0,\ldots,6}$, solving \eqref{eq:primal1}, the piece-wise constant interpolants $\overline{\boldsymbol{v}}_n\coloneqq \overline{\boldsymbol{v}}_n^{\tau_n}\in \mathbb{P}^0(\mathcal{I}_{\tau_n};\Vo_{h_n,0})$, $n=0,\ldots,6$, 
and,~for~every~$n=0,\ldots,6$, the natural error quantities\vspace{-1mm}
\begin{align}\label{eq:errors}
	\begin{aligned}
		e_{\mathbf{D\overline{\boldsymbol{v}}},n}&\coloneqq \bigg(\sum_{k=1}^{K_n}{\tau_n\,\|\mathbf{F}_n(t_k,\cdot,\mathbf{D}_x\mathbf{v}_n^k)-\mathbf{F}_n(t_k,\cdot,\mathbf{D}_x\overline{\boldsymbol{v}}(t_k))\|_{2,\Omega}^2}\bigg)^{\smash{\frac{1}{2}}}\,,\\[-0.5mm]
			e_{\smash{\overline{\boldsymbol{S}}},n}&\coloneqq \bigg(\sum_{k=1}^{K_n}{\tau_n\,\|\mathbf{F}_n^*(t_k,\cdot,\mathbf{S}_n(t_k,\cdot,\mathbf{D}_x\mathbf{v}_n^k))-\mathbf{F}_n^*(t_k,\cdot,\mathbf{S}(t_k,\cdot,\mathbf{D}_x\overline{\boldsymbol{v}}(t_k)))\|_{2,\Omega}^2}\bigg)^{\smash{\frac{1}{2}}}\,,\\[-0.5mm]
		e_{\overline{\boldsymbol{v}},n}&\coloneqq \max_{k=1,\ldots,K_n}{\|\mathbf{v}_n^k-\overline{\boldsymbol{v}}(t_k)\|_{2,\Omega}} \,,
		\\[-0.5mm]
		e_{\overline{\pi},n}&\coloneqq \bigg(\sum_{k=1}^{K_n}{\tau_n\,
			\Big\|\big((\delta+\vert \mathbf{D}\overline{\boldsymbol{v}}(t_k)\vert)^{p_n(t_k,\cdot)-1}+\vert \pi_n^k-\overline{\pi}(t_k)\vert\big)^{\smash{\frac{(p_n)'(t_k,\cdot)-2}{2}}}\pi_n^k-\overline{\pi}(t_k)\Big\|^2_{2,\Omega}} \bigg)^{\smash{\frac{1}{2}}}\,,
	\end{aligned}
\end{align}
where  
$\mathbf{F}_n,\mathbf{F}_n^*\colon  Q_T\times \mathbb{R}^{2\times 2}_{\textup{sym}}\to  \mathbb{R}^{d\times d}_{\textup{sym}}$,  for a.e.\ $(t,x)^\top\in Q_T$ and every $\mathbf{A}\in \mathbb{R}^{2\times 2}_{\textup{sym}}$, are defined by\enlargethispage{12mm}
\begin{align*}
	\mathbf{F}_n(t,x,\mathbf{A})\coloneqq (\delta+\vert  \mathbf{A}\vert)^{\smash{\frac{p_n(t,x)-2}{2}}}\mathbf{A}\,,\qquad
    \mathbf{F}_n^*(t,x,\mathbf{A})\coloneqq (\delta^{p_n(t,x)-1}+\vert  \mathbf{A}\vert)^{\smash{\frac{p_n'(t,x)-2}{2}}}\mathbf{A}\,.
\end{align*}

In order to measure convergence rates,  we compute experimental order of convergence~(EOC),~\textit{i.e.},\vspace*{-0.5mm}
\begin{align} \label{eq:eoc}
	\texttt{EOC}_n(e_n)\coloneqq\frac{\log\big(\frac{e_n}{e_{n-1}}\big)}{\log\big(\frac{h_n+\tau_n}{h_{n-1}+\tau_{n-1}}\big)}\,, \quad n=1,\ldots,6\,,
\end{align}
where, for every $n= 0,\ldots,6$, we denote by $e_n$
either 
$e_{\mathbf{D\overline{\boldsymbol{v}}},n}$, $e_{\smash{\overline{\boldsymbol{S}}},n}$,
$e_{\overline{\boldsymbol{v}},n}$, and $e_{\overline{\pi},n}$, respectively.

Motivated by the discussion in \cite[Lem.\ 3.4]{BK24}, we restrict to the particular case $\alpha=\beta=\gamma$.~In~addition, we always
set $p^+\hspace*{-0.1em}\coloneqq \hspace*{-0.1em}p^-+1$. Then, with regard to the contributions \cite{breit-mensah,BK24}, 
we~expect~the~convergence~rate $	\texttt{EOC}_n(e_n)=\alpha(p^+)'/2$ for the error quantities $e_n\in \{e_{\mathbf{D\overline{\boldsymbol{v}}},n},e_{\smash{\overline{\boldsymbol{S}}},n},e_{\overline{\boldsymbol{v}},n},e_{\overline{\pi},n}\}$, $n=1,\ldots,6$,~(\textit{cf}.~\eqref{eq:errors}).

For different values of $p^- \in \{2.25, 2.5, 2.75\}$, fractional exponents $\alpha =\beta=\gamma\in \{1.0,0.5,0.25,0.125\}$, and   triangulations $\mathcal{T}_{h_n}$,
$n = 0, \ldots , 6$, obtained by uniform mesh refinement  as described above, the EOCs (\textit{cf}.\ \eqref{eq:eoc}) with respect to the natural error quantities \eqref{eq:errors} are
computed and presented in the~Tables~\mbox{\ref{tab:1}--\ref{tab:8}}:
for both the Mini and the Taylor--Hood element, for the errors $e_n\in \{e_{\mathbf{D\overline{\boldsymbol{v}}},n}, e_{\smash{\overline{\boldsymbol{S}}},n},e_{\overline{\pi},n}\}$, $n=0,\ldots,6$,   we report
the expected convergence rate of about $\texttt{EOC}_n(e_n) \approx \alpha (p^+)'/2$,~$n=1,\ldots,6$,~while for the errors $e_{\overline{\boldsymbol{v}},n}$, \hspace*{-0.1mm}$n\hspace*{-0.1em}=\hspace*{-0.1em}0,\ldots,6$,  \hspace*{-0.1mm}we \hspace*{-0.1mm}report \hspace*{-0.1mm}the \hspace*{-0.1mm}increased \hspace*{-0.1mm}convergence \hspace*{-0.1mm}rate  \hspace*{-0.1mm}${\texttt{EOC}_n(e_{\overline{\boldsymbol{v}},n}) \gtrapprox  1}$,~\hspace*{-0.1mm}${n\hspace*{-0.1em}=\hspace*{-0.1em}1,\ldots,6}$,~which~\mbox{previously} has  been reported in related fully-discrete approximations of the unsteady $p$-Navier--Stokes equations. For small fractional exponents, \textit{i.e.}, 
$\alpha =\beta=\gamma\in \{0.25,0.125\}$,
the reported convergence rates are smaller than \hspace*{-0.1mm}the \hspace*{-0.1mm}expected \hspace*{-0.1mm}convergence \hspace*{-0.1mm}rates, \hspace*{-0.1mm}but \hspace*{-0.1mm}seem \hspace*{-0.1mm}to \hspace*{-0.1mm}approach \hspace*{-0.1mm}the \hspace*{-0.1mm}expected \hspace*{-0.1mm}convergence~\hspace*{-0.1mm}rates~\hspace*{-0.1mm}\mbox{asymptotically}. In summary, the  convergence rates are stable, indicating that Scheme~\ref{alg:p-NS-scheme}~converges~(at~least)~weakly.\vspace*{-1mm}
\enlargethispage{2mm}

\begin{table}[H]
	\setlength\tabcolsep{4.0pt}
	\centering
	\begin{tabular}{c |c|c|c|c|c|c|c|c|c|c|c|c|} \cmidrule(){1-13}
		\multicolumn{1}{|c||}{\cellcolor{lightgray}$\alpha$}	
		& \multicolumn{3}{c||}{\cellcolor{lightgray}$1.0$}   & \multicolumn{3}{c||}{\cellcolor{lightgray}$0.5$} & \multicolumn{3}{c||}{\cellcolor{lightgray}$0.25$} & \multicolumn{3}{c}{\cellcolor{lightgray}$0.125$}\\ 
		\hline 
		
		\multicolumn{1}{|c||}{\cellcolor{lightgray}\diagbox[height=1.1\line,width=0.11\dimexpr\linewidth]{\vspace{-0.6mm}$i$}{\\[-5mm] $p^-$}}
		& \cellcolor{lightgray}2.25 & \cellcolor{lightgray}2.5  & \multicolumn{1}{c||}{\cellcolor{lightgray}2.75}
		& \cellcolor{lightgray}2.25 & \cellcolor{lightgray}2.5  & \multicolumn{1}{c||}{\cellcolor{lightgray}2.75}
		& \cellcolor{lightgray}2.25 & \cellcolor{lightgray}2.5  & \multicolumn{1}{c||}{\cellcolor{lightgray}2.75}
		& \cellcolor{lightgray}2.25 & \cellcolor{lightgray}2.5  & \cellcolor{lightgray}2.75  \\ \hline\hline 
\multicolumn{1}{|c||}{\cellcolor{lightgray}$3$}             & 0.668 & 0.650 & \multicolumn{1}{c||}{0.634} & \multicolumn{1}{c|}{0.329} & 0.328 & \multicolumn{1}{c||}{0.325} & \multicolumn{1}{c|}{0.142} & 0.149 & \multicolumn{1}{c||}{0.153} & \multicolumn{1}{c|}{0.049} & 0.061 & 0.068 \\ \hline
\multicolumn{1}{|c||}{\cellcolor{lightgray}$4$}             & 0.704 & 0.684 & \multicolumn{1}{c||}{0.667} & \multicolumn{1}{c|}{0.344} & 0.339 & \multicolumn{1}{c||}{0.334} & \multicolumn{1}{c|}{0.152} & 0.156 & \multicolumn{1}{c||}{0.157} & \multicolumn{1}{c|}{0.056} & 0.064 & 0.069 \\ \hline
\multicolumn{1}{|c||}{\cellcolor{lightgray}$5$}             & 0.716 & 0.695 & \multicolumn{1}{c||}{0.677} & \multicolumn{1}{c|}{0.351} & 0.343 & \multicolumn{1}{c||}{0.337} & \multicolumn{1}{c|}{0.159} & 0.160 & \multicolumn{1}{c||}{0.160} & \multicolumn{1}{c|}{0.060} & 0.066 & 0.070 \\ \hline 
\multicolumn{1}{|c||}{\cellcolor{lightgray}$6$}             & 0.720 & 0.698 & \multicolumn{1}{c||}{0.680} & \multicolumn{1}{c|}{0.355} & 0.346 & \multicolumn{1}{c||}{0.338} & \multicolumn{1}{c|}{0.164} & 0.163 & \multicolumn{1}{c||}{0.162} & \multicolumn{1}{c|}{0.064} & 0.068 & 0.071 \\ \hline\hline
		\multicolumn{1}{|c||}{\cellcolor{lightgray}\small expected} & 0.722 & 0.700 & \multicolumn{1}{c||}{0.682} & \multicolumn{1}{c|}{0.361} & 0.350 & \multicolumn{1}{c||}{0.341} & \multicolumn{1}{c|}{0.181} & 0.175 & \multicolumn{1}{c||}{1.171} & \multicolumn{1}{c|}{0.091} & 0.088 & 0.086 \\ \hline
	\end{tabular}\vspace{-3mm}
	\caption{Experimental order of convergence (MINI): $\texttt{EOC}_n(e_{\mathbf{D\overline{\boldsymbol{v}}},n})$,~${n=3,\dots,6}$.}
	\label{tab:1}
		\begin{tabular}{c |c|c|c|c|c|c|c|c|c|c|c|c|} \cmidrule(){1-13}
		\multicolumn{1}{|c||}{\cellcolor{lightgray}$\alpha$}	
		& \multicolumn{3}{c||}{\cellcolor{lightgray}$1.0$}   & \multicolumn{3}{c||}{\cellcolor{lightgray}$0.5$} & \multicolumn{3}{c||}{\cellcolor{lightgray}$0.25$} & \multicolumn{3}{c}{\cellcolor{lightgray}$0.125$}\\ 
		\hline 
		
		\multicolumn{1}{|c||}{\cellcolor{lightgray}\diagbox[height=1.1\line,width=0.11\dimexpr\linewidth]{\vspace{-0.6mm}$i$}{\\[-5mm] $p^-$}}
		& \cellcolor{lightgray}2.25 & \cellcolor{lightgray}2.5  & \multicolumn{1}{c||}{\cellcolor{lightgray}2.75}
		& \cellcolor{lightgray}2.25 & \cellcolor{lightgray}2.5  & \multicolumn{1}{c||}{\cellcolor{lightgray}2.75}
		& \cellcolor{lightgray}2.25 & \cellcolor{lightgray}2.5  & \multicolumn{1}{c||}{\cellcolor{lightgray}2.75}
		& \cellcolor{lightgray}2.25 & \cellcolor{lightgray}2.5  & \cellcolor{lightgray}2.75  \\ \hline\hline 
\multicolumn{1}{|c||}{\cellcolor{lightgray}$3$}             & 0.663 & 0.645 & \multicolumn{1}{c||}{0.630} & \multicolumn{1}{c|}{0.320} & 0.320 & \multicolumn{1}{c||}{0.319} & \multicolumn{1}{c|}{0.130} & 0.140 & \multicolumn{1}{c||}{0.146} & \multicolumn{1}{c|}{0.039} & 0.053 & 0.062 \\ \hline
\multicolumn{1}{|c||}{\cellcolor{lightgray}$4$}             & 0.701 & 0.681 & \multicolumn{1}{c||}{0.665} & \multicolumn{1}{c|}{0.337} & 0.333 & \multicolumn{1}{c||}{0.329} & \multicolumn{1}{c|}{0.143} & 0.148 & \multicolumn{1}{c||}{0.152} & \multicolumn{1}{c|}{0.046} & 0.057 & 0.063 \\ \hline
\multicolumn{1}{|c||}{\cellcolor{lightgray}$5$}             & 0.714 & 0.693 & \multicolumn{1}{c||}{0.676} & \multicolumn{1}{c|}{0.346} & 0.340 & \multicolumn{1}{c||}{0.334} & \multicolumn{1}{c|}{0.151} & 0.154 & \multicolumn{1}{c||}{0.155} & \multicolumn{1}{c|}{0.052} & 0.060 & 0.065 \\ \hline
\multicolumn{1}{|c||}{\cellcolor{lightgray}$6$}             & 0.719 & 0.698 & \multicolumn{1}{c||}{0.680} & \multicolumn{1}{c|}{0.352} & 0.343 & \multicolumn{1}{c||}{0.336} & \multicolumn{1}{c|}{0.158} & 0.158 & \multicolumn{1}{c||}{0.158} & \multicolumn{1}{c|}{0.057} & 0.063 & 0.067 \\ \hline\hline
		\multicolumn{1}{|c||}{\cellcolor{lightgray}\small expected} & 0.722 & 0.700 & \multicolumn{1}{c||}{0.682} & \multicolumn{1}{c|}{0.361} & 0.350 & \multicolumn{1}{c||}{0.341} & \multicolumn{1}{c|}{0.181} & 0.175 & \multicolumn{1}{c||}{1.171} & \multicolumn{1}{c|}{0.091} & 0.088 & 0.086 \\ \hline
	\end{tabular}\vspace{-3mm}
	\caption{Experimental order of convergence (MINI): $\texttt{EOC}_n(e_{\smash{\overline{\boldsymbol{S}}},n})$,~${n=3,\dots,6}$.}
	\label{tab:2} 
	\end{table}
	\begin{table}[H]
		\setlength\tabcolsep{4.0pt}
		\centering
	\begin{tabular}{c |c|c|c|c|c|c|c|c|c|c|c|c|} \cmidrule(){1-13}
		\multicolumn{1}{|c||}{\cellcolor{lightgray}$\alpha$}	
		& \multicolumn{3}{c||}{\cellcolor{lightgray}$1.0$}   & \multicolumn{3}{c||}{\cellcolor{lightgray}$0.5$} & \multicolumn{3}{c||}{\cellcolor{lightgray}$0.25$} & \multicolumn{3}{c}{\cellcolor{lightgray}$0.125$}\\ 
		\hline 
		
		\multicolumn{1}{|c||}{\cellcolor{lightgray}\diagbox[height=1.1\line,width=0.11\dimexpr\linewidth]{\vspace{-0.6mm}$i$}{\\[-5mm] $p^-$}}
		& \cellcolor{lightgray}2.25 & \cellcolor{lightgray}2.5  & \multicolumn{1}{c||}{\cellcolor{lightgray}2.75}
		& \cellcolor{lightgray}2.25 & \cellcolor{lightgray}2.5  & \multicolumn{1}{c||}{\cellcolor{lightgray}2.75}
		& \cellcolor{lightgray}2.25 & \cellcolor{lightgray}2.5  & \multicolumn{1}{c||}{\cellcolor{lightgray}2.75}
		& \cellcolor{lightgray}2.25 & \cellcolor{lightgray}2.5  & \cellcolor{lightgray}2.75  \\ \hline\hline 
\multicolumn{1}{|c||}{\cellcolor{lightgray}$3$}             & 0.726 & 0.703 & \multicolumn{1}{c||}{0.685} & \multicolumn{1}{c|}{0.372} & 0.361 & \multicolumn{1}{c||}{0.352} & \multicolumn{1}{c|}{0.186} & 0.181 & \multicolumn{1}{c||}{0.177} & \multicolumn{1}{c|}{0.087} & 0.087 & 0.087 \\ \hline
\multicolumn{1}{|c||}{\cellcolor{lightgray}$4$}             & 0.725 & 0.703 & \multicolumn{1}{c||}{0.684} & \multicolumn{1}{c|}{0.369} & 0.357 & \multicolumn{1}{c||}{0.348} & \multicolumn{1}{c|}{0.183} & 0.177 & \multicolumn{1}{c||}{0.173} & \multicolumn{1}{c|}{0.084} & 0.083 & 0.083 \\ \hline
\multicolumn{1}{|c||}{\cellcolor{lightgray}$5$}             & 0.724 & 0.702 & \multicolumn{1}{c||}{0.683} & \multicolumn{1}{c|}{0.367} & 0.355 & \multicolumn{1}{c||}{0.346} & \multicolumn{1}{c|}{0.182} & 0.176 & \multicolumn{1}{c||}{0.172} & \multicolumn{1}{c|}{0.083} & 0.082 & 0.081 \\ \hline
\multicolumn{1}{|c||}{\cellcolor{lightgray}$6$}             & 0.724 & 0.701 & \multicolumn{1}{c||}{0.683} & \multicolumn{1}{c|}{0.365} & 0.354 & \multicolumn{1}{c||}{0.344} & \multicolumn{1}{c|}{0.182} & 0.176 & \multicolumn{1}{c||}{0.172} & \multicolumn{1}{c|}{0.084} & 0.082 & 0.081 \\ \hline\hline
		\multicolumn{1}{|c||}{\cellcolor{lightgray}\small expected} & 0.722 & 0.700 & \multicolumn{1}{c||}{0.682} & \multicolumn{1}{c|}{0.361} & 0.350 & \multicolumn{1}{c||}{0.341} & \multicolumn{1}{c|}{0.181} & 0.175 & \multicolumn{1}{c||}{1.171} & \multicolumn{1}{c|}{0.091} & 0.088 & 0.086 \\ \hline
	\end{tabular}\vspace{-3mm}
	\caption{Experimental order of convergence (MINI): $\texttt{EOC}_n(e_{\overline{\pi},n})$,~${n=3,\dots,6}$.}
	\label{tab:3}
	\begin{tabular}{c |c|c|c|c|c|c|c|c|c|c|c|c|} \cmidrule(){1-13}
		\multicolumn{1}{|c||}{\cellcolor{lightgray}$\alpha$}	
		& \multicolumn{3}{c||}{\cellcolor{lightgray}$1.0$}   & \multicolumn{3}{c||}{\cellcolor{lightgray}$0.5$} & \multicolumn{3}{c||}{\cellcolor{lightgray}$0.25$} & \multicolumn{3}{c}{\cellcolor{lightgray}$0.125$}\\ 
		\hline 
		
		\multicolumn{1}{|c||}{\cellcolor{lightgray}\diagbox[height=1.1\line,width=0.11\dimexpr\linewidth]{\vspace{-0.6mm}$i$}{\\[-5mm] $p^-$}}
		& \cellcolor{lightgray}2.25 & \cellcolor{lightgray}2.5  & \multicolumn{1}{c||}{\cellcolor{lightgray}2.75}
		& \cellcolor{lightgray}2.25 & \cellcolor{lightgray}2.5  & \multicolumn{1}{c||}{\cellcolor{lightgray}2.75}
		& \cellcolor{lightgray}2.25 & \cellcolor{lightgray}2.5  & \multicolumn{1}{c||}{\cellcolor{lightgray}2.75}
		& \cellcolor{lightgray}2.25 & \cellcolor{lightgray}2.5  & \cellcolor{lightgray}2.75  \\ \hline\hline 
\multicolumn{1}{|c||}{\cellcolor{lightgray}$3$}             & 1.638 & 1.620 & \multicolumn{1}{c||}{1.598} & \multicolumn{1}{c|}{1.501} & 1.517 & \multicolumn{1}{c||}{1.524} & \multicolumn{1}{c|}{1.390} & 1.428 & \multicolumn{1}{c||}{1.453} & \multicolumn{1}{c|}{1.291} & 1.347 & 1.389 \\ \hline
\multicolumn{1}{|c||}{\cellcolor{lightgray}$4$}             & 1.693 & 1.673 & \multicolumn{1}{c||}{1.651} & \multicolumn{1}{c|}{1.536} & 1.550 & \multicolumn{1}{c||}{1.557} & \multicolumn{1}{c|}{1.418} & 1.454 & \multicolumn{1}{c||}{1.479} & \multicolumn{1}{c|}{1.317} & 1.367 & 1.409 \\ \hline
\multicolumn{1}{|c||}{\cellcolor{lightgray}$5$}             & 1.709 & 1.680 & \multicolumn{1}{c||}{1.651} & \multicolumn{1}{c|}{1.549} & 1.550 & \multicolumn{1}{c||}{1.540} & \multicolumn{1}{c|}{1.427} & 1.457 & \multicolumn{1}{c||}{1.474} & \multicolumn{1}{c|}{1.335} & 1.368 & 1.401 \\ \hline
\multicolumn{1}{|c||}{\cellcolor{lightgray}$6$}             & 1.674 & 1.618 & \multicolumn{1}{c||}{1.564} & \multicolumn{1}{c|}{1.529} & 1.487 & \multicolumn{1}{c||}{1.411} & \multicolumn{1}{c|}{1.416} & 1.415 & \multicolumn{1}{c||}{1.401} & \multicolumn{1}{c|}{1.354} & 1.354 & 1.340  \\ \hline\hline
		\multicolumn{1}{|c||}{\cellcolor{lightgray}\small expected} & 0.722 & 0.700 & \multicolumn{1}{c||}{0.682} & \multicolumn{1}{c|}{0.361} & 0.350 & \multicolumn{1}{c||}{0.341} & \multicolumn{1}{c|}{0.181} & 0.175 & \multicolumn{1}{c||}{1.171} & \multicolumn{1}{c|}{0.091} & 0.088 & 0.086 \\ \hline
	\end{tabular}\vspace{-3mm}
	\caption{Experimental order of convergence (MINI): $\texttt{EOC}_n(e_{\overline{\boldsymbol{v}},n})$,~${n=3,\dots,6}$.}
	\label{tab:4}
	\begin{tabular}{c |c|c|c|c|c|c|c|c|c|c|c|c|} \cmidrule(){1-13}
		\multicolumn{1}{|c||}{\cellcolor{lightgray}$\alpha$}	
		& \multicolumn{3}{c||}{\cellcolor{lightgray}$1.0$}   & \multicolumn{3}{c||}{\cellcolor{lightgray}$0.5$} & \multicolumn{3}{c||}{\cellcolor{lightgray}$0.25$} & \multicolumn{3}{c}{\cellcolor{lightgray}$0.125$}\\ 
		\hline 
		
		\multicolumn{1}{|c||}{\cellcolor{lightgray}\diagbox[height=1.1\line,width=0.11\dimexpr\linewidth]{\vspace{-0.6mm}$i$}{\\[-5mm] $p^-$}}
		& \cellcolor{lightgray}2.25 & \cellcolor{lightgray}2.5  & \multicolumn{1}{c||}{\cellcolor{lightgray}2.75}
		& \cellcolor{lightgray}2.25 & \cellcolor{lightgray}2.5  & \multicolumn{1}{c||}{\cellcolor{lightgray}2.75}
		& \cellcolor{lightgray}2.25 & \cellcolor{lightgray}2.5  & \multicolumn{1}{c||}{\cellcolor{lightgray}2.75}
		& \cellcolor{lightgray}2.25 & \cellcolor{lightgray}2.5  & \cellcolor{lightgray}2.75  \\ \hline\hline 
\multicolumn{1}{|c||}{\cellcolor{lightgray}$3$}             & 0.685 & 0.666 & \multicolumn{1}{c||}{0.650} & \multicolumn{1}{c|}{0.348} & 0.343 & \multicolumn{1}{c||}{0.337} & \multicolumn{1}{c|}{0.157} & 0.161 & \multicolumn{1}{c||}{0.163} & \multicolumn{1}{c|}{0.056} & 0.066 & 0.072 \\ \hline
\multicolumn{1}{|c||}{\cellcolor{lightgray}$4$}             & 0.711 & 0.690 & \multicolumn{1}{c||}{0.673} & \multicolumn{1}{c|}{0.356} & 0.348 & \multicolumn{1}{c||}{0.340} & \multicolumn{1}{c|}{0.165} & 0.165 & \multicolumn{1}{c||}{0.164} & \multicolumn{1}{c|}{0.062} & 0.068 & 0.072 \\ \hline
\multicolumn{1}{|c||}{\cellcolor{lightgray}$5$}             & 0.719 & 0.697 & \multicolumn{1}{c||}{0.679} & \multicolumn{1}{c|}{0.359} & 0.349 & \multicolumn{1}{c||}{0.341} & \multicolumn{1}{c|}{0.169} & 0.167 & \multicolumn{1}{c||}{0.165} & \multicolumn{1}{c|}{0.066} & 0.070 & 0.073 \\ \hline
\multicolumn{1}{|c||}{\cellcolor{lightgray}$6$}             & 0.721 & 0.699 & \multicolumn{1}{c||}{0.681} & \multicolumn{1}{c|}{0.360} & 0.350 & \multicolumn{1}{c||}{0.341} & \multicolumn{1}{c|}{0.172} & 0.169 & \multicolumn{1}{c||}{0.167} & \multicolumn{1}{c|}{0.070} & 0.073 & 0.074 \\ \hline\hline
		\multicolumn{1}{|c||}{\cellcolor{lightgray}\small expected} & 0.722 & 0.700 & \multicolumn{1}{c||}{0.682} & \multicolumn{1}{c|}{0.361} & 0.350 & \multicolumn{1}{c||}{0.341} & \multicolumn{1}{c|}{0.181} & 0.175 & \multicolumn{1}{c||}{1.171} & \multicolumn{1}{c|}{0.091} & 0.088 & 0.086 \\ \hline
	\end{tabular}\vspace{-3mm}
	\caption{Experimental order of convergence (Taylor--Hood): $\texttt{EOC}_n(e_{\mathbf{D\overline{\boldsymbol{v}}},n})$,~${n=3,\dots,6}$.}
	\label{tab:5}
	\begin{tabular}{c |c|c|c|c|c|c|c|c|c|c|c|c|} \cmidrule(){1-13}
		\multicolumn{1}{|c||}{\cellcolor{lightgray}$\alpha$}	
		& \multicolumn{3}{c||}{\cellcolor{lightgray}$1.0$}   & \multicolumn{3}{c||}{\cellcolor{lightgray}$0.5$} & \multicolumn{3}{c||}{\cellcolor{lightgray}$0.25$} & \multicolumn{3}{c}{\cellcolor{lightgray}$0.125$}\\ 
		\hline 
		
		\multicolumn{1}{|c||}{\cellcolor{lightgray}\diagbox[height=1.1\line,width=0.11\dimexpr\linewidth]{\vspace{-0.6mm}$i$}{\\[-5mm] $p^-$}}
		& \cellcolor{lightgray}2.25 & \cellcolor{lightgray}2.5  & \multicolumn{1}{c||}{\cellcolor{lightgray}2.75}
		& \cellcolor{lightgray}2.25 & \cellcolor{lightgray}2.5  & \multicolumn{1}{c||}{\cellcolor{lightgray}2.75}
		& \cellcolor{lightgray}2.25 & \cellcolor{lightgray}2.5  & \multicolumn{1}{c||}{\cellcolor{lightgray}2.75}
		& \cellcolor{lightgray}2.25 & \cellcolor{lightgray}2.5  & \cellcolor{lightgray}2.75  \\ \hline\hline 
	\multicolumn{1}{|c||}{\cellcolor{lightgray}$3$}             & 0.680 & 0.663 & \multicolumn{1}{c||}{0.648} & \multicolumn{1}{c|}{0.341} & 0.337 & \multicolumn{1}{c||}{0.333} & \multicolumn{1}{c|}{0.145} & 0.151 & \multicolumn{1}{c||}{0.155} & \multicolumn{1}{c|}{0.045} & 0.058 & 0.065 \\ \hline
	\multicolumn{1}{|c||}{\cellcolor{lightgray}$4$}             & 0.709 & 0.688 & \multicolumn{1}{c||}{0.672} & \multicolumn{1}{c|}{0.352} & 0.345 & \multicolumn{1}{c||}{0.338} & \multicolumn{1}{c|}{0.155} & 0.158 & \multicolumn{1}{c||}{0.159} & \multicolumn{1}{c|}{0.053} & 0.061 & 0.066 \\ \hline
	\multicolumn{1}{|c||}{\cellcolor{lightgray}$5$}             & 0.718 & 0.696 & \multicolumn{1}{c||}{0.679} & \multicolumn{1}{c|}{0.356} & 0.348 & \multicolumn{1}{c||}{0.340} & \multicolumn{1}{c|}{0.162} & 0.162 & \multicolumn{1}{c||}{0.161} & \multicolumn{1}{c|}{0.058} & 0.064 & 0.068 \\ \hline
	\multicolumn{1}{|c||}{\cellcolor{lightgray}$6$}             & 0.721 & 0.699 & \multicolumn{1}{c||}{0.681} & \multicolumn{1}{c|}{0.359} & 0.349 & \multicolumn{1}{c||}{0.341} & \multicolumn{1}{c|}{0.167} & 0.165 & \multicolumn{1}{c||}{0.163} & \multicolumn{1}{c|}{0.063} & 0.067 & 0.070 \\ \hline\hline
		\multicolumn{1}{|c||}{\cellcolor{lightgray}\small expected} & 0.722 & 0.700 & \multicolumn{1}{c||}{0.682} & \multicolumn{1}{c|}{0.361} & 0.350 & \multicolumn{1}{c||}{0.341} & \multicolumn{1}{c|}{0.181} & 0.175 & \multicolumn{1}{c||}{1.171} & \multicolumn{1}{c|}{0.091} & 0.088 & 0.086 \\ \hline
	\end{tabular}\vspace{-3mm}
	\caption{Experimental order of convergence (Taylor--Hood): $\texttt{EOC}_n(e_{\smash{\overline{\boldsymbol{S}}},n})$,~${n=3,\dots,6}$.}
	\label{tab:6} 
	\begin{tabular}{c |c|c|c|c|c|c|c|c|c|c|c|c|} \cmidrule(){1-13}
		\multicolumn{1}{|c||}{\cellcolor{lightgray}$\alpha$}	
		& \multicolumn{3}{c||}{\cellcolor{lightgray}$1.0$}   & \multicolumn{3}{c||}{\cellcolor{lightgray}$0.5$} & \multicolumn{3}{c||}{\cellcolor{lightgray}$0.25$} & \multicolumn{3}{c}{\cellcolor{lightgray}$0.125$}\\ 
		\hline 
		
		\multicolumn{1}{|c||}{\cellcolor{lightgray}\diagbox[height=1.1\line,width=0.11\dimexpr\linewidth]{\vspace{-0.6mm}$i$}{\\[-5mm] $p^-$}}
		& \cellcolor{lightgray}2.25 & \cellcolor{lightgray}2.5  & \multicolumn{1}{c||}{\cellcolor{lightgray}2.75}
		& \cellcolor{lightgray}2.25 & \cellcolor{lightgray}2.5  & \multicolumn{1}{c||}{\cellcolor{lightgray}2.75}
		& \cellcolor{lightgray}2.25 & \cellcolor{lightgray}2.5  & \multicolumn{1}{c||}{\cellcolor{lightgray}2.75}
		& \cellcolor{lightgray}2.25 & \cellcolor{lightgray}2.5  & \cellcolor{lightgray}2.75  \\ \hline\hline 
	\multicolumn{1}{|c||}{\cellcolor{lightgray}$3$}             & 0.728 & 0.706 & \multicolumn{1}{c||}{0.688} & \multicolumn{1}{c|}{0.372} & 0.361 & \multicolumn{1}{c||}{0.352} & \multicolumn{1}{c|}{0.182} & 0.179 & \multicolumn{1}{c||}{0.176} & \multicolumn{1}{c|}{0.082} & 0.084 & 0.085 \\ \hline
	\multicolumn{1}{|c||}{\cellcolor{lightgray}$4$}             & 0.725 & 0.703 & \multicolumn{1}{c||}{0.685} & \multicolumn{1}{c|}{0.368} & 0.357 & \multicolumn{1}{c||}{0.347} & \multicolumn{1}{c|}{0.180} & 0.176 & \multicolumn{1}{c||}{0.172} & \multicolumn{1}{c|}{0.079} & 0.080 & 0.081 \\ \hline
	\multicolumn{1}{|c||}{\cellcolor{lightgray}$5$}             & 0.724 & 0.702 & \multicolumn{1}{c||}{0.684} & \multicolumn{1}{c|}{0.366} & 0.355 & \multicolumn{1}{c||}{0.345} & \multicolumn{1}{c|}{0.179} & 0.175 & \multicolumn{1}{c||}{0.171} & \multicolumn{1}{c|}{0.080} & 0.080 & 0.080 \\ \hline 
	\multicolumn{1}{|c||}{\cellcolor{lightgray}$6$}             & 0.724 & 0.701 & \multicolumn{1}{c||}{0.683} & \multicolumn{1}{c|}{0.365} & 0.353 & \multicolumn{1}{c||}{0.344} & \multicolumn{1}{c|}{0.180} & 0.175 & \multicolumn{1}{c||}{0.171} & \multicolumn{1}{c|}{0.081} & 0.080 & 0.080 \\ \hline\hline
		\multicolumn{1}{|c||}{\cellcolor{lightgray}\small expected} & 0.722 & 0.700 & \multicolumn{1}{c||}{0.682} & \multicolumn{1}{c|}{0.361} & 0.350 & \multicolumn{1}{c||}{0.341} & \multicolumn{1}{c|}{0.181} & 0.175 & \multicolumn{1}{c||}{1.171} & \multicolumn{1}{c|}{0.091} & 0.088 & 0.086 \\ \hline
	\end{tabular}\vspace{-3mm}
	\caption{Experimental order of convergence (Taylor--Hood): $\texttt{EOC}_n(e_{\overline{\pi},n})$,~${n=3,\dots,6}$.}
	\label{tab:7}
	\begin{tabular}{c |c|c|c|c|c|c|c|c|c|c|c|c|} \cmidrule(){1-13}
		\multicolumn{1}{|c||}{\cellcolor{lightgray}$\alpha$}	
		& \multicolumn{3}{c||}{\cellcolor{lightgray}$1.0$}   & \multicolumn{3}{c||}{\cellcolor{lightgray}$0.5$} & \multicolumn{3}{c||}{\cellcolor{lightgray}$0.25$} & \multicolumn{3}{c}{\cellcolor{lightgray}$0.125$}\\ 
		\hline 
		
		\multicolumn{1}{|c||}{\cellcolor{lightgray}\diagbox[height=1.1\line,width=0.11\dimexpr\linewidth]{\vspace{-0.6mm}$i$}{\\[-5mm] $p^-$}}
		& \cellcolor{lightgray}2.25 & \cellcolor{lightgray}2.5  & \multicolumn{1}{c||}{\cellcolor{lightgray}2.75}
		& \cellcolor{lightgray}2.25 & \cellcolor{lightgray}2.5  & \multicolumn{1}{c||}{\cellcolor{lightgray}2.75}
		& \cellcolor{lightgray}2.25 & \cellcolor{lightgray}2.5  & \multicolumn{1}{c||}{\cellcolor{lightgray}2.75}
		& \cellcolor{lightgray}2.25 & \cellcolor{lightgray}2.5  & \cellcolor{lightgray}2.75  \\ \hline\hline 
	\multicolumn{1}{|c||}{\cellcolor{lightgray}$3$}             & 1.652 & 1.628 & \multicolumn{1}{c||}{1.600} & \multicolumn{1}{c|}{1.553} & 1.566 & \multicolumn{1}{c||}{1.571} & \multicolumn{1}{c|}{1.434} & 1.470 & \multicolumn{1}{c||}{1.496} & \multicolumn{1}{c|}{1.320} & 1.375 & 1.417 \\ \hline
	\multicolumn{1}{|c||}{\cellcolor{lightgray}$4$}             & 1.706 & 1.681 & \multicolumn{1}{c||}{1.654} & \multicolumn{1}{c|}{1.576} & 1.588 & \multicolumn{1}{c||}{1.593} & \multicolumn{1}{c|}{1.460} & 1.491 & \multicolumn{1}{c||}{1.516} & \multicolumn{1}{c|}{1.349} & 1.397 & 1.436 \\ \hline
	\multicolumn{1}{|c||}{\cellcolor{lightgray}$5$}             & 1.732 & 1.705 & \multicolumn{1}{c||}{1.675} & \multicolumn{1}{c|}{1.581} & 1.590 & \multicolumn{1}{c||}{1.595} & \multicolumn{1}{c|}{1.469} & 1.494 & \multicolumn{1}{c||}{1.512} & \multicolumn{1}{c|}{1.368} & 1.406 & 1.435 \\ \hline
	\multicolumn{1}{|c||}{\cellcolor{lightgray}$6$}             & 1.736 & 1.718 & \multicolumn{1}{c||}{1.686} & \multicolumn{1}{c|}{1.561} & 1.565 & \multicolumn{1}{c||}{1.566} & \multicolumn{1}{c|}{1.463} & 1.462 & \multicolumn{1}{c||}{1.462} & \multicolumn{1}{c|}{1.384} & 1.403 & 1.399 \\ \hline\hline
		\multicolumn{1}{|c||}{\cellcolor{lightgray}\small expected} & 0.722 & 0.700 & \multicolumn{1}{c||}{0.682} & \multicolumn{1}{c|}{0.361} & 0.350 & \multicolumn{1}{c||}{0.341} & \multicolumn{1}{c|}{0.181} & 0.175 & \multicolumn{1}{c||}{1.171} & \multicolumn{1}{c|}{0.091} & 0.088 & 0.086 \\ \hline
	\end{tabular}\vspace{-3mm}
	\caption{Experimental order of convergence (Taylor--Hood): $\texttt{EOC}_n(e_{\overline{\boldsymbol{v}},n})$,~${n=3,\dots,6}$.}\label{tab:8}
\end{table}

\subsection{An example for an electro-rheological fluid flow}

\hspace*{5mm}In this subsection, we examine a less academic example describing an unsteady \textit{electro-rheological fluid} flow in three dimensions.  Solutions to the unsteady $p(\cdot,\cdot)$-Navier--Stokes equations \eqref{eq:ptxNavierStokes}~also model  electro-rheological fluid flow behavior, if the right-hand side is given~via~${\boldsymbol{f}\hspace*{-0.2em}\coloneqq\hspace*{-0.2em}\widehat{\boldsymbol{f}}\hspace*{-0.175em}+\hspace*{-0.175em}\chi_E\textup{div}_x(\boldsymbol{E}\hspace*{-0.15em}\otimes\hspace*{-0.15em}\boldsymbol{E})\colon \hspace*{-0.175em}Q_T\hspace*{-0.2em}\to \hspace*{-0.2em}\mathbb{R}^3}$, where $\widehat{\boldsymbol{f}}\colon Q_T\to \mathbb{R}^3$ is a given \textit{mechanical body force},
$\chi_E>0$ the \textit{di-eletric susceptibility}, $\boldsymbol{E}\colon \overline{Q_T}\to \mathbb{R}^3$ a given \textit{electric field}, solving the  \textit{quasi-static Maxwell's equations}, \textit{i.e.},
\begin{align}\label{eq:maxwell}
	\begin{aligned}
		\textup{div}_x\boldsymbol{E} &= 0&&\quad\text{ in }Q_T\,,\\
		\textrm{curl}_x\boldsymbol{E} &= \mathbf{0}&&\quad\text{ in }Q_T\,,\\
		\boldsymbol{E}\cdot\mathbf{n}&=\boldsymbol{E}_0\cdot\mathbf{n}&&\quad\text{ on }\Gamma_T\,,
	\end{aligned}
\end{align} 
and the power-law index $p\colon \overline{Q_T}\to (1,+\infty)$ depends on the strength of the electric field $\vert \boldsymbol{E}\vert\colon \overline{Q_T}\to \mathbb{R}_{\ge 0}$, \textit{i.e.}, 
there exists a material function $\hat{p}\colon \mathbb{R}_{\ge 0}\to\mathbb{R}_{\ge 0} $ such that
$p(t,x)\coloneqq \hat{p}(\vert \boldsymbol{E}(t,x)\vert)$ for all $(t,x)^\top\in \overline{Q_T}$.\enlargethispage{10mm}

In \hspace*{-0.1mm}the \hspace*{-0.1mm}numerical \hspace*{-0.1mm}experiments, \hspace*{-0.1mm}we \hspace*{-0.1mm}choose \hspace*{-0.1mm}as \hspace*{-0.1mm}spatial \hspace*{-0.1mm}domain \hspace*{-0.1mm}$\Omega\hspace*{-0.15em}\coloneqq\hspace*{-0.15em} (0,1)^3\setminus(B_{1/16}^3(\frac{1}{4}\mathbf{e}_1)\cup B_{1/16}^3(\frac{3}{4}\mathbf{e}_1))$\footnote{Here, $\mathbf{e}_1\coloneqq (1,0,0)^\top\in\mathbb{S}^2 $ denotes the first three-dimensional unit vector.},~\textit{i.e.}, the \hspace*{-0.1mm}unit \hspace*{-0.1mm}cube \hspace*{-0.1mm}with \hspace*{-0.1mm}two \hspace*{-0.1mm}holes \hspace*{-0.1mm}(\text{cf}.\ \hspace*{-0.1mm}Figure \hspace*{-0.1mm}\ref{fig:E}(LEFT)), \hspace*{-0.1mm}as \hspace*{-0.1mm}final \hspace*{-0.1mm}time \hspace*{-0.1mm}$T\hspace*{-0.15em}=\hspace*{-0.15em}0.1$, \hspace*{-0.1mm}as \hspace*{-0.1mm}material~\hspace*{-0.1mm}function~\hspace*{-0.1mm}${\hat{p}\hspace*{-0.15em}\in\hspace*{-0.15em} C^{0,1}(\mathbb{R}_{\ge 0})}$, defined by $\hat{p}(t)\coloneqq 2+\frac{2}{1+10t}$~for~all~${t\ge 0}$, and as electric field $\boldsymbol{E} \in W^{1,\infty}(Q_T;\mathbb{R}^3)$, for every $(t,x)^\top\in \overline{Q_T}$ defined by
\begin{align}\label{def:E}
	\boldsymbol{E}(t,x)\coloneqq 100.0\times \min\bigg\{ \min\bigg\{ \frac{T}{4}, t \bigg\}, T - t \bigg\}\times\left(\frac{x-\frac{1}{4}\mathbf{e}_1}{\vert x-\frac{1}{4}\mathbf{e}_1\vert^3}-\frac{x-\frac{3}{4}\mathbf{e}_1}{\vert x-\frac{3}{4}\mathbf{e}_1\vert^3}\right)\,.
\end{align}
in order to model \textit{shear-thickening} (note that, by definition, it holds that $p^-\hspace*{-0.15em}> \hspace*{-0.15em}2.2$) between~two~\mbox{\textit{electrodes}},  located at the two holes of the domain $\Omega$ (\textit{cf}.\ Figure \ref{fig:E}(LEFT)).  It is readily checked that $\boldsymbol{E}\colon \overline{Q_T}\to \mathbb{R}^3$ indeed solves the quasi-static Maxwell's equations \eqref{eq:maxwell} if we prescribe the normal boundary~\mbox{condition}, \textit{e.g.}, if we set $\boldsymbol{E}_0\coloneqq \boldsymbol{E}$ on $\Gamma_T$. For sake of simplicity, we set $\chi_E\coloneqq 1$. In order to \mbox{generate}~a~\mbox{vortex}~flow, we choose as mechanical body force 
$\widehat{\boldsymbol{f}}\hspace*{-0.15em}\in \hspace*{-0.15em} C^\infty(\overline{\Omega};\mathbb{R}^3)$,~\mbox{defined}~by~${ \widehat{\boldsymbol{f}}(x)\hspace*{-0.175em}\coloneqq \hspace*{-0.175em}(2x_2\hspace*{-0.15em}-\hspace*{-0.15em}1)\mathbf{e}_1}$~for~all~${x\hspace*{-0.175em}=\hspace*{-0.175em}(x_1,x_2,x_3)^\top\hspace*{-0.25em}\in \hspace*{-0.175em}\overline{\Omega}}$, which becomes the total force $\boldsymbol{f}\hspace*{-0.15em}\in\hspace*{-0.15em} W^{1,\infty}(Q_T;\mathbb{R}^3)$ in absence of the electric field at initial time $t=0$ and at final time $t=T$ (\textit{cf}.\ \mbox{Figure} \ref{fig:f}(LEFT)). 
The fully-charged electric field $\boldsymbol{E}(\frac{T}{2})\in C^\infty(\overline{\Omega};\mathbb{R}^3)$~and~total force
  $\boldsymbol{f}(\frac{T}{2})\in C^\infty(\overline{\Omega};\mathbb{R}^3)$ are depicted in Figure~\ref{fig:E}(RIGHT) and Figure~\ref{fig:f}(RIGHT), respectively.~More~precisely, this setup simulates the states  of linear charging from $t=0$ to $t=\frac{T}{4}$, fully-charged from~$t=\frac{T}{4}$~to~$t=\frac{3T}{4}$, and linear discharging from $t=\frac{3T}{4}$ to $t=T$ of the electric field $\boldsymbol{E} \in W^{1,\infty}(Q_T;\mathbb{R}^3)$ defined by \eqref{def:E}.

\begin{figure}[H]\centering
	\includegraphics[width=6.885cm]{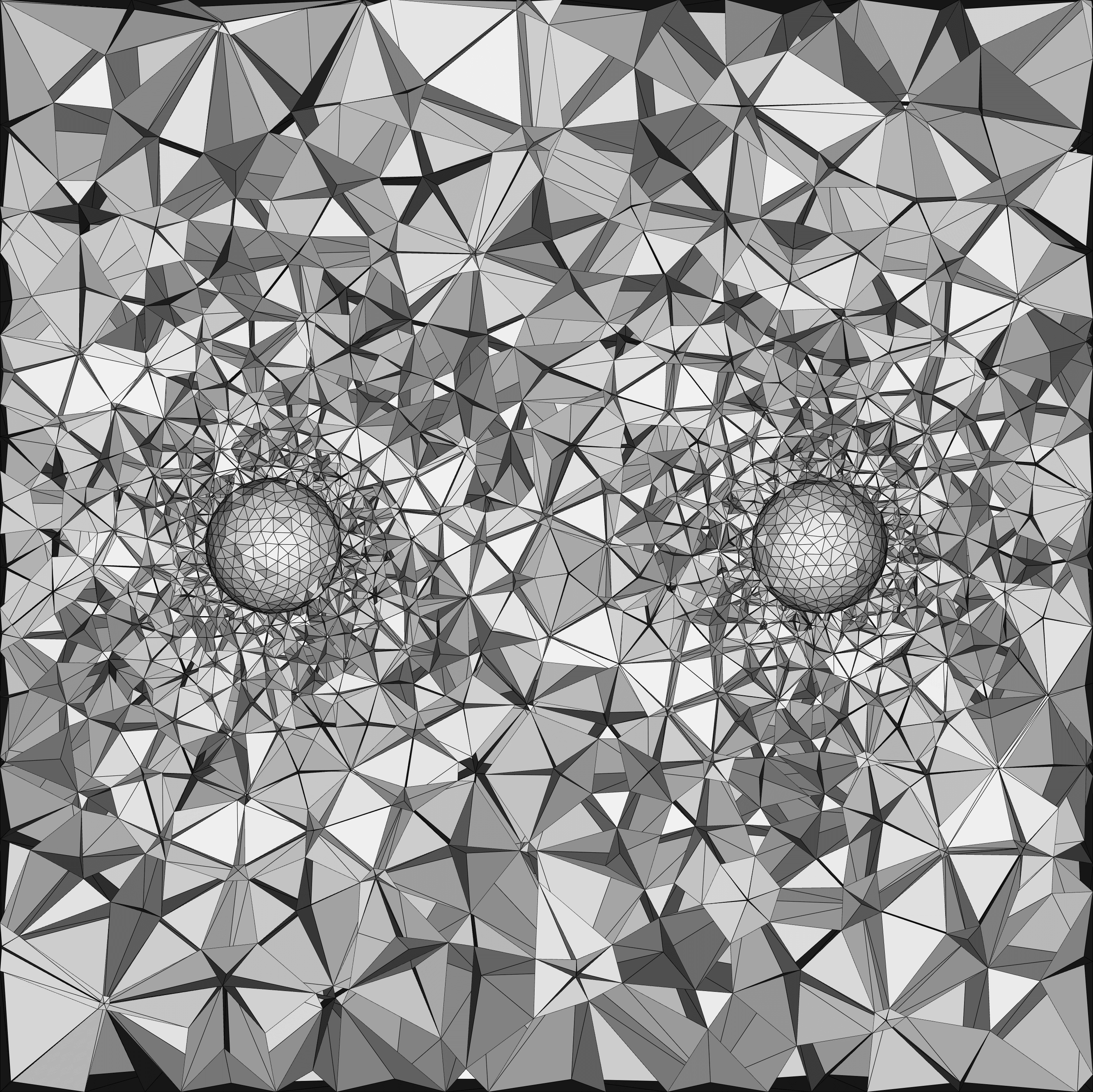} 	\includegraphics[width=8.58cm]{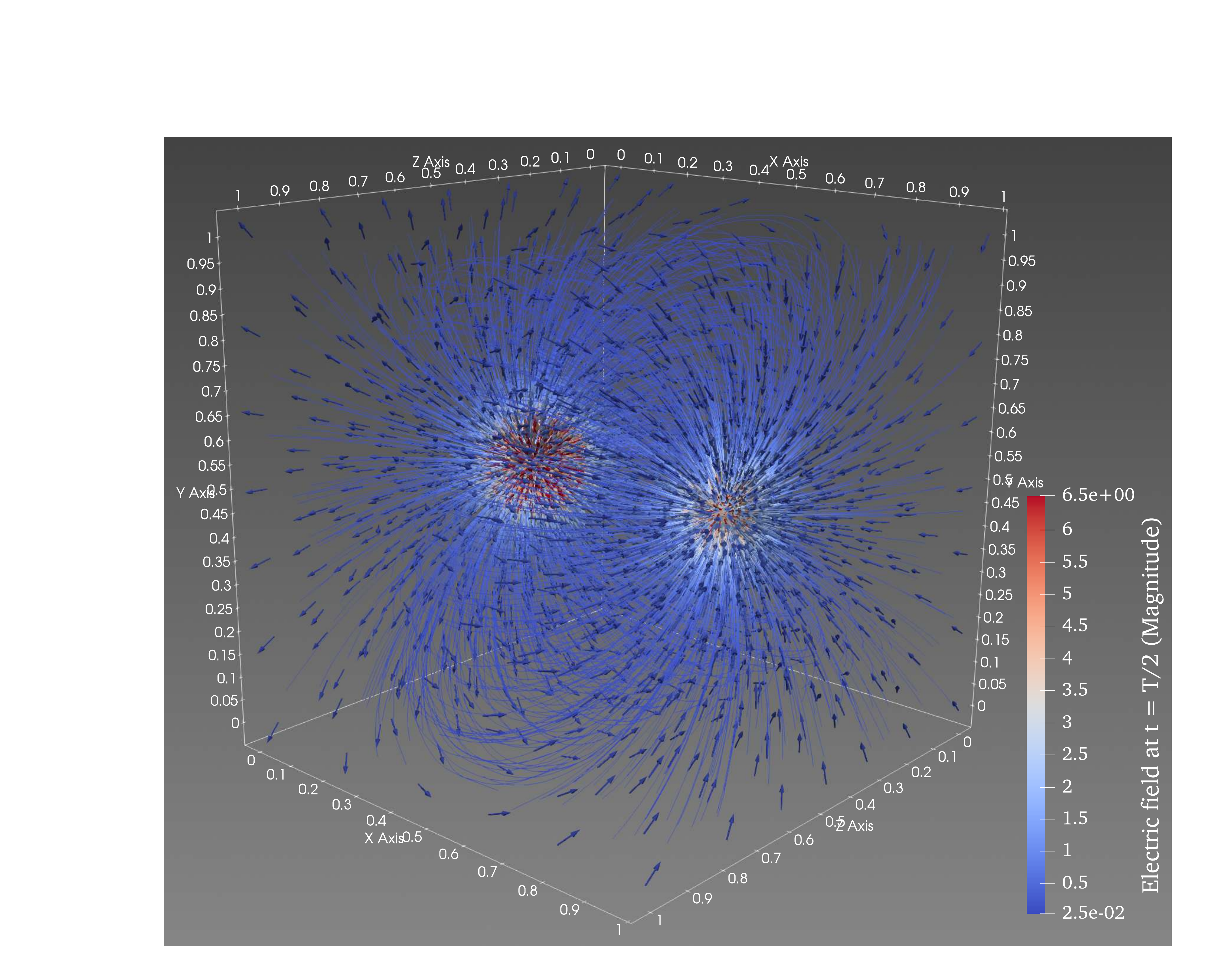} 
	\caption{LEFT: intersection of triangulation $\mathcal{T}_h$ with ($\mathbb{R}\times \{\frac{1}{2}\}\times \mathbb{R}$)-plane; RIGHT:  fully-charged~electric field $\boldsymbol{E}(\frac{T}{2})\in C^\infty(\overline{\Omega};\mathbb{R}^3)$ at time $t=\frac{T}{2}$.}
	\label{fig:E}
\end{figure}\vspace*{-3mm}

As initial condition, we consider a smooth, incompressible vortex flow fulfilling~a~\mbox{no-slip} boundary condition on $\partial\Omega$. More precisely, the initial velocity vector field $\mathbf{v}_0\in H$, for every $x=(x_1,x_2,x_3)^\top \in \Omega$, is defined by
\begin{align}\label{def:v0}
	\mathbf{v}_0(x)\coloneqq (-\sin( 2\pi x_2 ) ( 1 - \cos( 2\pi x_1 ) ),\sin( 2\pi x_1 ) ( 1 - \cos( 2 \pi x_2 ) ),0)^\top\,.
\end{align}
 
\begin{figure}[H]\centering
	\includegraphics[width=7.7cm]{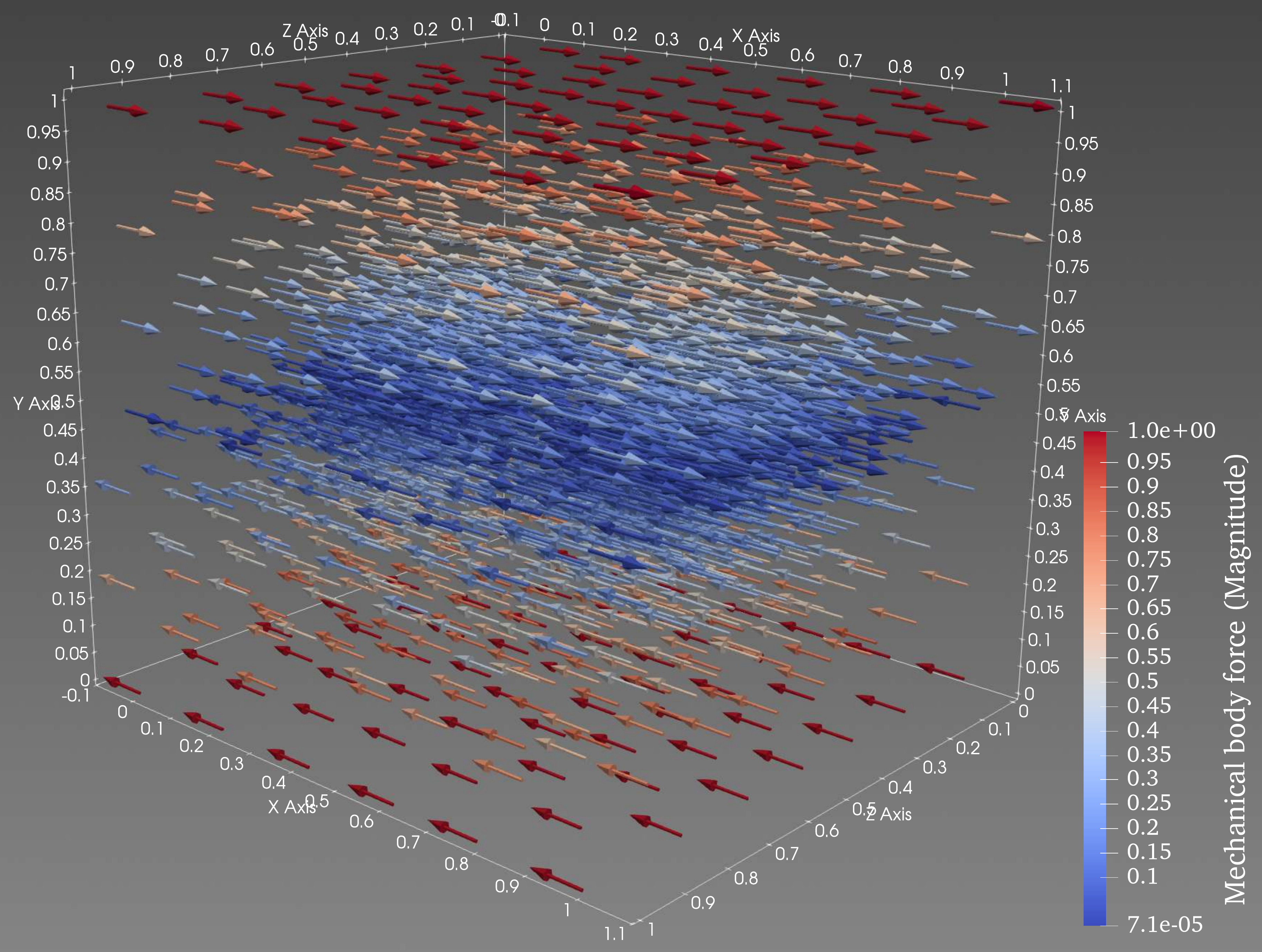}	\includegraphics[width=7.67cm]{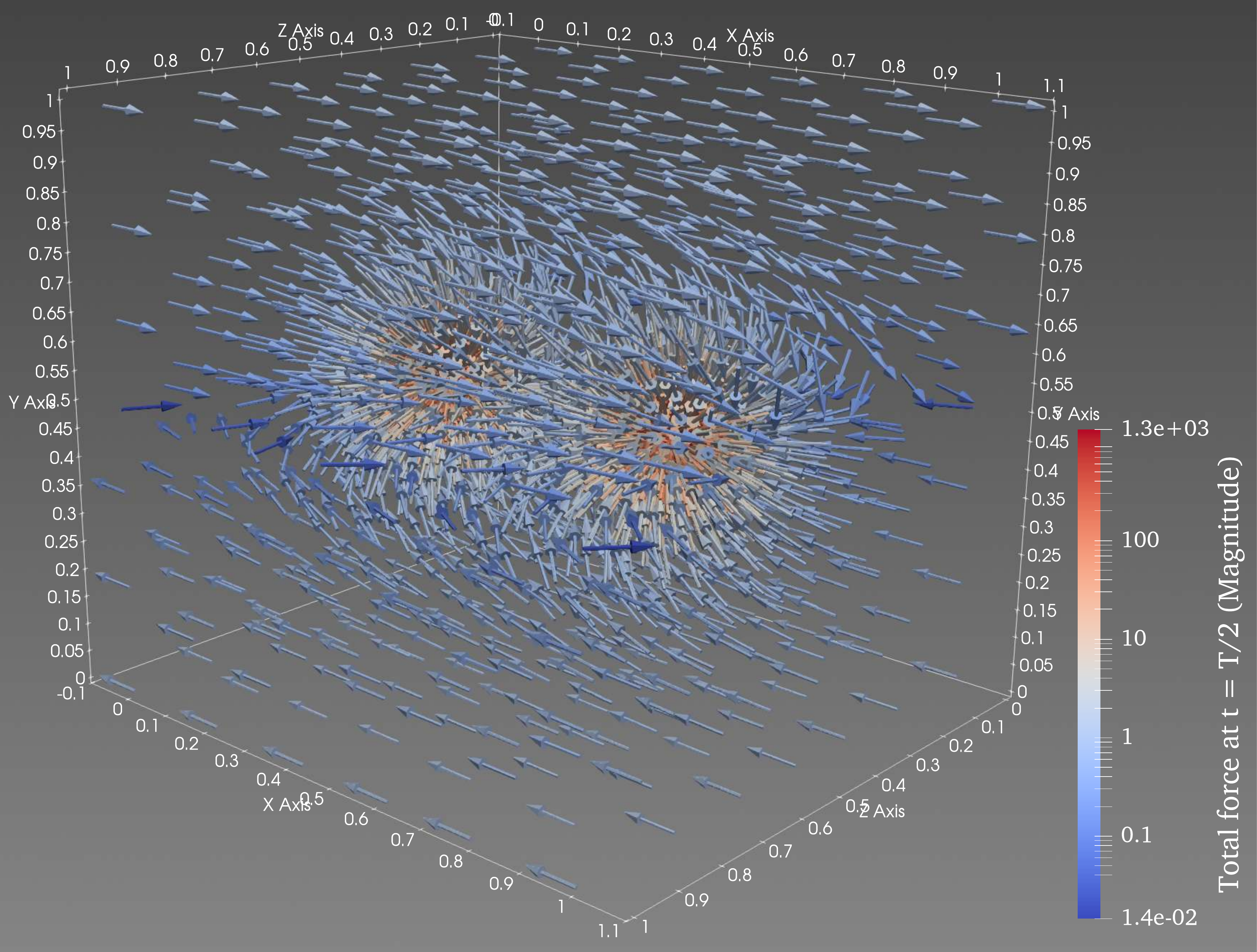} 
	\caption{LEFT: mechanical body force $\widehat{\boldsymbol{f}}\in C^\infty(\overline{\Omega};\mathbb{R}^3)$; RIGHT: total force $\boldsymbol{f}(\frac{T}{2})\in C^\infty(\overline{\Omega};\mathbb{R}^3)$ at time $t=\frac{T}{2}$, \textit{i.e.}, when the electric field is fully-charged (in a $\log$-plot).}
	\label{fig:f}
\end{figure}\vspace*{-1mm}

Since  $\hat{p}\hspace*{-0.1em}\in\hspace*{-0.1em} W^{1,\infty}(\mathbb{R}_{\ge 0})$,
we have that $p\hspace*{-0.1em}\in\hspace*{-0.1em} W^{1,\infty}(Q_T)$. In particular, the power-law index~${p\hspace*{-0.1em}\in\hspace*{-0.1em} W^{1,\infty}(Q_T)}$ is approximated by $p_h\in \mathbb{P}^0(\mathcal{I}_h;\mathbb{P}^0(\mathcal{T}_h))$, $h\in (0,1]$, which is obtained by employing the same one-point quadrature rule from the previous section (\textit{i.e.}, \eqref{def:p_n} with \eqref{eq:explicit_quadpoints}).
We use the Taylor--Hood element  on a triangulation $\mathcal{T}_h$ with 4.947 vertices and 26.903 tetrahedra to compute ${(\mathbf{v}_h^k,\pi_h^k)^\top\hspace*{-0.2em}\in\hspace*{-0.15em} \Vo_{h,0}\hspace*{-0.15em}\times \hspace*{-0.15em}\Qo_h}$,~${k\hspace*{-0.15em}=\hspace*{-0.15em}1,\ldots,K}$, where~$K=128$,~solving~\eqref{eq:primal1}. In doing so, due to $\boldsymbol{f}\in W^{1,\infty}(Q_T;\mathbb{R}^3)$, again, for every $k=1,\ldots,K$, we replace the $k$-th.\ temporal mean $\langle \Pi_h^{0,\mathrm{x}}\boldsymbol{f}\rangle_k\in (\mathbb{P}^0(\mathcal{T}_h))^3$ by the temporal evaluation $\Pi_h^{0,\mathrm{x}}\boldsymbol{f}(t_k)\in (\mathbb{P}^0(\mathcal{T}_h))^3$. To better compare the influence of the charging and discharging of the electric field $\boldsymbol{E} \in W^{1,\infty}(Q_T;\mathbb{R}^3)$, we carry out two different numerical test cases:\enlargethispage{12mm}

$\bullet$ \textit{Numerical test cases:}

\textit{(Test Case 1)}.\ \hypertarget{TC1}{} In \hspace*{-0.1mm}this \hspace*{-0.1mm}case, \hspace*{-0.1mm}the \hspace*{-0.1mm}electric \hspace*{-0.1mm}field \hspace*{-0.1mm}is \hspace*{-0.1mm}never \hspace*{-0.1mm}charged, \hspace*{-0.1mm}\textit{i.e.}, \hspace*{-0.1mm}instead \hspace*{-0.1mm}of \hspace*{-0.1mm}\eqref{def:E},~\hspace*{-0.1mm}we~\hspace*{-0.1mm}set~\hspace*{-0.1mm}${\boldsymbol{E}\hspace*{-0.15em}\coloneqq\hspace*{-0.15em}\mathbf{0}}$~in~$Q_T$, so that the total force just becomes the mechanical body force, \textit{i.e.},  
$\boldsymbol{f}=\widehat{\boldsymbol{f}}$ in $Q_T$ (\textit{cf}.\ Figure \ref{fig:f}(LEFT)).

\textit{(Test Case 2)}.\ \hypertarget{TC2}{} In this case, the electric field is defined by \eqref{def:E} (\textit{i.e.},
it  linearly charges from $t=0$ to $t=\frac{T}{4}$, is fully-charged from $t=\frac{T}{4}$ to $t=\frac{3T}{4}$ (\textit{cf}.\ Figure \ref{fig:E}(RIGHT)), and linearly discharges from $t=\frac{3T}{4}$ to $t=T$), so that the total forces is given via $\boldsymbol{f}=\widehat{\boldsymbol{f}}+\chi_E\textup{div}_x(\boldsymbol{E}\otimes \boldsymbol{E})$~in~$Q_T$~(\textit{cf}.~Figure~\ref{fig:f}(RIGHT)).

$\bullet$ \textit{Observations:}

\textit{(Test Case 1)}.\  In this case, the magnitudes of the velocity vector field (\textit{cf}.\ Figure \ref{fig:v1}) and the kinematic pressure  (\textit{cf}.\ Figure \ref{fig:p1})  slightly reduce as the temporally constant total force $\boldsymbol{f}=\widehat{\boldsymbol{f}}$~(\textit{cf}.~\mbox{Figure}~\ref{fig:f}(LEFT)) is not sufficient to fully maintain the initial vortex flow $\mathbf{v}_0\in H$ (\textit{cf}.\ \eqref{def:v0}).

\textit{(Test Case 2)}.\  In this case, at time $t=\frac{T}{2}$, when the electric field  $\boldsymbol{E}(\frac{T}{2})\in C^\infty(\overline{\Omega};\mathbb{R}^3)$~is~fully-charged (\textit{cf}.\ Figure~\ref{fig:E}(RIGHT)), the total force  $\boldsymbol{f}(\frac{T}{2})\in C^\infty(\overline{\Omega};\mathbb{R}^3)$ generates a high attraction close to the two holes (\textit{cf}.\ Figure~\ref{fig:f}(RIGHT)), where the electrodes are located, so
that 
the magnitudes of both velocity vector field and the kinematic pressure significantly increase close to these two holes.~To~be~more~precise,
the magnitude of velocity vector field  is increased by a factor of about $10$ (\textit{cf}.\ Figure~\ref{fig:v2}(LEFT)) compared to (\hyperlink{TC1}{Test Case 2}) (\textit{cf}.\ Figure~\ref{fig:v1}(LEFT)), while the magnitude of the kinematic pressure is  increased by a factor of about $10^5$ (\textit{cf}.\ Figure~\ref{fig:p2}(LEFT))  compared to (\hyperlink{TC1}{Test Case 2}) (\textit{cf}.\ Figure~\ref{fig:p1}(LEFT)).~At~time~$t=T$, when the electric field is fully-discharged, the total force becomes the mechanical body force enforcing a simple vortex flow, so that
the velocity vector field (\textit{cf}.\ Figure~\ref{fig:v2}(RIGHT))  and kinematic pressure (\textit{cf}.\ Figure~\ref{fig:p2}(RIGHT))  arrive at states similar to (\hyperlink{TC1}{Test Case 1}) at time $t=T$ (\textit{cf}.\ Figure~\ref{fig:v1}(RIGHT) and Figure~\ref{fig:p1}(RIGHT)).
This mimics the rapid change of states possible with electro-rheological fluids.

\begin{figure}[H]\centering
	\includegraphics[width=7.745cm]{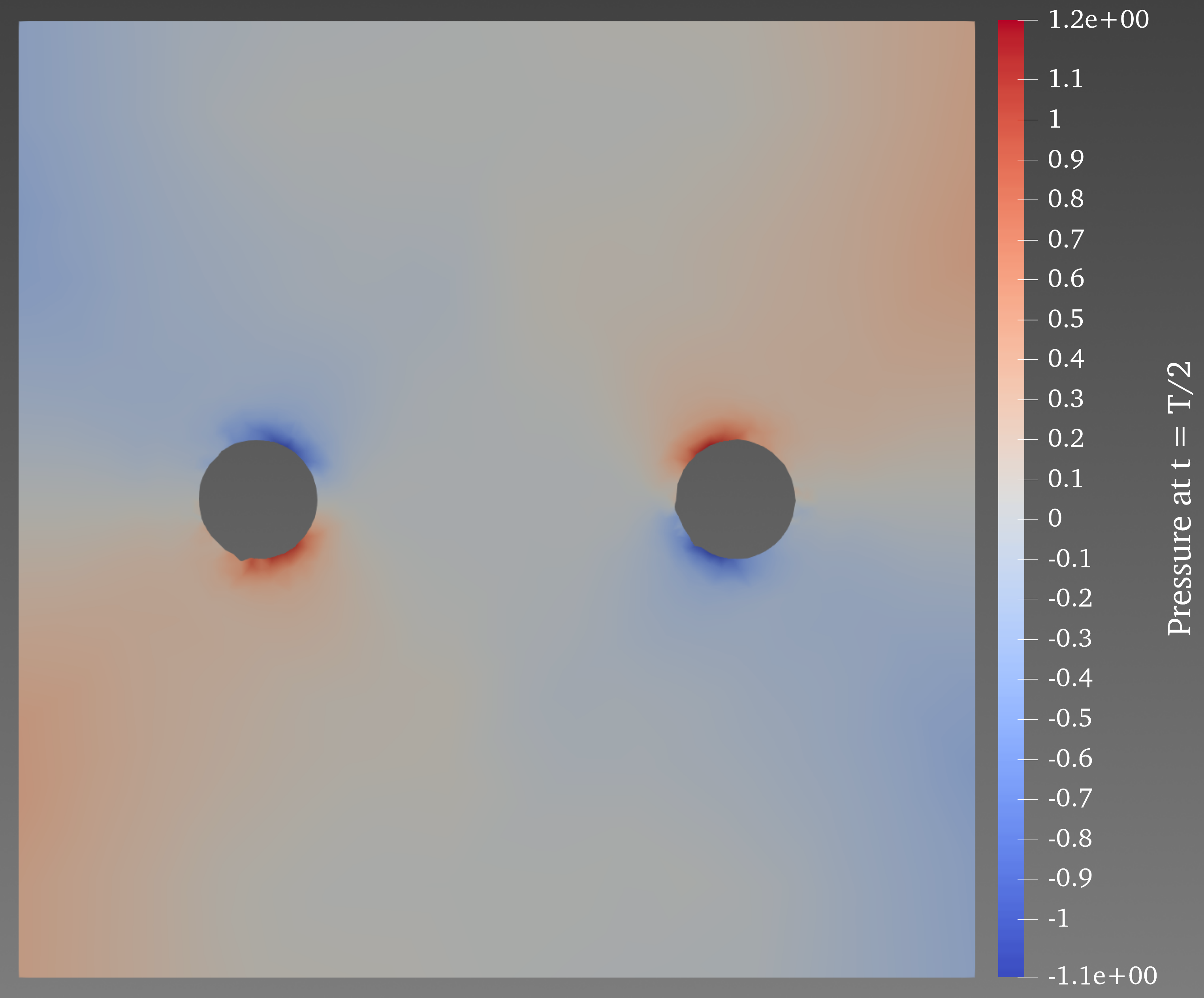}	\includegraphics[width=7.655cm]{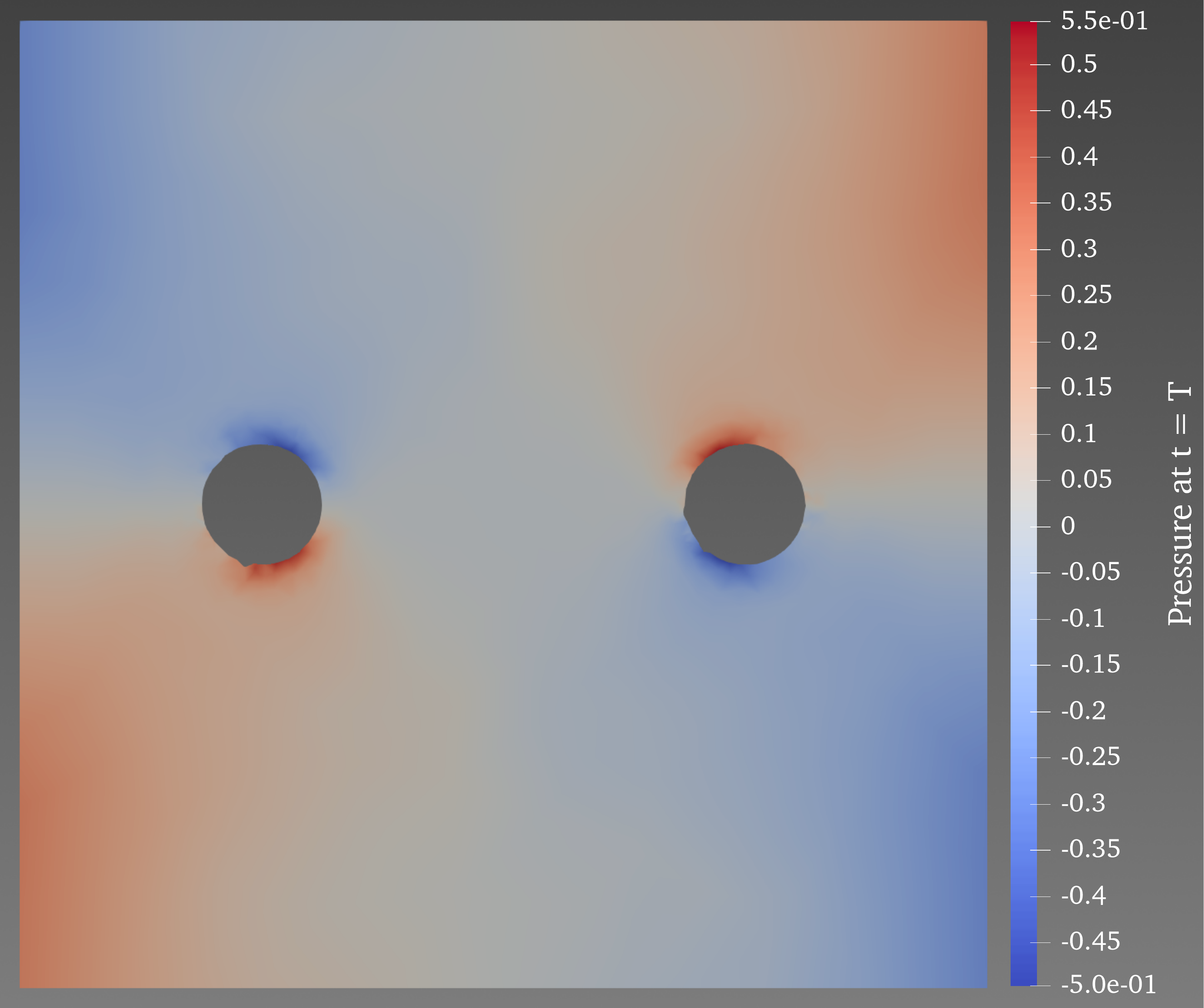}\vspace*{-1.75mm}
	\caption{LEFT: discrete kinematic pressure $\pi_h^{64}\in \mathbb{P}^1_c(\mathcal{T}_h)$ intersected with the ($\mathbb{R}\times \{\frac{1}{2}\}\times \mathbb{R}$)-plane (\textit{i.e.}, at time $t=\frac{T}{2}$) in (\protect\hyperlink{TC1}{Test Case 1});
	RIGHT: discrete kinematic pressure $\pi_h^{128}\in \mathbb{P}^1_c(\mathcal{T}_h)$ intersected with the ($\mathbb{R}\times \{\frac{1}{2}\}\times \mathbb{R}$)-plane (\textit{i.e.}, at time $t=T$) in (\protect\hyperlink{TC1}{Test Case 1}).}
	\label{fig:p1}
\end{figure}\vspace*{-6.5mm}

\begin{figure}[H]\centering
	\includegraphics[width=7.73cm]{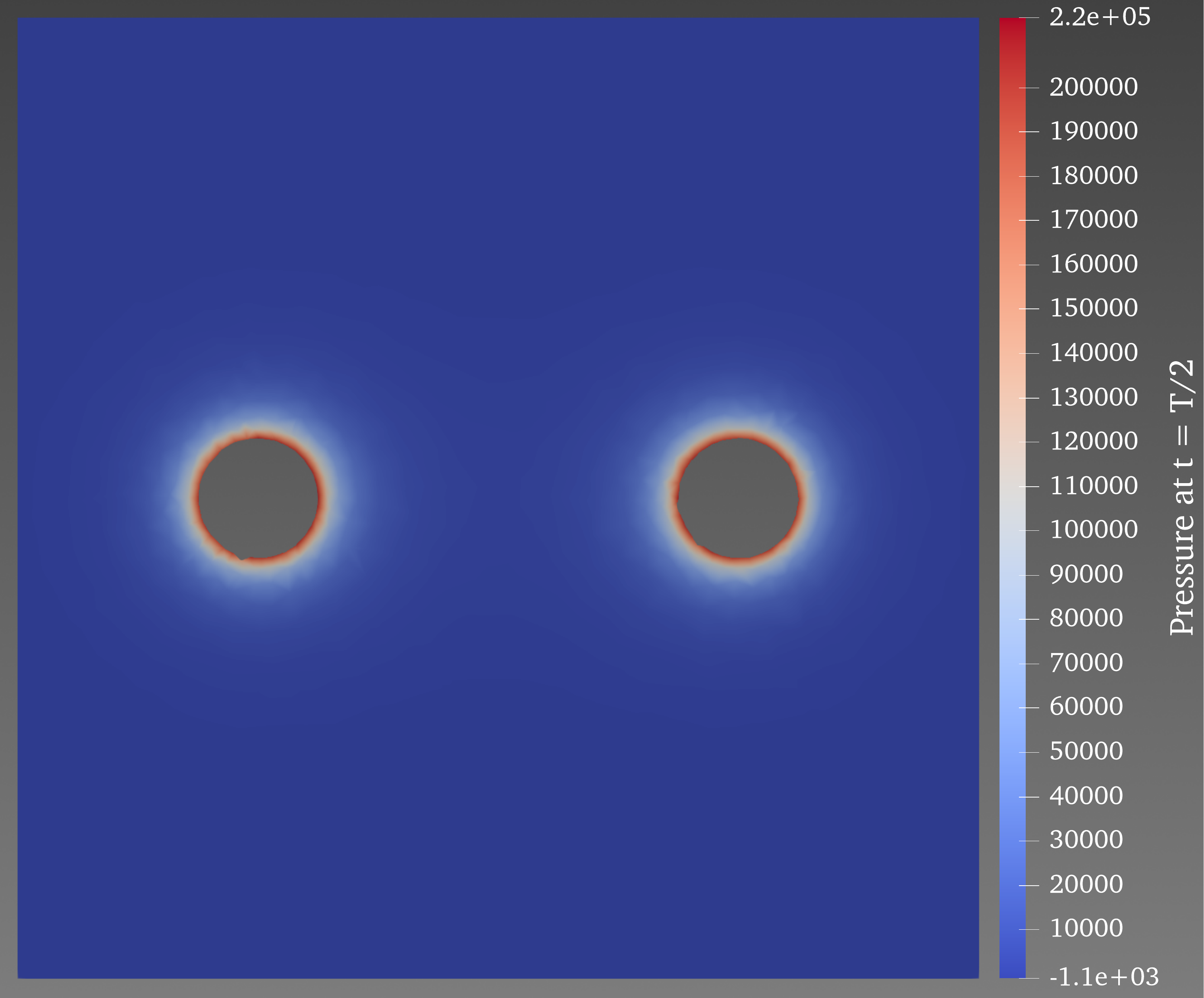}	\includegraphics[width=7.68cm]{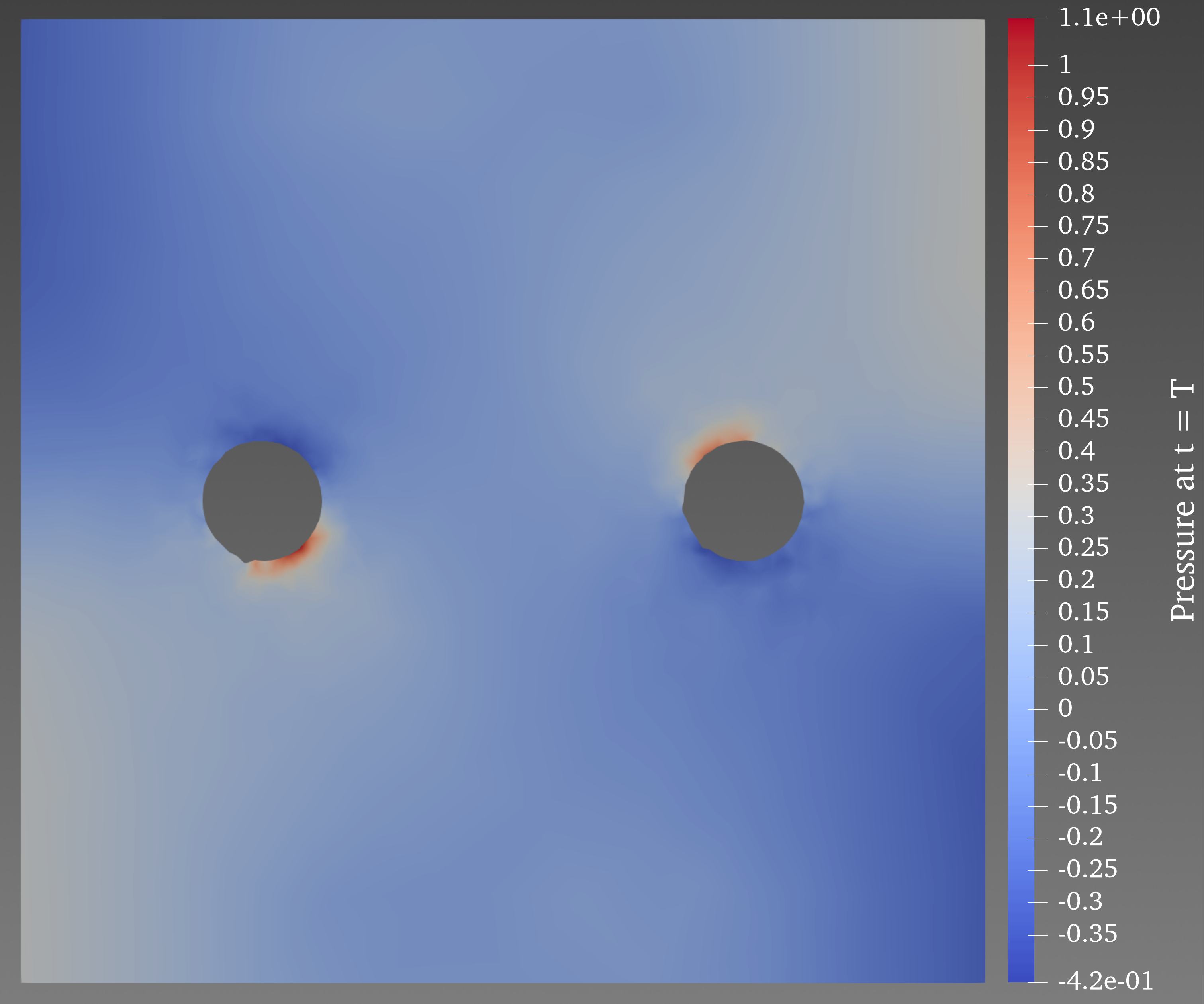}\vspace*{-1.75mm}
\caption{LEFT: discrete kinematic pressure $\pi_h^{64}\in \mathbb{P}^1_c(\mathcal{T}_h)$ intersected with the ($\mathbb{R}\times \{\frac{1}{2}\}\times \mathbb{R}$)-plane (\textit{i.e.}, at time $t=\frac{T}{2}$) in (\protect\hyperlink{TC2}{Test Case 2});
	RIGHT: discrete kinematic pressure $\pi_h^{128}\in \mathbb{P}^1_c(\mathcal{T}_h)$ intersected with the ($\mathbb{R}\times \{\frac{1}{2}\}\times \mathbb{R}$)-plane (\textit{i.e.}, at time $t=T$) in (\protect\hyperlink{TC2}{Test Case 2}).}
	\label{fig:p2}
\end{figure}\vspace*{-6.5mm}\enlargethispage{15mm}

\begin{figure}[H]\centering
	\includegraphics[width=7.7cm]{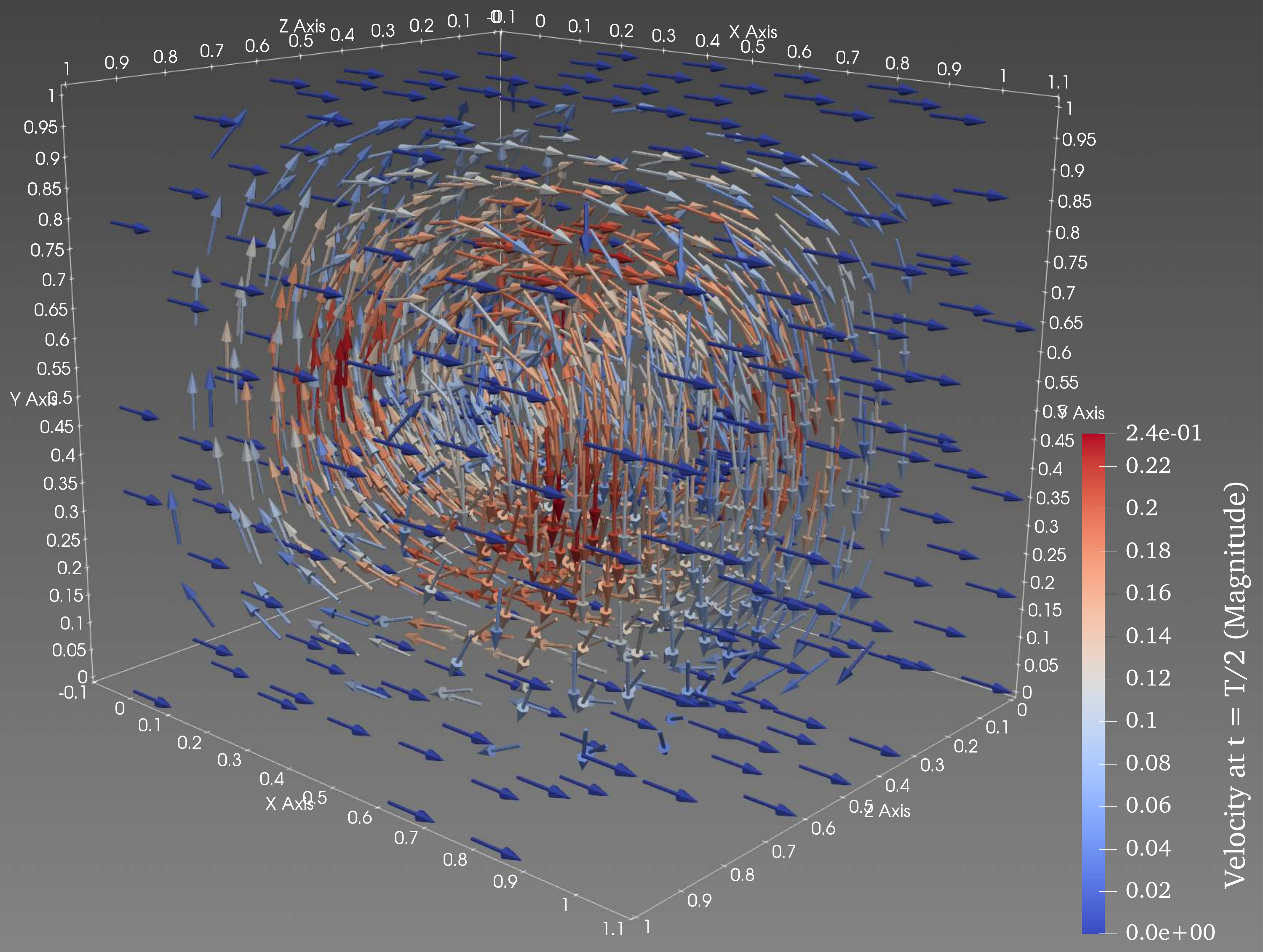}	\includegraphics[width=7.7cm]{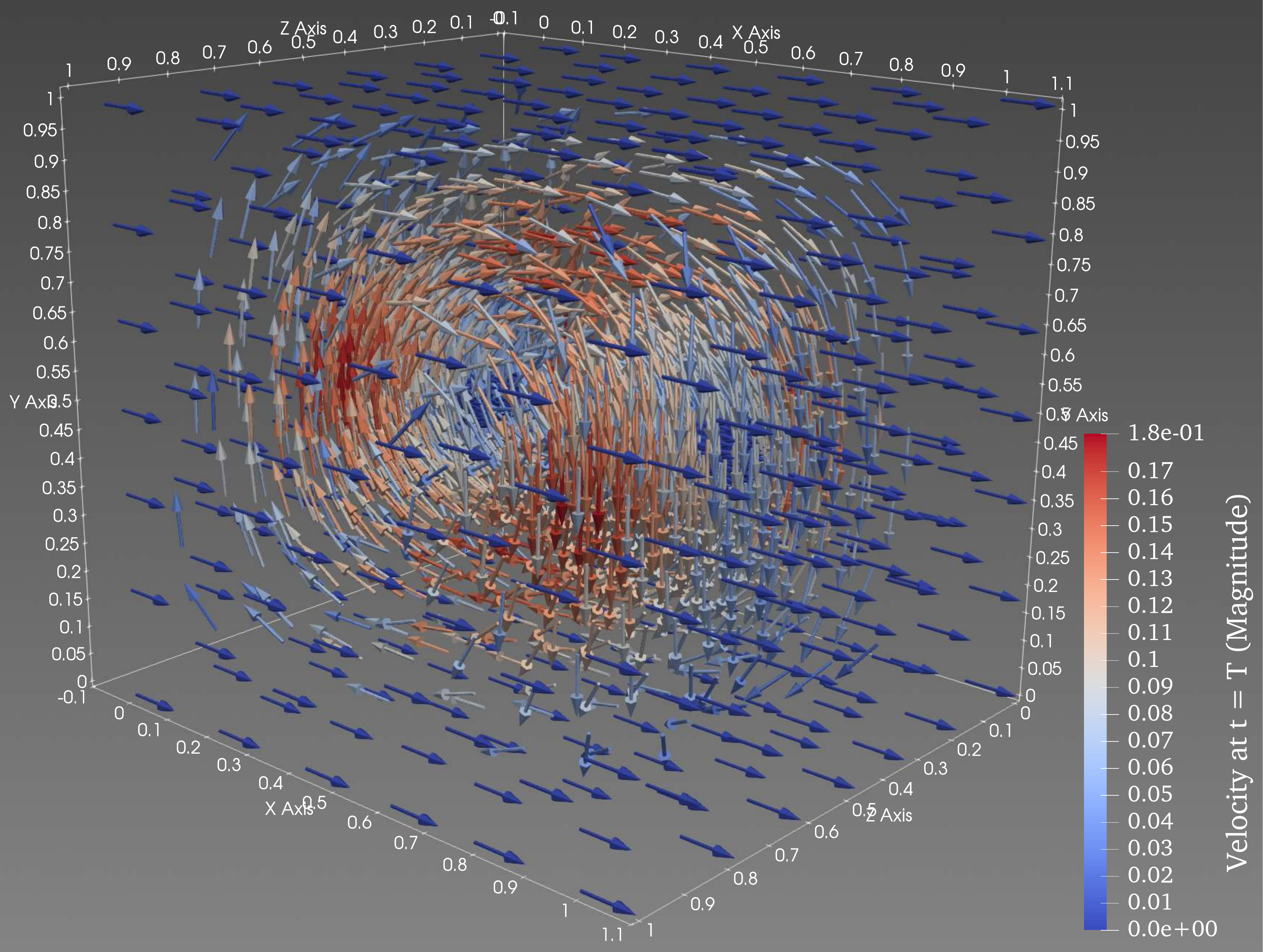}\vspace*{-1.75mm}
	\caption{LEFT: discrete velocity vector field $\mathbf{v}_h^{64}\in (\mathbb{P}^2(\mathcal{T}_h))^3$  (\textit{i.e.}, at time $t=\frac{T}{2}$)  in (\protect\hyperlink{TC1}{Test Case 1});
	RIGHT: discrete velocity vector field $\mathbf{v}_h^{128}\in (\mathbb{P}^2(\mathcal{T}_h))^3$  (\textit{i.e.}, at time $t=T$)  in (\protect\hyperlink{TC1}{Test Case 1}).}
	\label{fig:v1}
\end{figure}
	
\begin{figure}[H]\centering
	\includegraphics[width=7.57cm]{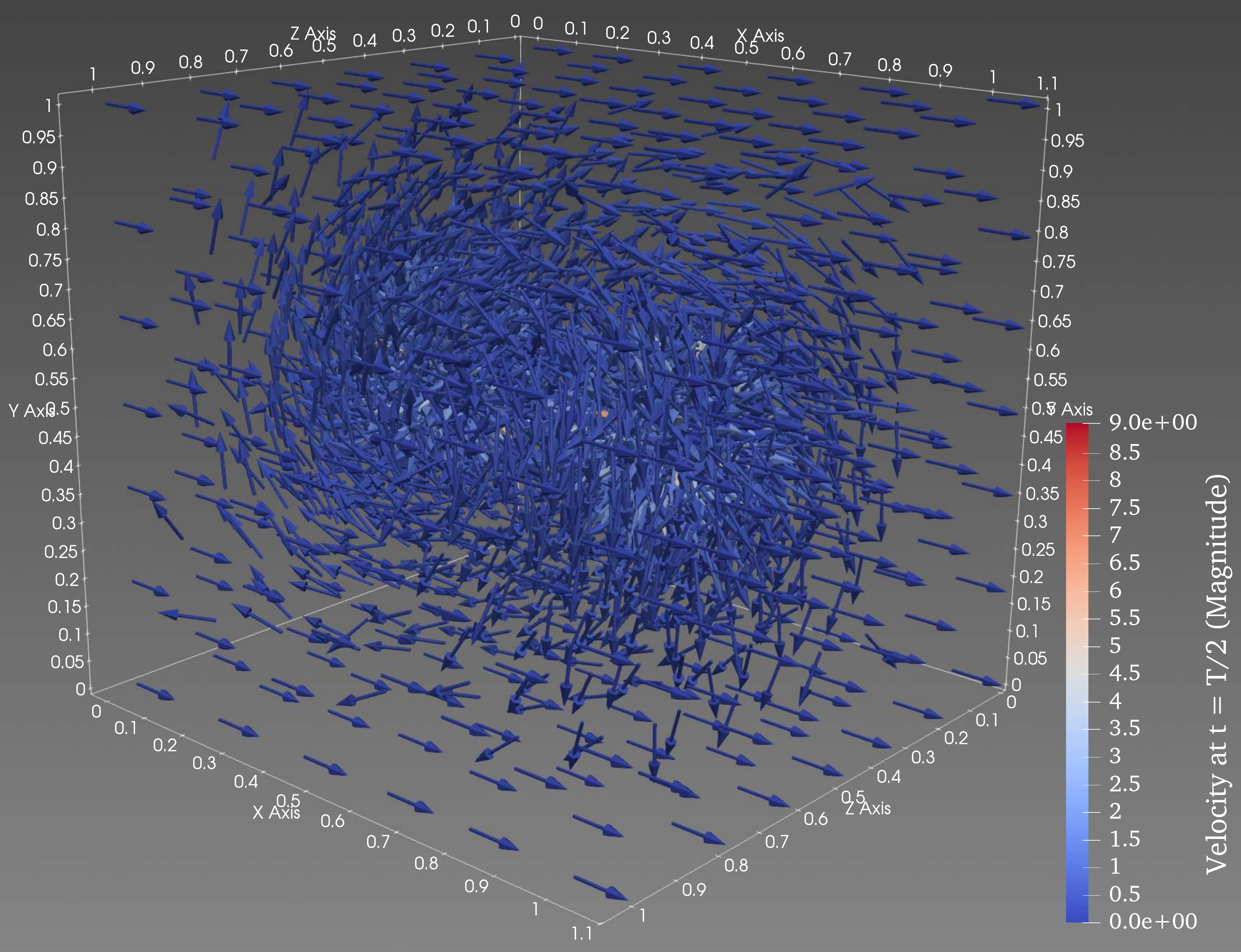}	\includegraphics[width=7.73cm]{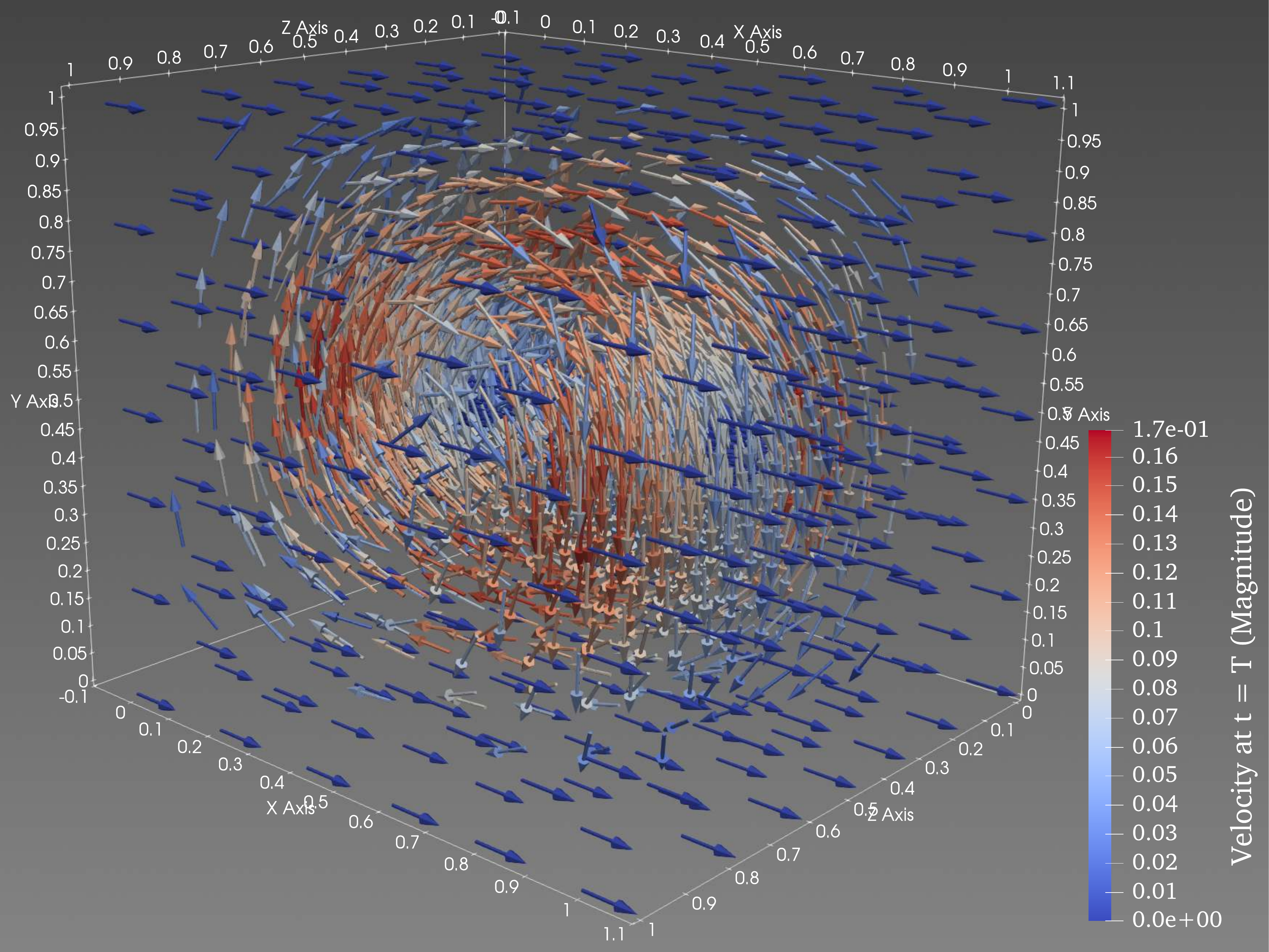}\vspace*{-1.75mm}
	\caption{LEFT: discrete velocity vector field $\mathbf{v}_h^{64}\in (\mathbb{P}^2(\mathcal{T}_h))^3$  (\textit{i.e.}, at time $t=\frac{T}{2}$)  in (\protect\hyperlink{TC2}{Test Case 2});
	RIGHT: discrete velocity vector field $\mathbf{v}_h^{128}\in (\mathbb{P}^2(\mathcal{T}_h))^3$  (\textit{i.e.}, at time $t=T$)  in (\protect\hyperlink{TC2}{Test Case 2}).}
	\label{fig:v2}
\end{figure}

 {\setlength{\bibsep}{0pt plus 0.0ex}\small

%
   
   \providecommand{\MR}[1]{}\def\cprime{$'$} \def\cprime{$'$} \def\cprime{$'$}
   \providecommand{\MR}[1]{}
   \providecommand{\bysame}{\leavevmode\hbox to3em{\hrulefill}\thinspace}
   \providecommand{\noopsort}[1]{}
   \providecommand{\mr}[1]{\href{http://www.ams.org/mathscinet-getitem?mr=#1}{MR~#1}}
   \providecommand{\zbl}[1]{\href{http://www.zentralblatt-math.org/zmath/en/search/?q=an:#1}{Zbl~#1}}
   \providecommand{\jfm}[1]{\href{http://www.emis.de/cgi-bin/JFM-item?#1}{JFM~#1}}
   \providecommand{\arxiv}[1]{\href{http://www.arxiv.org/abs/#1}{arXiv~#1}}
   \providecommand{\doi}[1]{\url{https://doi.org/#1}}
   \providecommand{\MR}{\relax\ifhmode\unskip\space\fi MR }
   \providecommand{\MRhref}[2]{%
   	\href{http://www.ams.org/mathscinet-getitem?mr=#1}{#2}
   }
   \providecommand{\href}[2]{#2}

}
	
\end{document}